\documentclass[11pt]{article}
\usepackage[a4paper,margin=2.5cm]{geometry}
\usepackage[utf8]{inputenc}
\usepackage[T1]{fontenc}
\usepackage{lmodern}
\usepackage{amssymb,amsmath,amsthm}
\usepackage{dsfont}
\usepackage[colorlinks,allcolors=magenta!60!black]{hyperref}
\usepackage{tikz}
\usetikzlibrary{arrows.meta}
\usepackage{hyperref}
\usepackage{empheq}

\newtheorem{thm}{Theorem}[section]
\newtheorem{lem}[thm]{Lemma}
\newtheorem{prop}[thm]{Proposition}
\newtheorem{rmk}[thm]{Remark}

\DeclareMathOperator{\erfc}{erfc}
\let\P\relax\DeclareMathOperator{\P}{\mathbb{P}} 
\DeclareMathOperator{\E}{\mathbb{E}}
\newcommand{\p}[1]{{\mathbb P}\left(#1\right)}
\newcommand{\norm}[1]{\big\lVert #1 \big\rVert}

\def\R{\mathbb{R}}
\def\N{\mathbb{N}}
\def\diffd{\mathrm{d}}

\newcommand{\indic}[1]{\mathds{1}\raisebox{-.4ex}{$\scriptstyle\{#1\}$}}
\def\uppar#1{^{\scriptscriptstyle(#1)}}

\long\def\metanote#1#2{{\color{#1}\
\ifmmode\hbox\fi{\sffamily\mdseries\upshape [#2]}\ }}

\def\vl{v^\ell }

\def\B{{\mathcal B}}

\def\Zint#1{\mathchoice
{\ZZint\displaystyle\textstyle{#1}}%
{\ZZint\textstyle\scriptstyle{#1}}%
{\ZZint\scriptstyle\scriptscriptstyle{#1}}%
{\ZZint\scriptscriptstyle\scriptscriptstyle{#1}}%
\!\int}
\def\ZZint#1#2#3{{\setbox0=\hbox{$#1{#2#3}{\int}$}
\vcenter{\hbox{$#2#3$}}\kern-.5\wd0}}
\def\dashint{\Zint-}

\begin{document}

\allowdisplaybreaks

\author{Julien Berestycki\footnote{Department of Statistics, University of Oxford, UK}, Éric Brunet\footnote{Laboratoire de Physique de l’\'Ecole normale sup\'erieure, ENS, Universit\'e PSL, CNRS, Sorbonne Universit\'e, Universit\'e de Paris, F-75005 Paris, France}, James Nolen\footnote{Department of Mathematics, Duke University Box 90320, Durham, NC 27708, USA}, Sarah Penington\footnote{Department of Mathematical Sciences, University of Bath, UK}}

\title{A free boundary problem arising from \\branching Brownian motion with selection}

\date{\today}
\maketitle

\begin{abstract}
We study a free boundary problem for a parabolic partial differential equation in which the solution is coupled to the moving boundary through an integral constraint.  The problem arises as the hydrodynamic limit of an interacting particle system involving branching Brownian motion with selection, the so-called {\it Brownian bees} model which is studied in the companion paper \cite{BBNP2}. In this paper we prove existence and uniqueness of the solution to the free boundary problem, and we characterise the behaviour of the solution in the large time limit.
\end{abstract}

\section{Introduction and main results}
Given a probability measure $\mu_0$ on $\R^d$, $d \geq 1$, we consider the following free boundary problem: find $u(x,t)
: \R^d\times (0,\infty) \to [0,\infty)$ and $R_t:(0,\infty) \to [0,\infty]$ such that
\begin{equation}
\begin{cases}
\partial_t u =  \Delta u  +u ,  & \text{for $t>0$ and  $\|x\| < R_t$},   \\
u(x,t)=0, &\text{for $t>0$ and $\|x\|\ge R_t$}, \\
u(x, t) \quad \text{is continuous on $\R^d \times (0,\infty)$}, \\[1ex]
 \displaystyle \int_{\R^d} u(x,t)\, \diffd x =1, & \text{for } t > 0,
    \\[1ex]
u(\cdot,t)\to \mu_0 &\text{weakly as $t \searrow 0$}.
\end{cases}
\label{pbu}
\end{equation}
Although the PDE and Dirichlet boundary condition are linear in $u$, this system is non-linear in $u$ due to the coupling of the unknown boundary $R_t$ with the function $u$ through the integral constraint
\begin{align}
\int_{\R^d} u(x,t) \,\diffd x = \int_{\B(R_t)} u(x,t) \,\diffd x = 1, \label{intconstraint}
\end{align}
where $\B(r)=\{x\in\R^d\;:\; \|x\|<r\}$ is the open ball of radius $r$.
For a given function $R_t$ which is sufficiently smooth, existence of a unique solution to the associated Dirichlet initial/boundary value problem is classical, but analysis of the free boundary problem~\eqref{pbu}, where $R_t$ may be regarded as a control needed to guarantee~\eqref{intconstraint}, is more difficult.  

We say that a pair $(u,R)$ is a classical solution to \eqref{pbu} if $R$ is measurable, $u \in C(\R^d \times (0,\infty)) \cap C^{2,1}(\Omega^+)$ where $\Omega^+ = \{ (x,t) \in \R^d \times (0,\infty) \;:\; \|x\| < R_t \}$, and \eqref{pbu} holds. Our first main result is well-posedness of this problem:

\goodbreak

\begin{thm}\label{thm:exists u}
Let $\mu_0$ be a Borel probability measure on $\R^d$. Then
there exists a unique classical solution to the free boundary problem~\eqref{pbu}. Furthermore,
\begin{itemize}
\item $t\mapsto R_t$ is continuous (and finite) for $t > 0$. 
\item As $t \searrow 0$, $R_t \to R_0:=\inf\big\{r>0 : \mu_0\big( \B(r)\big) = 1 \big\}\in [0,\infty]$. 
\item For any $\alpha < 1/2$, there exists $C_\alpha<\infty$ such that $R_t - R_s \leq C_\alpha (t - s)^\alpha$ for all $s \geq 0$ and $t \in (s,s+1]$.
\item For $t>0$ and $\|x\|<R_t$,
$u(x,t)>0$.
\end{itemize}
\end{thm}

The initial measure $\mu_0$ may be singular with respect to Lebesgue measure, and it may put positive mass on the boundary of its support (at $\|x\| = R_0$). In this case, $R_t - R_0$ must increase very quickly at small times in order that the mass constraint \eqref{intconstraint} be satisfied, as we explain later in Remark \ref{rmk:holder}.

The second main result of this paper concerns the behaviour of $u(x,t)$ and $R_t$ as $t \to \infty$.  Observe that a stationary solution of \eqref{pbu} is given by the principal Dirichlet eigenfunction of the Laplacian in a spherical domain with radius uniquely chosen so that the eigenvalue is precisely $1$. More precisely, there is a unique value $R_\infty > 0$ such that the eigenvalue problem
\begin{equation}
\begin{cases}
- \Delta U(x)  = \lambda U(x) ,  & \|x\| <  R_\infty,   \\
U(x) > 0, \quad &  \| x\| < R_\infty, \\
U(x)=0, \quad &  \|x\|= R_\infty, \\
\end{cases} \label{eqU}
\end{equation}
has a solution $(U,\lambda)$ with $\lambda = 1$. The principal eigenfunction $U$ is unique up to a multiplicative factor, so we normalize $U$ by
\[
\int_{\B(R_\infty )} U(x) \,\diffd x = 1.
\]
We may regard $U$ as a continuous function on all of $\R^d$ by extending $U(x) = 0$ for $\|x\| > R_\infty$.  For this choice of $R_\infty$, the principal eigenfunction $U(x)$ is a time-independent solution to \eqref{pbu}, with $R_t \equiv R_\infty$.  In particular, if $d = 1$, then $R_\infty = \frac{\pi}{2}$ and the eigenfunction is $U(x) = \frac{1}{2} \cos(x)$.  The following shows that this time-independent solution is the unique attractor of all solutions to~\eqref{pbu}:

\begin{thm}\label{thm:conv u}
For any initial Borel probability measure $\mu_0$, the solution $(u,R)$ to the free boundary problem \eqref{pbu} satisfies 
\[
\lim_{t \to \infty} R_t = R_\infty \quad \quad \text{and} \quad \quad \lim_{t\to\infty} \Vert u(\cdot,t)-U(\cdot)\Vert_{L^\infty}=0, 
\]
where $R_\infty > 0$ is the unique value for which the eigenvalue problem~\eqref{eqU} has a solution with eigenvalue $\lambda = 1$.
\end{thm}

Although the limit $U(x)$ is spherically symmetric, the solution $u(x,t)$ of \eqref{pbu} need not be, since the initial measure $\mu_0$ is not assumed to be spherically symmetric.

\subsection*{Motivation and related works}

Our motivation for this work comes from the study of an interacting particle system known as $N$-particle branching Brownian motion ($N$-BBM)  with spatial selection. A general form of $N$-BBM with spatial selection might be described as follows: There are $N$ particles  moving in $\R^d$ with locations at time $t$ given by $\big\{X\uppar N_k(t)\big\}{}_{k=1}^N$.   Each particle moves independently as a Brownian motion with diffusivity $\sqrt 2$ and branches independently into two particles at rate~1. Whenever a particle branches, however, the particle having least ``fitness'' or ``score'' (out of the entire ensemble) is instantly removed (killed), so that there are exactly $N$ particles in the system at all times.  The fitness of a particle is a function $\mathcal{F}(x)$ of its position $x \in \R^d$, and as a result, the elimination of least-fit particles tends to push the ensemble toward regions of higher fitness.  Variants of this stochastic process were first studied in one spatial dimension, beginning with work of Brunet, Derrida, Mueller, and Munier \cite{BDMM06, BDMM07} on discrete-time processes, and work of Maillard \cite{Maillard16} on the continuous-time model involving Brownian motions.  These works involve a monotone fitness function (e.g. $\mathcal{F}(x) = x$, for $x \in \R$) so that selection always occurs on one side of the ensemble.   The general $d$-dimensional model which we have described above was first studied by N. Berestycki and Zhao \cite{BZ18}; specifically, they studied the particle system with fitness functions $\mathcal{F}(x) = \|x\|$ and $\mathcal{F}(x) = \lambda \cdot x$ for some fixed $\lambda\in \R^d$, both of which have the effect of pushing the ensemble of particles away from the origin.  

It is natural to ask how such a particle system behaves in the limit $N \to \infty$:
Suppose the initial particle locations are independent and identically distributed, with distribution given by $\mu_0$.
  Does the (random) empirical measure of particles
\[
\mu\uppar N(\diffd x,t) = \frac{1}{N} \sum_{k=1}^N \delta_{X\uppar N_k(t)}(\diffd x)
\]
converge in some sense to a solution of a partial differential equation? 
Such a limiting partial differential equation is known as a hydrodynamic limit.

 In the setting of one spatial dimension and with monotone fitness function $\mathcal{F}(x) = x \in \R$, DeMasi, Ferrari, Presutti, and Soprano-Loto \cite{DMFPSL} proved that under certain assumptions about the initial configuration of particles, the family of measures $\mu\uppar N(\diffd x,t)$ does converge, as $N \to \infty$, to a limit which can be identified with a solution $u(x,t)$ to a free boundary problem:
\begin{equation}
\begin{cases}
\partial_t u = \partial_{x}^2 u + u, \quad & x > \gamma_t, \;\; t > 0,  \\
 u(x,t) =0, \quad & x \leq \gamma_t,\;\; t > 0, \\[1ex]
\displaystyle \int_{\gamma_t}^\infty u(x,t) \,\diffd x = 1, \quad & t > 0,
\end{cases} \label{d1fbp}
\end{equation}
where the free boundary at $x = \gamma_t \in \R$ is related to $u$ through the integral constraint.  Global existence of solutions to this free boundary problem was proved by J. Berestycki, Brunet, and Penington \cite{BBP}.  Building on the approach of \cite{DMFPSL}, Beckman \cite{Beck19} derived a similar hydrodynamic limit in the one-dimensional setting with symmetric fitness $\mathcal{F}(x) = -|x|$.  Durrett and Remenik \cite{DR11} derived and analysed a non-local free boundary problem corresponding to a related model in which non-diffusing particles in $\R$ are born at random displacements from their parent particles but do not move during their lifetimes.

For more general fitness functions $\mathcal{F}$ and in higher dimensions, one expects the hydrodynamic limit of $\mu\uppar N$ to be a solution $u$ of the following free boundary problem: find $\big(u(x,t),\ell(t)\big)$ such that
\begin{equation}
\begin{cases}
\partial_t u =  \Delta u  +u ,  & \text{for $t>0$ and  $x \in \Omega_{\ell(t)}$},   \\
u(x,t)=0, &\text{for $t>0$ and $x \notin \Omega_{\ell(t)}$}, \\
u(x, t) \quad \text{is continuous on $\R^d \times (0,\infty)$}, \\[1ex]
 \displaystyle \int_{\R^d} u(x,t)\, \diffd x =1, & \text{for } t> 0,
    \\[1ex]
u(\cdot,t)\to \mu_0 &\text{weakly as $t \searrow 0$,}
\end{cases}
\label{pbugeneral}
\end{equation}
where
\[
\Omega_\ell = \{ x \in \R^d \;:\;\; \mathcal{F}(x) > \ell\}.
\]
A solution to this problem is a pair $\big(u(x,t),\ell(t)\big)$; the function $u(x,t)$ is non-zero only inside the super-level set $\Omega_{\ell(t)}$. At each time $t > 0$, the free boundary is constrained to be a level-set of $\mathcal{F}$, i.e.~$\partial \Omega_{\ell(t)} = \mathcal{F}^{-1}(\ell(t))$, assuming $\mathcal{F}$ is continuous. We may interpret $u$ as the density of a population that evolves on a fitness landscape described by $\mathcal{F}$. The PDE $\partial_t u = \Delta u + u$ arises naturally from the diffusion and growth of the population (branching at rate $1$), and the boundary condition arises from the selection mechanism whereby particles are removed at the boundary of $\Omega_{\ell(t)}$, where $\mathcal{F}(x) = \ell(t)$. We interpret $\ell(t)$ as the current fitness level of the least-fit individuals in the population.  Because a particle is removed at the fitness boundary each time an interior particle branches, the total mass is conserved. The problem~\eqref{d1fbp} is a special case of this problem~\eqref{pbugeneral}, in one spatial dimension and with $\mathcal{F}(x)$ given by any continuous monotonically increasing function (e.g.~$\mathcal{F}(x) = x$, $\Omega_{\ell(t)} = (\gamma_t,\infty)$). 

The free boundary problem \eqref{pbu}, which is the focus of this paper, is also a particular case of~\eqref{pbugeneral}, but in multiple spatial dimensions and with a fitness function that has a confining effect.  Specifically,~\eqref{pbu} corresponds to a spherically symmetric fitness function $\mathcal{F}(x) = f(\|x\|)$, where $f:[0,\infty) \to \R$ is any continuous strictly decreasing function, and $R_t = f^{-1}\big(\ell(t)\big)$ (e.g. $\mathcal{F}(x) = -\|x\|$).  In the $N$-BBM process with this fitness function, the least-fit particle is the one that is furthest from the origin; hence, the selection mechanism has a confining effect on the ensemble of particles.  
This system is sometimes known as the Brownian bees model.
In a companion paper \cite{BBNP2}, we prove that \eqref{pbu} is indeed the hydrodynamic limit for this particle system, as $N \to \infty$. Defining
\[
R\uppar N_t = \max  \left\{ \|X\uppar N_k(t)\| \;:\; \; k \in \{ 1,\dots,N\} \right\},
\]
which is the radius of the ensemble of particles, we also show that for any $t > 0$, $R\uppar N_t \to R_t$ almost surely, as $N \to \infty$, where $R_t$ is determined by solving \eqref{pbu}. Thus, for large but finite $N$, the selection (or removal) of particles in this $N$-BBM is happening at a location near the free boundary for the solution of \eqref{pbu}, at $\|x\| \approx R_t$.

For finite $N$, this $N$-BBM process converges in distribution as time $t\to \infty$ to a unique stationary distribution $\pi\uppar N$; this is a probability distribution on $(\R^d)^N$.  As we also prove in~\cite{BBNP2}, the limiting behaviour of $\pi\uppar N$ as $N\to \infty$ is characterised by the large time limit $(U, R_\infty)$ of the free boundary problem (c.f.~Theorem \ref{thm:conv u}, above).  This is called a ``strong selection principle''~\cite{DMFPSL} for the particle system. In particular, this means that the marginal distribution of a uniformly chosen particle under $\pi\uppar N$ converges to the measure that has density $U(x)$ on the domain $\{ \|x\| < R_\infty \}$ and density $0$ outside that domain.  All together, the results in the present article and in \cite{BBNP2} give meaning to the following informal diagram: 
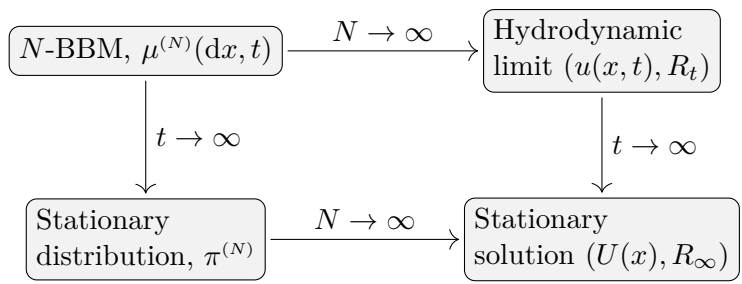
\begin{quote}
\centering
\begin{tikzpicture}[box/.style={draw,rounded corners,align=left,outer
sep=2pt,fill=black!5},arrows={[scale=1.5]}]
\node[box] (a) at (0,0) {$N$-BBM, $\mu\uppar N(\diffd x,t)$};
\node[box] (b) at (6,0) {Hydrodynamic\\limit $(u(x,t),R_t)$};
\node[box] (c) at (0,-2.5) {Stationary\\distribution, $\pi\uppar N$};
\node[box] (d) at (6,-2.5) {Stationary\\solution $(U(x),R_\infty)$};
\draw[->] (a.east) -- node[above] {$N\to\infty$} (b.west);
\draw[->] (a) -- node[right] {$t\to\infty$} (c);
\draw[->] (c.east) -- node[above] {$N\to\infty$} (d.west);
\draw[->] (b) -- node[right] {$t\to\infty$} (d);
\end{tikzpicture}
\end{quote}
The present article deals with right hand side of the diagram: well-posedness and properties of the free boundary problem defining $u$ and $R$, and their long-time behaviour; the companion paper \cite{BBNP2} gives rigorous meaning to the rest of the diagram, pertaining to the particle system and the limit $N \to \infty$.  In the future, we hope to extend these results to a more general class of fitness functions $\mathcal{F}$ which are not spherically symmetric.  Our restriction to spherically symmetric $\mathcal{F}$ enables a connection to a parabolic obstacle problem, as explained below.

It is instructive to compare the free boundary problem \eqref{pbu} to some other free boundary problems. If we assume that $R_t$ is differentiable and that $u$ is $C^{2,1}$ up to the boundary, then by differentiating (formally) the integral constraint and boundary condition in \eqref{pbu},
\[
0 = \frac{\diffd }{\diffd t} \int_{\B(R_t)} u(x,t)\, \diffd x  \quad \quad \text{and} \quad \quad 0 = \frac{\diffd }{\diffd t} \int_{\partial \B(R_t)} u(x,t)\, \diffd S(x),
\]
we arrive at the relations
\begin{equation}
\int_{\partial \B(R_t)} (\nu \cdot \nabla u)\,\diffd S(x) = - 1  \quad \quad \text{and} \quad \quad R_t' =  \int_{\partial\B(R_t)} \Delta u(x,t)\, \diffd S(x), \label{Rprime}
\end{equation}
where $\nu=\nu(x)$ is the outward unit vector at $x$. Similarly, for the one-dimensional version \eqref{d1fbp}, $u$ satisfies $\partial_xu(\gamma_t,t) = 1$, and the free boundary velocity is $\gamma_t' = - \partial_x^2u(\gamma_t,t)$. For the Stefan problem, a well-studied free boundary problem for the heat equation which also arises from limits of certain interacting diffusions (e.g. \cite{CDSS17, DNS19}), the free boundary velocity is proportional to $\partial_x u$ at the boundary, rather than $\partial_x^2u$.  If the initial data $\mu_0$ for \eqref{pbu} is spherically symmetric, then $u$ has spherical symmetry for all $t > 0$; in this case, the relations in \eqref{Rprime} reduce to $\|\nabla u(R_t,t)\| = |\partial \B(R_t)|^{-1} = 1/(c_d R_t^{d-1})$ and to $R_t' = c_d R_t^{d-1}\Delta u(R_t,t)$, where $c_d$ is a constant depending on the dimension $d$. In this case, the problem \eqref{pbu} is more like the flame propagation model studied by Caffarelli and V\'azquez \cite{CV95}, where the free boundary moves with normal velocity proportional to $\Delta u$ at the boundary, and $\|\nabla u\| = c$ is held constant along the free boundary.  For general initial condition, however, \eqref{Rprime} shows that the boundary velocity depends in a non-local way on values of $\Delta u$ at $\|x\| = R_t$.

The paper is organised as follows.  In Section \ref{sec:obstacle} we describe a parabolic obstacle problem, which is related to \eqref{pbu} and which we will eventually use to prove Theorems \ref{thm:exists u} and \ref{thm:conv u}.  Our results about the obstacle problem are proved in Sections \ref{sec:vexists} and \ref{sec:vsteady}, after we first develop some preliminary analytical results in Section \ref{sec:tools}.  Theorem \ref{thm:exists u}  about well-posedness of the problem~\eqref{pbu} is proved in Section \ref{sec:uprops}. Theorem \ref{thm:conv u} is proved in Section \ref{sec:uconv}.

\vspace{0.3in}

{\bf Acknowledgements:}   The work of JN was partially funded through grant DMS-1351653 from the US National Science Foundation.

\section{The obstacle problem} \label{sec:obstacle}

Our approach to proving Theorem~\ref{thm:exists u} and Theorem~\ref{thm:conv u} relies on a connection between the free boundary problem \eqref{pbu} and a related parabolic obstacle problem in one spatial dimension.   For the moment, let us assume there exists a solution $(u,R)$ to \eqref{pbu}.  For $r>0$, recall that $\B(r)=\{x \in \R^d : \|x\|<r\}$ is the open ball of radius $r$ 
centred at the origin, 
and introduce the function $v:[0,\infty) \times (0,\infty) \to\R$ as the mass of $u$ within distance
$x$ of the origin at time $t$:
\begin{equation}
v(x,t) = \int_{\B(x)} u(y,t) \, \diffd y. \label{vdef}
\end{equation}
Notice that $x\mapsto v(x,t)$ is non-decreasing and that $v(x,t)=1$ for $x\ge
R_t$; $v$ is the cumulative radial distribution function for the probability density $u$.  Let $v_0(x)=\mu_0\big(\B(x)\big)$. If $(u,R)$ is a solution of the free boundary problem~\eqref{pbu}, then $v$ must satisfy
\begin{equation}
\begin{cases}
\displaystyle	0 \leq v(x,t) \leq 1,& \text{for $t>0$, $x \geq 0$,}\\[1ex]
\displaystyle	\partial_t v = \partial_x^2v -\frac{d-1}x\partial_x v+v, &\text{if $v(x,t) <1$,}   \\[1ex]
v(0,t)=0, &\text{for $t>0$},\\
v(x,t) \text{ is continuous on $[0,\infty) \times (0,\infty)$}, \\
\partial_x v(\cdot ,t) \text{ is continuous on $[0,\infty)$}, & \text{for $t>0$,}\\
v(\cdot,t)\to v_0 &\text{in $L^1_\text{loc}$ as $ t\searrow0$ }.
\end{cases}
\label{pbv}
\end{equation}
(See Lemma \ref{lem:vfromu} below.) This is a parabolic obstacle problem:  $v$ is bounded from above by the obstacle $v \leq 1$, and $v$ satisfies a parabolic PDE wherever $v$ is strictly below the obstacle. For $v$ defined by \eqref{vdef}, the set $\{ x \;:\; v(x,t) < 1 \}$ coincides with $[0,R_t)$, and $v(x,t) = 1$ for $x \geq R_t$.  
Note also that the initial condition $v_0(x)=\mu_0(\B(x))$ corresponding to~\eqref{vdef} is non-decreasing.
However, as we will show, the formulation in \eqref{pbv} also makes sense for a non-monotone
initial condition $v_0: [0,\infty) \to [0,1]$.

We say that $v$ is a classical solution to~\eqref{pbv} if $v\in C([0,\infty)\times (0,\infty))\cap C^{2,1}(\Omega)$, where
$\Omega = \{(x,t)\in (0,\infty)\times (0,\infty) \, : \, v(x,t)<1\}$, and~\eqref{pbv} holds.
An important step in proving Theorem~\ref{thm:exists u} is to first show the
corresponding existence and uniqueness result for $v$:

\begin{thm}\label{thm:exists v}
Let $v_0: [0,\infty) \to[0,1]$ be a measurable initial condition. Then there exists
a unique function $v(x,t)$ defined on $[0,\infty) \times (0,\infty)$ which is a classical solution to~\eqref{pbv}. 

Furthermore, this unique solution has the following properties:
\begin{itemize}
\item Let $\tilde v$ and $v$ denote the two solutions corresponding
to the initial data $\tilde v_0$ and $v_0$. If $\tilde
v_0\ge v_0$, then $\tilde v \ge v$.
\item If $v_0$ is non-decreasing, 
	the map $x\mapsto v(x,t)$ is 
non-decreasing for all $t>0$.
\item The solution $v$ is continuous with respect to the initial condition in the
following sense:
if $v$ and $\tilde v$ are the two solutions to \eqref{pbv} corresponding to
the initial data $v_0$ and $\tilde v_0$, then for $t>0$,
$\big\|v(\cdot,t)-\tilde v(\cdot,t)\big\|_{L^\infty} \le e^t 
  \big\|v_0-\tilde v_0\big\|_{L^\infty}$, and $\big\|v(\cdot,t)-\tilde v(\cdot,t)\big\|_{L^1} \le e^t  \big\|v_0-\tilde v_0\big\|_{L^1}$.
\item If $v_0$ is non-decreasing and $v_0(+\infty) = 1$,
then
for $t>0$, the boundary position $R_t:=\inf\{x: v(x,t)=1\}$ exists and
is finite, and the function $t\mapsto R_t$ is
continuous for $t>0$.
Moreover, $\lim_{t \searrow 0} R_t = \inf \{x: v_0(x) = 1\}\in [0,\infty]$, and for any $\alpha<1/2$, there exists $C_\alpha <\infty$ such that $R_t-R_s\le C_\alpha (t-s)^\alpha$ for all $s\ge 0$ and $t\in (s,s+1]$.
\item If $v_0$ is continuous at $x_0 \geq 0$, then $v(x,t)\to v_0(x_0)$ as $(x,t)\to (x_0,0)$.
 \end{itemize}
\end{thm}

Our strategy for solving the free boundary problem \eqref{pbu} is to first solve the obstacle problem~\eqref{pbv} with the non-decreasing initial condition
\[
v_0(x) = \int_{\B(x)} \mu_0(\diffd y), \quad x \geq 0.
\] 
The solution $v$ of the obstacle problem then determines the location of the free boundary $R_t$ in~\eqref{pbu}, according to $R_t = \inf \{ x  \;:\; v(x,t) = 1\}$. Having determined the continuous free boundary, we then construct a solution to \eqref{pbu}.  Knowing the cumulative radial distribution function $v$ is not enough to determine $u$, as the function $u$ may not be spherically symmetric.  We propose a probabilistic representation (see \eqref{udef}, below) for the solution $u$, and then verify that the function defined by this representation is indeed a solution to \eqref{pbu}. Notice that only non-decreasing initial data with $v_0(\infty)=1$ are relevant for studying the hydrodynamic limit of the $N$-BBM, but the problem~\eqref{pbv} is interesting in its own right for an arbitrary initial datum.

Parabolic obstacle problems have been studied by several authors, and they
can be formulated as a variational problem in appropriate Sobolev spaces
(see Chapter 1, Section 8 \cite{Friedman82} or Chapter 3, Section 2 of
\cite{BL78}) or via stochastic control representations and viscosity
solution techniques (see \cite{EKPP97}, or Chapter 3 of \cite{BL78}).
Because of the particularities of~\eqref{pbv} (including the unbounded
spatial domain, the singularity of the drift at $x = 0$, the $L^1_\text{loc}$ convergence to initial data), Theorem \ref{thm:exists v} does not seem to follow immediately from existing results. Instead, given the explicit form of the operator in \eqref{pbv}, we find it convenient to make use of the Green's function for the associated linear equation.  In this way, we give a self-contained treatment of the obstacle problem \eqref{pbv}, including estimates for the free boundary in the case of a non-decreasing initial condition.  Our proof of Theorem \ref{thm:exists v} is given in Section \ref{sec:vexists} below. The proof of existence of $v$ and the continuity of $R_t$  builds on ideas from~\cite{BBP}, using the Green's function for the linear problem.  The  proof of uniqueness is based on a comparison principle, in the spirit of uniqueness for viscosity solutions.

\subsection*{Convergence of $v$ to the steady state}

For dimension $d \geq 1$, the eigenfunction $U$ and the constant $R_\infty$ as defined in~\eqref{eqU} are related to Bessel functions. The function
\[
V(x) = \int_{\B(x)} U(y) \,\diffd y
\]
is the unique non-negative continuously differentiable function on $[0,\infty)$ that satisfies 
\begin{equation}\label{eqV}
V''-\frac{d-1}xV' +V=0\quad\text{for
$0<x<R_\infty$},\qquad V(0)=0,\qquad V(x)=1\quad\text{for $x\ge
R_\infty$}.
\end{equation}
Specifically, $V$ is given by
\begin{equation}\label{eq:VJ}
V(x)=\int_{\B(x)}U(y)\,\diffd y=
\begin{cases}
\displaystyle\alpha x^{\frac d2} J_{\frac d 2}(x)&\text{for $0\le x<
R_\infty$},\\[1ex] 1&\text{for $x \ge R_\infty$},\end{cases}
\end{equation}
where $J_\nu(x)$
is the Bessel function of the first kind, solution to $x^2J_\nu''+
x J_\nu'+(x^2-\nu^2)J_\nu=0$ with $J_\nu(0)=0$, $R_\infty$ is the position of the
first positive local maximum of $x\mapsto x^{\frac d2} J_{\frac d 2}(x)$, and
$\alpha$ is chosen such that $V(x)\to1$ as $x\nearrow R_\infty$.

The proof of Theorem \ref{thm:conv u} is based on the corresponding convergence result for the one-dimensional obstacle problem \eqref{pbv}:

\begin{thm}\label{thm:conv}
Let $v$ be the solution to the problem~\eqref{pbv} with initial condition $v_0:[0,\infty) \to [0,1]$, where $v_0$ is non-decreasing and not identically zero. 
For $t>0$, let $R_t=\inf\{x:v(x,t)=1\}$.
Then $R_t \to R_\infty$ as $t \to \infty$, and
\begin{equation}
\lim_{t\to \infty} \|v(\cdot,t) - V(\cdot) \|_{L^\infty} =0.  \label{vtoV}
\end{equation}
Moreover, for $c,K\in (0,\infty)$,
there exist $A>0$, $\lambda>0$ (independent of $c$ and $K$) and $B=B(c,K)>0$ such that
if $v_0(K)\geq c$ then
\begin{align} \label{eq:vconvbounds}
& -\frac B t\le v(x,t) -V(x) \le A e^{-\lambda t},\qquad \forall x\ge0, \, t>0,
\end{align}
and
\begin{align}\label{eq:Rconvbounds}
R_t & \geq R_\infty - A e^{-\lambda t}, \qquad \text{for all $t>0$}, \notag \\
R_t & \leq R_\infty + \frac{B}{t} ,\hspace{1.35cm} \text{for all $t\ge B$}.
\end{align}
\end{thm}

We prove this result in Section \ref{sec:vsteady} below.  The bounds in~\eqref{eq:vconvbounds} and~\eqref{eq:Rconvbounds} will be used in~\cite{BBNP2} for the control of the long term behaviour of the $N$-BBM particle system for large $N$.

\section{Toolbox} \label{sec:tools}

In this section we gather some tools which will be useful for analysing solutions of the obstacle problem~\eqref{pbv}.

\subsection{Green's function} \label{sec:green}
Let $G(y,x,t)$ denote the fundamental solution, or Green's function, for the linear equation
\begin{equation}\label{eqG}
\partial_t G =\partial_x^2 G -\frac{d-1}x \partial_x G\quad\text{for } x>0, \, t>0,\qquad
G(y,0,t)=0,\qquad G(y,x,0)=\delta(y-x).
\end{equation}
In this section, we introduce
several properties of $G$ which we will use in later sections.

The Green's function is related to Brownian motion in dimension $d$. Let $\Phi$ denote the transition function for a $d$-dimensional Brownian motion $B_t$ with diffusivity $\sqrt{2}$, i.e.~let
\begin{equation} \label{eq:PhiSec3def}
\Phi(z_1,z_2,t)=(4\pi t)^{-d/2} e^{-\frac 1 {4t}\|z_2-z_1\|^2},
\end{equation}
so that $\Phi(z_1,\cdot,t)$ is the density of $B_t$ conditional on $B_0=z_1$.
Then the cumulative distribution of the norm process $\|B_t\|$ conditional on $\|B_0\| = y$ is, by symmetry,
\begin{equation} \label{wdef}
w(y,x,t) := \P\big(\|B_t\| < x \;\big|\; \|B_0\| = y\big) =  \int_{\B( x)} \Phi(y \,\mathbf e_1,z,t) \, \diffd z,
\end{equation}
where $\mathbf e_1\in \R^d$ is an arbitrary fixed unit vector. 
Since $(z_2,t) \mapsto \Phi(z_1,z_2,t)$ satisfies the heat equation in $\R^d$, the function $w$ satisfies
\begin{equation}
\partial_t w = \partial_x^2 w - \frac{d-1}{x} \partial_x w, \quad \quad w(y,0,t) = 0 \label{weqn1} 
\end{equation}
with initial condition 
\[
w(y,x,0) = \left \{\begin{array}{cc} 0  \quad &\text{for }x \leq y, \\ 1  \quad & \text{for }x >  y. \end{array} \right.
\]
Then, $-\partial_y w$ also satisfies \eqref{weqn1} with initial condition
      $-\partial_y w(y,x,0)=\delta(x-y)$ which means, comparing to \eqref{eqG}, that
\begin{equation}
G(y,x,t) = -\partial_y w(y,x,t)  =  -\int_{\B( x)}   \mathbf e_1
\cdot \nabla_{\!z_1} \Phi(y \,\mathbf e_1,z,t)  \, \diffd z.  \label{Gdef}
\end{equation}
In particular, notice that $G\ge 0$, and for $t>0$,
\begin{equation}
\int_0^\infty G(y,x,t) \,\diffd y = - \int_0^\infty \partial_y w(y,x,t) \,\diffd y =w(0,x,t)= \P\big( \|B_t\| < x \; \big| \; B_0=0\big), \label{Gyintegral}
\end{equation}
which converges to $1$ as $t \searrow 0$, for any $x > 0$. 

We now state some useful properties of $G$:

\begin{lem}\label{bounds int g}
For each dimension $d$, there exists a constant $C>0$ such that for $t>0$, $x>0$,
\begin{equation} \label{eq:Gintbounds}
\begin{gathered}
\int_0^\infty\diffd y\, G(y,x,t) \le \min\Big(1, C \frac{x^d}{t^{d/2}}\Big),\\
 \int_0^\infty\diffd y\,\big|\partial_x G(y,x,t)\big| \le  C \min \left( \frac{1}{t^{1/2}} , \frac{x^{d-1}}{t^{d/2}}\right) \qquad \text{and}
\qquad  \int_0^\infty\diffd y\,\big|\partial_x^2 G (y,x,t)\big| \le\frac{C}{t}.
\end{gathered}
\end{equation}
Furthermore, for all $t>0$ and $y_0>0$, if $x\in(0,y_0)$,
\begin{equation}\label{eq:Gintbound2}
\int_{y_0}^\infty\diffd y\, \big|\partial_x G(y,x,t)\big| \le
C \frac{x^{d-1}}{t^{d/2}}e^{-\frac{(y_0-x)^2}{4t}}.
 \end{equation}
For all $t > 0$, $y>0$,
\begin{equation}
 \int_0^\infty \diffd x\, G(y,x,t) \leq 1. \label{Grint1}
\end{equation}
If $v_0 \in L^\infty(0,\infty)$, then 
\begin{equation} \label{eq:Gv0bound}
\left| \int_0^\infty G(y,x,t) v_0(y) \,\diffd y\right|  \leq \Vert v_0 \Vert_{L^\infty}, \quad \forall\;\; x > 0, \; t > 0,
\end{equation}
and for any $p \geq 1$, the convergence 
\begin{equation}
\int_0^\infty G(y,\cdot ,t) v_0(y) \,\diffd y \to v_0(\cdot), \quad \text{as $t \to 0$} \label{GintConvv}
\end{equation}
holds in $L^p_{\mathrm{loc}}(0,\infty)$. If $v_0$ is continuous and bounded on $[0,\infty)$, then the convergence \eqref{GintConvv} holds locally uniformly on $[0,\infty)$ as $t \to 0$. 
\end{lem}
The proof of Lemma~\ref{bounds int g} is postponed to the Appendix, Section \ref{appendix}.

\subsection{Feynman-Kac formula}

\begin{prop}[Feynman-Kac]\label{FK}
Take $T>0$ and let
$\Omega$ be an open subset of $(0,\infty) \times (0,T)$.
Suppose $g\in C(\bar\Omega)$ is bounded, and
$w$ is bounded on $\bar\Omega$ and satisfies
\begin{subequations}\label{diri1}
\begin{align}
&\partial_t w = \partial_x^2 w -\frac{d-1}x \partial _x w+w g \quad\text{for
$(x,t)\in\Omega$},\label{dirieq}\\
& w\in C(\bar\Omega) \cap C^{2,1}(\Omega) \label{diriC0}.
\end{align}
\end{subequations}
Then for $(x,t)\in \Omega$,
\begin{equation} \label{eq:FKgen}
w(x,t)=\E_x\left[e^{\, \int_0^\tau g(X_s, t-s)\,\diffd s} w  (X_\tau,t-\tau)\right]
\end{equation}
where, under the probability measure corresponding to $\E_x$,
$(X_s)_{s\leq \tau}$ solves
\begin{equation}
\diffd X_s = \diffd W_s -\frac{d-1}{X_s}\,\diffd s,\qquad X_0=x,
\label{processX}
\end{equation}
with $W$ being a Brownian motion on $\R$ with diffusivity $\sqrt{2}$, and $\tau$ being the backward exit time of $X$ from the domain $\Omega$ given by
$$\tau:=\inf\big\{s>0: (X_s,t-s)\not\in\Omega \big\}.$$
\end{prop}

This result follows from a standard argument, e.g. Theorem~II.2.3 of~\cite{Freidlin85} or Theorem~5.7.6 of~\cite{KS91}, but for completeness we provide a proof later in the Appendix, Section~\ref{appendix}.

\section{Proof of Theorem \ref{thm:exists v}: Existence, uniqueness, and
properties of~\texorpdfstring{$v$}{v}} \label{sec:vexists}

The proof of Theorem~\ref{thm:exists v} is divided into three parts written in 
the next three
subsections: uniqueness, existence of the solution and basic properties, and properties of the free boundary $R_t$.

\subsection{Uniqueness}
\label{sec:uniq}
Before proving existence of a solution to~\eqref{pbv}, we show in this section that there can be at most one solution, and we establish a useful comparison principle. Throughout this
subsection, we suppose that a function $v$ is a solution to  \eqref{pbv} with initial condition $v_0:[0,\infty)\to [0,1]$ measurable.  For a given initial condition $\vl_0\in L^\infty (0,\infty)$, we let
\begin{equation} \label{eq:vlformula}
\vl(x,t)=e^t\int_0^\infty\diffd y\, G(y,x,t)\vl_0(y),
\end{equation}
where $G(y,x,t)$ is the fundamental solution introduced in~\eqref{Gdef}. This $\vl$ is a solution to the linear problem
\begin{equation}\label{linear equ}
\begin{cases}
\partial_t \vl = \partial_x^2 \vl -\frac{d-1} x \partial_x \vl
+ \vl, \quad &\text{for }t>0, \, x> 0,\\
\vl(0,t)=0, \quad &\text{for }t > 0, \\
\vl(\cdot,t)\to\vl_0 \quad &\text{in $L^1_\text{loc}$ as
$t\searrow0$},
\end{cases}
\end{equation}
and it is the unique solution to \eqref{linear equ} which is bounded on $[0,\infty) \times [0,T]$ for each $T > 0$.  The following lemmas, which will be proved in Section~\ref{prove 2 lemmas}, establish a comparison principle between $v$ and $v^\ell$:

\begin{lem} \label{bound above}
If $v_0\le \vl_0$, then $v(\cdot,t)\le \vl(\cdot,t)\
\forall t>0$.
\end{lem}

\begin{lem}\label{bound below}
If $\vl_0\le v_0$
and $\vl(\cdot,s)\le1$ for all $s\le t$, then $\vl(\cdot,t)\le v(\cdot,t)$.
\end{lem}

We introduce the operators $G_t$ and $C_m$ by letting
\begin{equation} \label{eq:GCdef}
G_tf(x)=\int_0^\infty\diffd y\, G(y,x,t)f(y)
\qquad \text{and}\qquad C_m f(x) = \min\big[f(x),m\big].
\end{equation}
In particular, $\vl=e^tG_t\vl_0$. For $\delta>0$ and $n\in \N_0$, we define
\begin{equation} \label{eq:defv+v-}
v^{n,\delta,-}=\big[e^\delta G_\delta C_{e^{-\delta}}\big]^n v_0,
\qquad
v^{n,\delta,+}=\big[C_1e^\delta G_\delta \big]^n v_0.
\end{equation}
The following two lemmas show that $v^{n,\delta,-}$ and $v^{n,\delta,+}$
bound $v$ and are close to each other:
\begin{lem}\label{+- bounds}
For any $\delta > 0$ and $n \in \N$,
\[
v^{n,\delta,-}(x)\le v(x,n\delta) \le v^{n,\delta,+}(x) \quad \forall \;\; x 
\geq 0.
\]
\end{lem}
\begin{proof}The second inequality follows immediately by induction on $n$, using 
Lemma~\ref{bound above} and the fact that $v\le 1$.
For the first inequality, suppose for an induction argument that for some $n\in \N_0$, $v^{n,\delta,-}\le v(\cdot,n\delta)$.
Then for $s\le \delta$, since $\|G_s f\|_{L^\infty}\le \|f\|_{L^\infty}$ (by~\eqref{eq:Gv0bound} in Lemma~\ref{bounds int g}),
$$
\|e^s G_s C_{e^{-\delta}}v^{n,\delta,-}\|_{L^\infty}
\le e^s\| C_{e^{-\delta}}v^{n,\delta,-}\|_{L^\infty}
\le e^s e^{-\delta}\le 1.
$$
Hence by Lemma~\ref{bound below}, $v^{n+1,\delta,-}\le v(\cdot,(n+1)\delta)$, and the result follows by induction on $n$.
\end{proof} 
 
\begin{lem}\label{+- close}
For any $\delta > 0$ and $n \in \N$,
\[
\big\Vert v^{n,\delta,+}-v^{n,\delta,-}\big\Vert_{L^\infty} \le (e^{n\delta}+1)(e^\delta-1).
\]
\end{lem}
\begin{proof}
Notice the following bounds (the first of which comes from \eqref{eq:Gv0bound}
in Lemma~\ref{bounds int g}): for $\delta>0$ and $f,g\in L^\infty [0,\infty)$,
\begin{equation}
\label{some bounds}
\begin{gathered}
\Vert G_\delta f -  G_\delta g\Vert_{L^\infty} =\Vert G_\delta (f -  g)\Vert_{L^\infty}  \le \Vert
f-g\Vert_{L^\infty},
\qquad
\Vert C_1e^\delta f -  C_1e^\delta g\Vert_{L^\infty}\le e^\delta\Vert
f-g\Vert_{L^\infty},\\
\Vert C_1e^\delta f - f\Vert_{L^\infty}\le e^\delta -1\text{ \ if $\Vert
f\Vert_{L^\infty}\le1$}.
\end{gathered}
\end{equation}
Note that $e^\delta G_\delta C_{e^{-\delta}}f
= G_\delta C_1 e^\delta f$,
 so we can write $v^{n,\delta,-}
=G_\delta \big[C_1e^\delta G_\delta\big]^{n-1}C_1e^\delta v_0$, and compare this to $v^{n,\delta,+}=C_1e^\delta G_\delta\big[C_1e^\delta
G_\delta\big]^{n-1}v_0$.
We bound the supremum of $v^{n,\delta,+}-v^{n,\delta,-}$ using the
triangle inequality:
\begin{align} \label{eq:trianglev+v-}
\big\Vert v^{n,\delta,+}-v^{n,\delta,-}\big\Vert_{L^\infty}&=\big\Vert C_1 e^\delta 
G_\delta\big[C_1e^\delta
G_\delta\big]^{n-1}v_0 - G_\delta \big[C_1e^\delta
G_\delta\big]^{n-1}C_1e^\delta v_0\big\Vert_{L^\infty} \notag 
\\&\le
\big\Vert C_1 e^\delta G_\delta\big[C_1e^\delta
G_\delta\big]^{n-1}v_0 - G_\delta\big[C_1e^\delta
G_\delta\big]^{n-1}v_0\big\Vert_{L^\infty} \notag 
\\&\quad +\big\Vert G_\delta\big[C_1e^\delta
G_\delta\big]^{n-1}v_0
-G_\delta \big[C_1e^\delta
G_\delta\big]^{n-1}C_1e^\delta v_0\big\Vert_{L^\infty}.
\end{align}
By the first bound in~\eqref{some bounds},
$\|G_\delta\big[C_1e^\delta
G_\delta\big]^{n-1}v_0\|_{L^\infty}\le 1$, and so by the third bound in~\eqref{some bounds},
the first term on the right hand side of~\eqref{eq:trianglev+v-} is smaller than
$e^\delta -1$.
By successive applications of the first two bounds in~\eqref{some bounds}, and then by the third bound in~\eqref{some bounds}, the second term on the right hand side of~\eqref{eq:trianglev+v-} is smaller than $e^{\delta(n-1)}\Vert
v_0-C_1e^\delta v_0\big\Vert_{L^\infty}\le e^{\delta(n-1)}(e^\delta-1)$.
\end{proof}
The uniqueness of solutions of~\eqref{pbv} is now straightforward. Suppose that $v$ and $v'$ are two
solutions to~\eqref{pbv} with the same initial condition
$v_0=v'_0$. Take $n\in \N$; then by
Lemmas~\ref{+- bounds} and~\ref{+- close}, for $t>0$,
$$\Vert v(\cdot,t)-v'(\cdot,t)\Vert_{L^\infty}
\le \Vert v^{n,\frac t n,+}-v^{n,\frac tn,-}\Vert_{L^\infty}\le
(e^t+1)(e^{\frac t n}-1).$$
By letting $n\to \infty$, we conclude that $v(\cdot,t)=v'(\cdot,t)$.

\subsubsection{Proofs of Lemmas~\ref{bound above} and \ref{bound below}}
\label{prove 2 lemmas}

For convenience, we define $L$ to be the differential operator appearing in \eqref{pbv}:
\begin{equation}
Lv = \partial_x^2 v - \frac{d-1} x \partial_x v. \label{Ldef}
\end{equation}
For $r > 0$ and a point $(x_0,t_0) \in (0,\infty) \times (0,\infty)$, let us define the backward parabolic cylinder:
\begin{equation}
Q^-_r(x_0,t_0) = \{ (x,t) \;:\;\; |x - x_0| < r, \;\; t_0 - r^2 < t \leq
t_0 \}. \label{Qcylinder}
\end{equation}
Recall that the parabolic boundary (see \cite{Evans2010}) of
$Q^-_r(x_0,t_0)$ is the set $\overline{Q_r^-(x_0,t_0)} \setminus
Q_r^-(x_0,t_0)$, which excludes the top of the cylinder $\{ (x,t_0) \;:\;\;
|x - x_0| < r\}$.

\begin{lem} \label{lem:viscsub}
Let $v$ be a solution to \eqref{pbv}.  Suppose there exist $t_0 > 0$, $x_0 > 0$, $r > 0$ and a function $\phi \in C^{2,1}$ such that
\begin{equation}
0 = \phi(x_0,t_0) - v(x_0,t_0) = \min_{(x,t) \in \overline{Q^-_r(x_0,t_0)}} \left( \phi(x,t) - v(x,t) \right).  \label{visctouch}
\end{equation}
Then $\partial_t \phi \leq L \phi + \phi$ must hold at $(x_0,t_0)$.
\end{lem}

\begin{proof}[Proof of Lemma \ref{lem:viscsub}]

If $v(x_0,t_0)<1$, then $v\in C^{2,1}$ locally.  So, if \eqref{visctouch} holds at such a point, then we must
have $ \partial_t (\phi-v)\le0$, $\partial_x(\phi-v)=0$ and
$\partial_x^2 (\phi-v)\ge 0$ (since $\phi-v$ is at a local minimum)
and $\partial_t v=L v +v$ (since we assumed $v<1$). This implies
that $\partial_t \phi\le Lv+v\le  L \phi +\phi$ at $(x_0,t_0)$, as required.

Suppose instead that \eqref{visctouch} holds at a point where $v(x_0,t_0) = 1$. Arguing by contradiction, let us suppose that  there is $\delta > 0$ such that $\partial_t \phi \geq L \phi + \phi
+ \delta$ at $(x_0,t_0)$. In the rest of the proof, we argue that this is impossible. We may make $r$ in \eqref{visctouch} smaller so that $\overline{Q^-_r(x_0,t_0)} \subset (0,\infty) \times (0,\infty)$ and
$\partial_t \phi - L \phi - \phi > \delta/2$ holds for all $(x,t) \in
Q^-_r(x_0,t_0)$.  Since $v(x_0,t_0)
= 1$, then $\phi(x_0,t_0) = 1$, by \eqref{visctouch}.  Moreover,
$\partial_x v(x_0,t_0) = \partial_x \phi(x_0,t_0) = 0$, since $x \mapsto
v(x,t_0)$ is $C^1$ and $v \leq 1$ attains a local maximum at $(x_0,t_0)$ and
$x \mapsto (\phi - v)(x,t_0)$ attains a local minimum. For $\beta > 0$, consider
the function 
$$\psi(x,t) = \phi(x,t) + \beta (x - x_0)^2 + \beta (t_0 - t).$$
If $\beta$ and $r$ are small enough, we now have
a function $\psi$ with the following properties: 
\begin{align*}
\text{(i)\quad}
	& \psi(x_0,t_0)=v(x_0,t_0)=1,
&
\text{(ii)\quad}
	& \psi(x,t) > v(x,t)\quad\text{in $\overline{Q_r^-(x_0,t_0)} \setminus
\{(x_0,t_0)\}$},
\\
\text{(iii)\quad}
	& \partial_t \psi - L \psi - \psi > 0\quad\text{in $\overline{Q_r^-(x_0,t_0)}$},
&
\text{(iv)\quad}
	& \partial_x \psi(x_0,t_0)=0,
\\
\text{(v)\quad}
	& \partial_x^2\psi \neq 0\quad\text{in $\overline{Q_r^-(x_0,t_0)}$},
&
\text{(vi)\quad}
	& \partial_t\psi \neq 0\quad\text{in $\overline{Q_r^-(x_0,t_0)}$}.
\end{align*}
Regarding points (v) and (vi), we first choose $\beta$ small
enough and such that $\partial_x^2\psi(x_0,t_0) \neq 0$ and $\partial_t
\psi(x_0,t_0) \neq 0$. Then, by continuity of $\partial_x^2\psi$ and
$\partial_t\psi$, we may decrease $r$, if necessary, so that $\partial_t \psi$ and $\partial_x^2 \psi$ do not change sign over $\overline{Q_r^-(x_0,t_0)}$.

Now, let $w = \psi - v$.  For $s \in (t_0 - r^2, t_0]$, define the subset
\[
A_s = Q_r^-(x_0,t_0) \cap \{ (x,t):t \leq s\}
\]
(thus, $A_{t_0} = Q_r^-(x_0,t_0)$). Let $m_s$ be the minimum of $w$ on
$\overline{ A_s}$ and $(x_s,t_s)$ a point where this minimum is reached:
\begin{equation}
m_s = \min_{(x,t) \in \overline{A_s}} w(x,t) = w(x_s,t_s)\qquad
\text{with }(x_s,t_s)\in \overline{ A_s}. \label{msmin}
\end{equation}
Observe that $m_{t_0} = w(x_0,t_0) = \psi(x_0,t_0) - v(x_0,t_0) = 0$ and $m_s > 0$ for $s < t_0$.  Because 
of the continuity of $w$, $m_s$ is continuous, non-increasing in $s$, and $m_s 
\to 0$ as $s \nearrow t_0$. Furthermore, again by continuity,
$$(x_s,t_s)\to (x_0,t_0)\quad\text{as $s\nearrow t_0$}.$$
In particular, for $s$ close enough to $t_0$, we have $(x_s,t_s)\in A_s$
(i.e.~$(x_s,t_s)$ is not in the parabolic boundary of $A_s$).
From now on, we only consider values of $s$ close
enough to $t_0$ for this property to hold.

By the same argument as the one we used to show that $v(x_0,t_0)=1$,
we claim that necessarily
\begin{equation} \label{eq:vxsts1}
v(x_s,t_s)=1.\end{equation}
Indeed, if we had $v(x_s,t_s)<1$ for some $s$, we would have at that
point $w=\psi-v\in C^{2,1}$ with
$\partial_t w\le0$, $\partial_x w = 0$ and $\partial_x^2w\ge0$ and then
$\partial_tv=Lv+v$ would impliy $\partial_t\psi\le L \psi+\psi$, a contradiction by property~(iii).
Then, $v(x_s,t_s)=1$ implies that $\partial_x v(x_s,t_s)=0$ and therefore
that $\partial_x\psi(x_s,t_s)=0$.

Because of properties (iv) and (v) of $\psi$, the Implicit Function
Theorem implies that for some $\tau< t_0$ close enough to $t_0$ there is
a $C^1$ curve $\gamma:[\tau,t_0] \to [x_0 - r,x_0 + r]$ such
that $\gamma(t_0) = x_0$, and for all $(x,t) \in
\overline{Q_{r}^-(x_0,t_0)}$ with $t\ge\tau$, one has $\partial_x\psi(x,t)
= 0$ if and only if $x
= \gamma(t)$. In
particular, if $t_s\ge \tau$ then since $\partial_x \psi(x_s,t_s)=0$, $$x_s=\gamma(t_s).$$  Depending on the signs of $\partial_x^2\psi$ and $\partial_t\psi$, we now consider three cases and arrive each time at a contradiction.

First, suppose $\partial_{x}^2 \psi > 0$ in $\overline{Q_{r}^-(x_0,t_0)}$.
Then 
for $t\in (\tau,t_0)$ such that $\psi(\gamma(t),t)\ge 0$, 
since
$\partial_x\psi(\gamma(t),t)=0$,
$$\frac{\diffd}{\diffd t}
\psi(\gamma(t),t) = \partial_t \psi(\gamma(t),t) >
L\psi(\gamma(t),t)+\psi(\gamma(t),t)>0,$$
where the last inequality follows since  $L\psi(\gamma(t),t)>0$. Since
$\psi(x_0,t_0) = 1$, this implies that $\psi(\gamma(t),t)<1$ for $t\in [\tau,t_0)$.
As $v\le\psi$ this leads to $v(x_s,t_s)<1$ for $s<t_0$ sufficiently close to $t_0$ that $t_s\ge \tau$, a contradiction by~\eqref{eq:vxsts1}.

Second, suppose that $\partial_{x}^2 \psi < 0$ and $\partial_t \psi > 0$ in
$\overline{Q_{r}^-(x_0,t_0)}$. These conditions, along with
$\partial_x\psi(x_0,t_0) = 0$,
and $\psi(x_0,t_0) = 1$, would imply that $v<\psi < 1$ in
$\overline{Q_{r'}^- (x_0,t_0)} \setminus \{(x_0,t_0)\}$
for some $r'>0$ small enough, contradicting $v(x_s,t_s)=1$ with
$(x_s,t_s)\to(x_0,t_0)$ as $s\nearrow t_0$.

Third, we consider the possibility that $\partial_{x}^2 \psi < 0$ and $\partial_t \psi < 0$ in $\overline{Q_{r}^-(x_0,t_0)}$. 
Under these conditions, for $t\in (\tau,t_0)$,
$\frac{\diffd}{\diffd t} \psi(\gamma(t),t) = \partial_t
\psi(\gamma(t),t)<0$ and hence the function $t\mapsto \psi(\gamma(t),t)$ is
strictly decreasing. With $m_s=w(x_s,t_s)=\psi(\gamma(t_s),t_s)-1$ for $s$ sufficiently close to $t_0$ that $t_s\ge \tau$, this
implies that in fact $t_s=s$ and
$$m_s = \psi(\gamma(s),s)-1.$$
It follows that for all $s$ close enough to $t_0$ and all
$x\in[x_0-r,x_0+r]$ we have
$$\psi(x,s)-v(x,s)\ge \psi(\gamma(s),s)-1.$$
As we assumed $\partial_x^2\psi<0$ and as $\partial_x\psi(\gamma(s),s)=0$,
we have $\psi(x,s)<\psi(\gamma(s),s)$ if $x\ne\gamma(s)$ and so
$$v(x,s)<1\quad\text{if $x\ne\gamma(s).$}$$
This means that, in the case $\partial_{x}^2 \psi < 0$ and $\partial_t \psi
< 0$, for $t$ sufficiently close to $t_0$, the function $v$ attains the value $1$ in $\overline{Q_{r}^-(x_0,t_0)}$ along and only
along the curve $(\gamma(t),t)$. To finish the argument, we now show that
this situation is impossible. Consider the function
$\widetilde\psi(x,t)=\psi(x,t)+h x(t_0-t)$ for some fixed $h>0$.
By choosing $h$ small enough, this function can be made to satisfy all
the properties (i), \ldots, (vi) satisfied by $\psi$, with
$\partial_x^2\widetilde\psi<0$ and
$\partial_t\widetilde\psi<0$ on $\overline{Q_r^-(x_0,t_0)}$, as for
$\psi$.
By repeating the argument
above, 
for $t$ sufficiently close to $t_0$,
$v$  attains the value $1$ in $\overline{Q_{r}^-(x_0,t_0)}$ along and only
along the curve $(\widetilde\gamma(t),t)$,
where $\widetilde\gamma(t)$ is such that $\partial_x\widetilde\psi(\widetilde \gamma (t),t)=0$. But, for
$t<t_0$, we have
$\gamma(t)\ne\widetilde\gamma(t)$  because  $\partial_x\widetilde\psi\ne
\partial_x\psi$, and we obtain a contradiction.
This shows that the case 
$\partial_{x}^2 \psi < 0$ and $\partial_t \psi < 0$ is untenable, as well.
\end{proof}

\begin{proof}[Proof of Lemma~\ref{bound above}]
By the comparison principle for $\vl$, it is sufficient to prove the result for
$\vl_0=v_0$.
By the assumptions in~\eqref{pbv}, $v$ is continuous for $t > 0$ and $v(\cdot,t) \to
v_0$ in $L^1_\text{loc}$ as $t \searrow 0$.  To prove the lemma, we
observe that it suffices to assume that $v$ is continuous on
$[0,\infty) \times [0,\infty)$ (\textit{i.e.}\@ at $t = 0$, as well). To
see why this is the case, let $\epsilon > 0$ and define $h^\epsilon(x,t)
= v(x,t +\epsilon)$. Then $h^\epsilon$ is a solution to \eqref{pbv} with
initial condition $h^\epsilon_0(x) = v(x,\epsilon)$. In particular,
$h^\epsilon$ is continuous on $[0,\infty) \times [0,\infty)$.  If the
conclusion of Lemma~\ref{bound above} holds for such solutions, then we
have
$$h^\epsilon(x,t) \leq e^t\int_0^\infty\diffd y\, G(y,x,t) v(y,\epsilon)
\quad\text{for $t>0$ and $x\ge0$}.$$
With $x\ge 0$ and $t>0$ fixed, take the $\epsilon\searrow0$ limit. The left hand
side converges to $v(x,t)$ by continuity of $v$ for $t>0$ and the right
hand side converges to $\vl(x,t)$ because $v(\cdot,\epsilon)\to v_0$ in
$L^1_\text{loc}$, and hence we have
$$v(x,t) \le \vl(x,t)=e^t\int_0^\infty\diffd y \, G(y,x,t) v_0(y),$$
as required.

Now we proceed assuming $v$ is continuous on $[0,\infty) \times [0,\infty)$, and hence $v_0^\ell=v_0$ is also continuous. For $\delta > 0$, let $\phi = \phi^\delta$ satisfy 
\begin{empheq}[left = \empheqlbrace]{align}
\partial_t \phi & = L \phi + (1 + \delta) \phi, \label{phipde}  \\
\phi(0,t) & = 0, \notag \\
\phi(x,0) & = v^\ell_0(x) +  \min(\delta x , 2). \label{phiic}
\end{empheq}
The comparison principle implies that $\phi^\delta(x,t) \geq v^\ell(x,t)$ for all $t \geq 0$ 
and $x \geq 0$. Moreover, $\phi^\delta (x,t)=e^{(1+\delta)t}\int_{0}^\infty G(y,x,t)\phi^\delta(y,0)\,\diffd y$ and $\phi^\delta(y,0) \to v_0^\ell(y)$ locally uniformly as $\delta \to 0$,
so by the dominated convergence theorem,
$\phi^\delta(x,t) \searrow v^\ell(x,t)$ as $\delta \searrow
0$.
Hence it suffices to show that $\phi^\delta \geq v$. 
Define
\begin{equation}
t_0 = \sup \big\{ t \geq 0 : \phi(x,s) > v(x,s) \quad \forall x > 0,\,
s \in [0,t] \big\}. \label{t0def}
\end{equation}
We will show that $t_0 = +\infty$. 

Suppose that $t_0 < \infty$. Then there is a sequence of points
$\{(x_n,t_n)\}$ with $t_n \searrow t_0$ as $n \to \infty$, and $x_n > 0$,
such that $\phi(x_n,t_n) \leq v(x_n,t_n)$.  
Since $v_0^\ell = v_0 \geq 0$, by~\eqref{Gdef} and~\eqref{wdef} we have that
$$
\phi(x,t)\ge 2\int_{0}^\infty G(y,x,t)\indic{y>2/\delta}\,\diffd y 
=2\P\big(\|B_t\|< x \; \big|\; \|B_0\|=2/\delta\big).
$$
Hence
 there exists an $r$
such that $\phi(x,t)\ge 1.5$ for all $(x,t)\in[r,\infty)\times [0,t_0+1].$
As $v(x,t)\le1$, this implies that $x_n<r$ for $n$ large enough.

So, the sequence $\{x_n\}$ must be confined to a compact interval, and by
taking a subsequence, we infer the existence of a point $(x_0,t_0)$ such
that $(x_{n_j},t_{n_j}) \to (x_0,t_0)$ as $j \to \infty$.  By continuity of
$\phi$ and $v$, we must have $\phi(x_0,t_0) = v(x_0,t_0)$,
and $\phi(x,t)>v(x,t)$ for all $x>0$ and $t<t_0$, 
and $\phi(x,t_0)\ge v(x,t_0)$ for all $x\ge0$.

Notice that, in fact, we must have $\phi(x,t_0)>v(x,t_0)$ for $x>0$.
Indeed, this is obvious from the
initial condition \eqref{phiic} if $t_0=0$, and by
Lemma~\ref{lem:viscsub} we cannot have $\phi(x,t_0)=v(x,t_0)$ for $x>0$,
 $t_0>0$ since $\partial_t\phi-L\phi-\phi=\delta\phi>0$ at such a point. Therefore we must have
$x_0=0$. 

We have shown that if $t_0 < \infty$, we must have $x_0 = 0$, and $\phi(x,t_0) > v(x,t_0)$ for all $x > 0$. Then, because $\phi$ and $v$ are continuous on $[0,\infty) \times [0,\infty)$, we may take $y_1 > 0$ small 
enough so that $v(x,t_0)<1$ for $x\in[0,y_1]$ and then $t_1 > t_0$
small enough so that $\phi(y_1,t) > v(y_1,t)$ holds for all $t \in
[t_0,t_1]$, and $v(x,t) < 1$ for all $(x,t) \in [0,y_1] \times
[t_0,t_1]$.

Then $w := \phi - v$ satisfes $\partial_t w \ge L w + w$ in $Q
:= (0,y_1) \times (t_0,t_1]$, with $w\ge 0$ on the parabolic
boundary of $Q$, and $w>0$ on the boundary $\{y_1 \} \times [t_0,t_1]$.  The
strong maximum principle implies that $w > 0$ everywhere in
$Q$. But the condition $\phi(x,t) > v(x,t)$ for all $(x,t)\in Q$
contradicts the existence of the subsequence $(x_{n_j},t_{n_j})\to (0,t_0)$ as $j\to \infty$. We conclude that $t_0 = +\infty$, and the proof is complete.
\end{proof}

\begin{proof}[Proof of Lemma~\ref{bound below}]
As in the proof of Lemma \ref{bound above}, it suffices to assume $v$ and $v^\ell$ are continuous on $[0,\infty)\times [0,\infty)$.
 Specifically, for $\epsilon > 0$, consider the functions
\[
\tilde v^\ell_\epsilon(x) = \min \left( v(x,\epsilon)\;,\; v^\ell(x,\epsilon) \right)
\]
and
\begin{equation} \label{eq:hepsilonformula}
h_\epsilon(x,s) = e^{s} \int_0^\infty G(y,x,s) \tilde v_\epsilon^\ell(y) \,\diffd y.
\end{equation}
Observe that $h_\epsilon(x,0) = \tilde v^\ell_\epsilon(x)$ is continuous and $h_\epsilon(x,0) \leq v(x,\epsilon)$.  Moreover, for $s \geq 0$, $h_\epsilon(x,s)$ satisfies the same linear PDE as $v^\ell$, and 
\[
h_\epsilon(x,s) \leq e^{s} \int_0^\infty G(y,x,s) v^\ell(y,\epsilon) \,\diffd y = v^\ell(x,s+\epsilon), \quad s \geq 0, \;\; x \geq 0.
\]
Therefore, if $v^\ell(\cdot,s) \leq 1$ holds for $s \leq t$, then $h_\epsilon(x,s) \leq 1$ for $s \in [0,t - \epsilon]$. So, if the lemma holds for solutions of~\eqref{pbv} and~\eqref{linear equ} which are continuous on $[0,\infty)\times [0,\infty)$, then we must have $h_\epsilon(x,s) \leq v(x,s + \epsilon)$ for all $s \in [0,t - \epsilon]$. Now we claim that as $\epsilon \to 0$, $\tilde v^\ell_\epsilon \to v_0^\ell$ in $L^1_{\text{loc}}$: this follows from the definition of $\tilde v^\ell_\epsilon$ and the elementary inequality
\[
\big|\min\big(v(x,\epsilon), v^{\ell}(x,\epsilon)\big)
- \min\big(v_0(x),v^\ell_0(x)\big)\big| \leq \big|v(x,\epsilon)
- v_0(x)\big| + \big|v^{\ell}(x,\epsilon) - v^\ell_0(x)\big|,
\]
because $v(\cdot,\epsilon) \to v_0$ and $v^\ell(\cdot,\epsilon) \to v_0^\ell$ in $L^1_{\text{loc}}$ as $\epsilon \to 0$, and $v_0 \geq v^\ell_0$.
  Therefore, since $\tilde v^\ell_\epsilon \to v_0^\ell$ in $L^1_{\text{loc}}$, we see that by~\eqref{eq:hepsilonformula}, for all $s >0$ and $x\ge 0$,  $h_\epsilon(x,s) \to v^\ell(x,s)$ as $\epsilon \to 0$. Also, for $s>0$ and $x\ge 0$, $v(x,s + \epsilon) \to v(x,s)$ as $\epsilon\to 0$. Hence $v^\ell(x,s) \leq v(x,s)$ must also hold for all $s \in (0,t]$ and $x \geq 0$.

So, we now proceed, assuming that $v$ and $v^\ell$ are continuous on $[0,\infty) \times [0,\infty)$, that $v_0^\ell \leq v_0$, and that $v^\ell(\cdot,s)\le 1$ for all $s\le t$. 
We use a standard maximum principle argument. Introduce
\[
\phi(x,s)=\frac{e^{-2s}\big[v(x,s)-v^\ell(x,s)\big]}{1+x}\qquad\text{and}\qquad
M=\inf\big\{\phi(x,s) \;:\; x\ge0,\,\, s\in[0,t]\big\}.
\]
We will show that $M\ge0$, which implies the lemma. Let $(x_n,t_n)$ be a sequence with $t_n\le t$ for each $n$, and such that $\phi(x_n,t_n)\to M$ as $n\to\infty$. If the sequence $x_n$ is unbounded, then $M=0$ because $|\phi(x,s)| \le \frac{1}{1 + x}$ for all $x \geq 0$ and $s \in [0,t]$. If instead the sequence $x_n$ is bounded, then up to extracting a subsequence, we can assume that $(x_n,t_n)$ converges to some $(x^*,t^*)$ with $t^*\le t$, and
$M=\phi(x^*,t^*)$ by continuity. If $x^*=0$, then $M=\phi(0,t^*)=0$.
If $t^*=0$, then $M\ge0$ since $v_0^\ell \le v_0$. It remains to consider the case $x^*>0$ and
$t^*>0$. At such a point, if $v=1$, then $v^\ell\le1$ and so $M\ge 0$. If
$v<1$, then $v$, $v^\ell$ are both $C^{2,1}$ and satisfy  $\partial_s h=Lh+h$ in a neighbourhood of the point $(x^*,t^*)$. This implies by direct substitution that $\phi$ is also $C^{2,1}$ in a neighbourhood of $(x^*,t^*)$ and satisfies \[
\partial_s \phi =L \phi +\frac2{1+  x}\partial_x \phi -\left(\frac{d-1}{x(1+x)} +1\right) \phi
\]
 at the point $(x^*,t^*)$. But as $(x^*,t^*)$ is the point in $[0,\infty)\times [0,t]$ where $\phi$ is minimal, we must also have
$\partial_s \phi\le0$,
$\partial_x \phi=0$ and
$\partial_x^2 \phi\ge0$ at that point. This implies $\phi(x^*,t^*)=M\ge0$. As we have exhausted all possibilities, the proof is complete.
\end{proof}

\subsection{Existence of the solution and basic properties}\label{sec:exists v}
In this subsection, we construct a solution to \eqref{pbv}. For a measurable initial condition $v_0:[0,\infty) \to [0,1]$ and for $n\ge2$, we
let $(x,t)\mapsto v_n(x,t)$ denote the unique solution to
\begin{equation}\label{eqvn}
\partial_t v_n= \partial_x^2 v_n-\frac{d-1}x\partial_x v_n
+v_n-v_n^n, \qquad\qquad v_n(0,t)=0,\qquad\qquad v_n(x,0)=v_0(x).
\end{equation}
(As $v_n$ is defined only for $n\ge2$, there is no clash of notation with
the initial condition $v_0$.)

By the maximum principle, for any fixed $(x,t)$, the map $n\mapsto
v_n(x,t)$ is non-decreasing. Furthermore, again by the maximum principle,
$v_n(x,t)\in[0,1]$ for all $x,t,n$. This implies that the following
pointwise limit exists:
$$v(x,t):= \lim_{n\to\infty} v_n(x,t).$$
We will show that this limit $v(x,t)$ satisfies all the
conditions in~\eqref{pbv}, thus proving the existence part of
Theorem~\ref{thm:exists v}.

\begin{lem}[Basic properties of $v_n$] \label{lem:uchi}
For any $x\ge 0$, $t_0\ge0$ and $t>t_0$,
\begin{align}
&v_n(x,t)=e^{t-t_0}\int_0^\infty \diffd y\, G(y,x,t-t_0)v_n(y,t_0) \notag \\
&\qquad\qquad\quad -\int_0^{t-t_0}\diffd 
s\,\int_0^\infty\diffd y
\,e^{t-t_0-s}G(y,x,t-t_0-s)v_n(y,s+t_0)^n\label{Uchi}
\\
\text{and }\qquad&\int_0^\infty\diffd y\, G(y,x,t-t_0) v_n(y,t_0) \le
v_n(x,t) \le e^{t-t_0} \int_0^\infty\diffd y\, G(y,x,t-t_0) v_n(y,t_0)
.
\label{Bounds vn}
\end{align}
\end{lem}
\begin{proof}
The first statement is proved by directly checking, using~\eqref{eqG} and~\eqref{GintConvv}, that
\eqref{Uchi} 
satisfies~\eqref{eqvn} and has the correct limit as $t\searrow t_0$. 
The second statement follows from applying the maximum principle
in~\eqref{eqvn}, using that $0\le v_n-v_n^n\le v_n$
(because $v_n\in [0,1]$) and using~\eqref{eqG}. 
\end{proof}

\begin{lem}[Basic properties of $v$]
\label{lem:basic prop v}
For $0\le t_0<t$,
\begin{equation}
\int_0^\infty\diffd y\, G(y,x,t-t_0) v(y,t_0) \le
v(x,t) \le e^{t-t_0} \int_0^\infty\diffd y\, G(y,x,t-t_0) v(y,t_0).
\label{Bound v}
\end{equation}
Moreover,
\begin{itemize}
\item $v(x,t)\in[0,1]$.
\item
$v(0,t)=0$ for $t>0$.
\item If $v_0$ is non-decreasing, then $v(\cdot,t)$ is
non-decreasing for all $t>0$.
\item Suppose $\tilde v_0:[0,\infty)\to [0,1]$ is measurable, and let $\tilde v(x,t)=\lim_{n\to \infty}\tilde v_n(x,t)$,
where $\tilde v_n$ solves~\eqref{eqvn} with initial condition $\tilde v_0$.
If $\tilde v_0\ge v_0$, then $\tilde v(\cdot,t)\ge
v(\cdot, t)$ for all $t>0$.
\item
$v(\cdot,t) \to v_0$ in $L^1_\text{loc}$ as $t\searrow 0$.
\item If $v_0$ is continuous at $x_0 \geq 0$, then $v(x,t)\to v_0(x_0)$ as $(x,t)\to (x_0,0)$. 
\end{itemize}
\end{lem}
\begin{proof}
Equation \eqref{Bound v} and the first four  itemized points are
properties valid for any $v_n$ and which remain true in the $n\to\infty$
limit. For the fifth point, by~\eqref{Bound v} with
$t_0=0$  it is sufficient to notice
that $\int_0^\infty \diffd y\, G(y,\cdot,t) v_0(y)$
converges to $v_0$ in $L^1_\text{loc}$ as $t\searrow 0$, which holds by~\eqref{GintConvv} in Lemma~\ref{bounds int g}.
For the last point, by~\eqref{Bound v} it suffices to show that $\int_0^\infty \diffd y\, G(y,x,t) v_0(y)$
converges to $v_0(x_0)$ as $(x,t) \to (x_0,0)$. If $v_0$ is continuous at $x_0$ (but possibly not continuous on all of $[0,\infty)$) then we can find functions $f^+$ and $f^-$ which are continuous and bounded on $[0,\infty)$, and satisfy
\[
f^-(x) \leq v_0(x) \leq f^+(x), \quad \forall \; x \geq 0, \quad \quad \text{and} \quad \quad f^-(x_0) = v_0(x_0) = f^+(x_0).
\]
Then, by Lemma~\ref{bounds int g}, the desired statement follows, since the convergence
\[
\lim_{t \to 0} \int_0^\infty G(y,x,t)f^\pm(y)\,\diffd y =  f^\pm(x), \quad x \geq 0,
\]
holds uniformly on a neighbourhood of $x_0$, while $\int_0^\infty \diffd y\, G(y,x,t) v_0(y)$ is bounded from above by $\int_0^\infty G(y,x,t)f^+(y)\,\diffd y$ and from below by $\int_0^\infty G(y,x,t)f^-(y)\,\diffd y$.
\end{proof}

We now prove a regularity result on $v(\cdot,t)$ for fixed $t>0$.
\begin{prop}\label{prop:v is C1}
For each $t>0$, the function $x\mapsto v(x,t)$ is $C^1$ on $[0,\infty)$. There exists a constant $C$ (which depends on the dimension $d$, but not on $v_0$) such that for $t>0$,
\begin{equation}\label{dxv<ct}\sup_{x \ge 0}\big|\partial_x v(x,t)\big|\le 
C(1 + t^{-1/2})\end{equation}
and for $x,x'\ge0$ and $t>0$,
\begin{equation}
\big|\partial_x v(x',t)-\partial_x v(x,t)\big|
\le C|x'-x|\left(t^{-1}+1+\log_+(|x'-x|^{-1})\right). \label{dvmodulus}
\end{equation}
For any $t_1 > 0$, there is a constant $C'$ (depending on $d$
and $t_1$, but not on $v_0$) such that
\begin{equation}\label{eq:vxNew}
|\partial_x v(x,t)| \leq C' x^{d-1}, \quad x \geq 0,\;\; t \geq t_1.
\end{equation}
\end{prop}
\begin{proof}
We are going to show that, for any fixed $t>0$, the sequence
of functions $\big(x\mapsto \partial_x
v_n(x,t)\big)_n$ is uniformly bounded and equicontinuous, i.e.~for some $c(t)<\infty$,
\begin{equation}
\label{twobounds}
\sup_{n} \sup_{x \geq 0}\big|\partial_x v_n(x,t)\big|\le c(t),\qquad \text{and} \qquad
\lim_{\epsilon \to0}\sup_{n}\sup_{\substack{x \geq 0\\ y \in (x, x+\epsilon ]}} \big|\partial_x v_n(y,t)-\partial_x
v_n(x,t)\big|=0.
\end{equation}
Then, by the
Arzel\`a–Ascoli Theorem, for each compact set $K\subset [0,\infty)$, there exists a subsequence
$n_k$ such that  $\partial_x v_{n_k}(\cdot,t)$ converges uniformly on $K$
to some continuous limit $\ell(\cdot,t)$ as $k\to\infty$, with
$\|\ell(\cdot,t)\|_{L^\infty}\le c(t)$. Since we know that
$v_{n_k}(\cdot,t)$ converges pointwise to $v(\cdot,t)$, we conclude that the limit
$\ell(\cdot,t)=\partial_xv(\cdot,t)$ on $K$. Therefore, $v(\cdot,t)$
is $C^1$ on $[0,\infty)$ and $|\partial_x v(x,t)|\le c(t)$.

We first prove \eqref{twobounds}. Let $\delta=\min(1,t)$;
write~\eqref{Uchi} in Lemma~\ref{lem:uchi} with $t_0=t-\delta$, change the variable $s$ into
$\delta-s$ and differentiate with respect to $x$:
\begin{align}
\partial_x v_n(x,t) = e^{\delta} \int_0^\infty \diffd y\, \partial_x G(y,x,\delta) v_n(y,t - \delta) -  \int_{0}^{\delta} \diffd s\,e^{s} \int_0^\infty\diffd 
y
\, \partial_x G(y,x,s) v_n(y,t-s)^n. \label{dxvn}
\end{align}
Since $v_n\in[0,1]$ and $\delta \le 1$, and then using~\eqref{eq:Gintbounds} in Lemma~\ref{bounds int g}, we can bound:
\begin{align}
\big|\partial_x v_n(x,t)\big|
&\le e\int_0^\infty\diffd y\,\big|\partial_x G(y,x,\delta)\big|
+e \int_{0}^\delta \diffd s\, \int_0^\infty\diffd y
\,\big|\partial_x G(y,x,s)\big| \notag \\
& \le
e C \bigg( \frac 1{\sqrt{\delta}} + 2\sqrt{\delta}\bigg) \notag \\
&\le e C (3+t^{-1/2}),
\label{bound dxvn}
\end{align}
since $\delta=\min(t,1)$.  This implies the first bound in
\eqref{twobounds}, with $c(t) = eC (3 + t^{-1/2})$.

We now prove equicontinuity of $\partial_x v_n$. 
Take $x'>x\ge0$. 
From
\eqref{dxvn}, we obtain
\begin{align}\label{deltadxvn}
\big|\partial_xv_n(x',t)-\partial_xv_n(x,t)\big|
 & \le e\int_0^\infty\diffd y\, \big|\partial_x G(y,x',\delta)-\partial_x
G(y,x,\delta)\big| \notag 
 \\
& \quad +e\int_0^{\delta}\diffd s\, \int_0^\infty\diffd y\,
\big|\partial_x G(y,x',s)-\partial_x G(y,x,s)\big|.
\end{align}
Using the third bound in~\eqref{eq:Gintbounds} in Lemma~\ref{bounds int g}, for $s>0$,
\begin{align} \label{eq:boundAdxG}
\int_0^\infty\diffd y\, \big|\partial_x G(y,x',s)-\partial_x G(y,x,s)\big|
&=
\int_0^\infty\diffd y\, \left| \int_x^{x'}\diffd z\, \partial_x^2
G(y,z,s) \right| \notag \\
&\le \int_x^{x'}\diffd z \int_0^\infty\diffd y\, \big|\partial_x^2
G(y,z,s)\big|\le \frac{C}{s} (x'-x).
\end{align}
By the second bound in~\eqref{eq:Gintbounds} in Lemma~\ref{bounds int g} we also have
\begin{equation} \label{eq:boundBdxG}
\int_0^\infty\diffd y\, \big|\partial_x G(y,x',s)-\partial_x G(y,x,s)\big| \le
\int_0^\infty\diffd y\, \big|\partial_x G(y,x',s)\big|+
\int_0^\infty\diffd y\, \big|\partial_x G(y,x ,s)\big|\le \frac {2C}{\sqrt s}.
\end{equation}
Let $t^*=(\frac{x'-x}2)^2$; we shall bound the second term on the right hand side of
\eqref{deltadxvn} by using~\eqref{eq:boundAdxG} for $s>t^*$ and~\eqref{eq:boundBdxG} for
$s<t^*$. This gives us that
\begin{align*}
\int_0^{\delta} \diffd s\, \int_0^\infty\diffd y\,
\big|\partial_x G(y,x',s)-\partial_x G(y,x,s)\big| & \le\int_0^{t^*\wedge
\delta}\diffd s \,\frac{2C}{\sqrt s}
+\int_{t^*\wedge \delta}^\delta\diffd s\,\frac {C} s (x'-x) \\
& \leq {C}(x'-x)\bigg[2+\log_+\frac{4\delta}{(x'-x)^2}\bigg].
\end{align*}
Using~\eqref{eq:boundAdxG} to bound the first term on the right hand side of~\eqref{deltadxvn},
 and recalling that $\delta = \min(t,1)$ so that $\frac1{\delta}\le 1+\frac1t$, we obtain
\[
\big|\partial_xv_n(x',t)-\partial_xv_n(x,t)\big|
\le e C(x'-x)\bigg[\frac1t+3+\log_+\frac{4}{(x'-x)^2}\bigg],
\]
for $t > 0$ and $0 \le x < x'$. This implies the second property in~\eqref{twobounds}. 

Since for $t>0$ and $0\le x<x'$, there exists a sequence $n_k \to \infty$ such that
$\partial_x v_{n_k}(\cdot,t)$ converges uniformly on $[x,x']$ to $\partial_x v(\cdot,t)$, it follows that
\begin{equation*}
\big|\partial_x v(x',t)-\partial_x v(x,t)\big|
\le e C(x'-x)\bigg[\frac1t+3+\log_+\frac{4}{(x'-x)^2}\bigg],
\end{equation*}
which completes the proof of \eqref{dxv<ct} and \eqref{dvmodulus}.

We now turn to proving \eqref{eq:vxNew}.  We already have by \eqref{dxv<ct}  that $|\partial_x v(x,t)|$ is bounded on $[0,\infty) \times [t_1,\infty)$; therefore, it suffices to prove that $|\partial_x v(x,t)| \leq C' x^{d-1}$ holds on $(0,b) \times [t_1,\infty)$ for some $b > 0$ and $C'>0$.  Let $\delta = t_1/2$.

We assume $t \geq t_1$, and use~\eqref{dxvn} again.
Due to~\eqref{eq:Gintbounds}, and since $v_n \in [0,1]$, the first integral term on the right hand side of~\eqref{dxvn} is bounded by $C e^{\delta} \frac{x^{d-1}}{\delta^{d/2}}$ for all $t \geq t_1$.  

We now bound the second integral term on the right hand side of~\eqref{dxvn}, using \eqref{eq:Gintbound2} in Lemma~\ref{bounds int g}. First notice from~\eqref{Bounds vn} in Lemma~\ref{lem:uchi} applied with $t_0=t-\delta$, and then using that $v_n\in [0,1]$ and using~\eqref{eq:Gintbounds}, that  for $t\ge\delta$,
\[
v_n(x,t)\le e^{\delta} \int_0^\infty \diffd y\, G(y,x,\delta)\le e^{\delta} C \frac{x^d}{\delta^{d/2}} .
\]
This implies that there exists $b>0$ (independent of $n$) such that $v_n(y,t-s)\le1/2$ for $(y,s) \in [0,2b] \times [0,\delta]$ and $t\ge t_1=2\delta$. So for $t \ge t_1$,
\begin{equation*}
\begin{aligned}
& \left| \int_{0}^{\delta} \diffd s\,e^{s} \int_0^\infty\diffd y
\, \partial_x G(y,x,s) v_n(y,t-s)^n \right| \\
& \qquad \qquad\leq
 e^\delta \int_{0}^{\delta} \diffd s\,\int_0^{2b}\diffd y
\,\big|\partial_x G(y,x,s)\big|\, 2^{-n}
+
e^\delta \int_{0}^{\delta} \diffd s\,\int_{2b}^\infty\diffd y
\,\big|\partial_x G(y,x,s)\big|  
.
\end{aligned}
\end{equation*}
Then, for $x\in(0,b)$, using \eqref{eq:Gintbounds} for the first term and \eqref{eq:Gintbound2}
for the second,
\begin{align}
&\left| \int_{0}^{\delta} \diffd s \,e^s \int_0^\infty\diffd y
\, \partial_x G(y,x,s) v_n(y,t-s)^n\right|
\le
2^{-n} C e^\delta \int_0^\delta\frac{\diffd s}{s^{1/2}}
+C e^\delta x^{d-1}\int_0^\delta \frac{\diffd s}{s^{d/2}}e^{-\frac{b^2}{4s}}
.
\notag
\end{align}

By combining these estimates, we conclude that $|\partial_x v_n(x,t)|\le C' 2^{-n}+ C' x^{d-1}$ if $t \ge t_1$ and $x<b$, for some $b$ and $C'$ that depend on $t_1$ but not on $n$. The desired bound now follows by taking the $n\to\infty$ limit: since $\partial_x v_n(x,t) \to \partial_x v(x,t)$ along a subsequence (in fact, along the entire sequence, since the limit is unique), this implies that $|\partial_x v(x,t)|\le C' x^{d-1}$ holds for $(x,t) \in (0,b) \times[t_1,\infty)$.  This completes the proof.
\end{proof}

We use~\eqref{dxv<ct} in Proposition~\ref{prop:v is C1} to prove the following result.
\begin{prop} \label{prop:vcont}
$(x,t)\mapsto v(x,t)$ is jointly continuous for $x \geq 0$ and $t > 0$.
\end{prop}
\begin{proof}
Fix $\epsilon > 0$. By \eqref{dxv<ct} and the triangle inequality, there is a constant $C(\epsilon) = C (1 + \epsilon ^{-1/2})$ such that whenever $x_0,x_1\ge 0$ and $\epsilon \le t_0 \le t_1$,
\[
\big|v(x_1,t_1)-v(x_0,t_0)| \le C(\epsilon) |x_1-x_0| + \big|v(x_0,t_1)-v(x_0,t_0)|.
\]
To estimate the second term on the right hand side, first notice
using~\eqref{Bound v} in Lemma~\ref{lem:basic prop v} that:
\begin{align*}
0\le v(x_0,t_1) - \int_0^\infty\diffd y\, G(y,x_0,t_1-t_0)v(y,t_0)&\le
(e^{t_1-t_0}-1)\int_0^\infty\diffd y\, G(y,x_0,t_1-t_0)v(y,t_0)\\&\le e^{t_1-t_0}-1
\end{align*}
by \eqref{Grint1} and since $v\in[0,1]$.  Then
\begin{align*}
\big|v(x_0,t_1)-v(x_0,t_0)| & \le  \bigg|\int_0^\infty\diffd y\, G(y,x_0,t_1-t_0) v(y,t_0)
- v(x_0,t_0)\bigg| + \Big|e^{t_1-t_0}-1\Big|.
\end{align*}
We bound \big|$v(y,t_0)
- v(x_0,t_0)\big|$ using \eqref{dxv<ct} again, and also note that it is smaller
than 1. Then, using the fact that $0 \leq v(x_0,t_0) \leq 1$, we have
\begin{align*}
\big|v(x_0,t_1)-v(x_0,t_0)| & \le  \int_0^\infty\diffd y\,
G(y,x_0,t_1-t_0) \min\big( 1, C(\epsilon) | y - x_0| \big) \notag  \\
&  \qquad + \bigg|1 -  \int_0^\infty\diffd y\, G(y,x_0,t_1-t_0)\bigg|+ \Big|e^{t_1-t_0}-1\Big|.
\end{align*}
It follows that for any $x_0,x_1\ge 0$ and $t_0,t_1\ge \epsilon$,
\begin{align*}
\big|v(x_1,t_1)-v(x_0,t_0)| & \le C(\epsilon) |x_1-x_0|
+ \int_0^\infty\diffd y\, G(y,x_0,|t_1-t_0|) \min\big( 1, C(\epsilon)
| y - x_0| \big) \notag \\
&  \quad + \bigg|1 -  \int_0^\infty\diffd y\, G(y,x_0,|t_1-t_0|)\bigg|+
\Big|e^{|t_1-t_0|}- 1\Big|. 
\end{align*}
Because the function $y \mapsto \min\big( 1, C(\epsilon) | y - x_0| \big)$
is continuous and bounded on $[0,\infty)$, the right hand side converges to zero
as $(x_1,t_1) \to (x_0,t_0)$, by \eqref{GintConvv} in Lemma~\ref{bounds int g}. 
\end{proof}

We can now complete the proof that $v$ is a classical solution to the obstacle problem~\eqref{pbv} by proving the following result.

\begin{prop}  \label{prop:vsmoothpde}
Suppose that $x_0 > 0$, $t_0 > 0$, and $v(x_0,t_0)<1$. Then in a neighbourhood of $(x_0,t_0)$, $v(x,t)$ is smooth and satisfies
\begin{equation}
\label{evolagain}
\partial_t v = \partial_x^2 v -\frac{d-1}x \partial_x v+v.
\end{equation}
\end{prop}
\begin{proof}
Let $(x_0,t_0)$ be as in the statement.
Choose $\epsilon \in \big(0,1-v(x_0,t_0)\big)$, and then $a$, $b$, $t_1$, $t_2$ such that $0<a<x_0<b$ and $0< t_1 < t_0 < t_2$, and such that
$v(x,t)<1 - \epsilon$ if $(x,t)$ lies in the rectangle $\mathcal{R} = (a,b)\times(t_1,t_2]$. This is possible by continuity of $v$.  Since $v$ is continuous on the parabolic boundary of $\mathcal{R}$, there is a unique function $\tilde v \in C(\overline{\mathcal{R}}) \cap C^{2,1}(\mathcal{R})$ which satisfies \eqref{evolagain} in $\mathcal{R}$ and is equal to $v$ on the parabolic boundary of $\mathcal{R}$.

Recall the definition of $v_n$ in~\eqref{eqvn}.
Applying the Feynman-Kac formula in Proposition~\ref{FK} to each function $v_n$ in the rectangle $\mathcal R$, we obtain
\[
v_n(x,t)=\E_{x}\Big[v_n(X_\tau,t-\tau)e^{\tau-\int_0^\tau\diffd s\,
v_n(X_s,t-s)^{n-1}}\Big], \quad \forall \;\;(x,t) \in \mathcal{R},
\]
where 
$X$ solves the SDE $\diffd X_s=\diffd W_s-\frac{d-1}{X_s}\,\diffd s$ with $X_0=x$, and
$\tau = \inf \{ s > 0\;:\; (X_s ,t-s) \notin \mathcal{R} \}$.  Now take the limit $n\to\infty$. Because $v_n^{n-1}\leq v^{n-1}\le (1-\epsilon)^{n-1}\to0$ uniformly on $\overline{\mathcal{R}}$, and $\tau \le t$, we get by dominated convergence
\begin{equation*}
v(x,t)=\E_{x}\big[v(X_\tau,t-\tau)e^{\tau}\big], \quad \forall \;\;(x,t) \in \mathcal{R}. 
\end{equation*}
By Proposition~\ref{FK}, the function $\tilde v$ defined above is also given by
\begin{equation*}
\tilde v(x,t)=\E_{x}\big[v(X_\tau,t-\tau)e^{\tau}\big], \quad \forall \;\;(x,t) \in \mathcal{R}. 
\end{equation*}
Hence $v = \tilde v$ in $\mathcal{R}$, which proves that $v \in C^{2,1}(\mathcal{R})$ and $v$ solves \eqref{evolagain}.  Since the coefficients in the parabolic equation \eqref{evolagain} are smooth for $x > 0$, standard regularity estimates (e.g. \cite{Evans2010} Theorem 7.1.7) imply that $v$ is infinitely differentiable in a neighbourhood of $(x_0,t_0)$.
\end{proof}

Having established existence and uniqueness of solutions to \eqref{pbv}, we now show that $v$ is continuous with respect to the initial condition.

\begin{lem}\label{lem:cont init cond}
Let $v$ and $\tilde v$ be the solutions to \eqref{pbv} corresponding to
the initial conditions $v_0$ and $\tilde v_0$.
Then
for $t>0$,
\[
\|v(\cdot,t) - \tilde v(\cdot,t)\|_{L^\infty}  \leq e^t \|v_0 - \tilde v_0 \|_{L^\infty} 
\]
and
\[
\|v(\cdot,t) - \tilde v(\cdot,t)\|_{L^1}  \leq e^t \|v_0 - \tilde v_0 \|_{L^1}. 
\]
\end{lem}

\begin{proof}
Let $v_n$ and $\tilde v_n$ be the solutions to \eqref{eqvn}
with initial conditions $v_0$ and $\tilde v_0$ respectively.  Having established uniqueness of solutions to \eqref{pbv} in Section~\ref{sec:uniq}, we know that $v_n$ and $\tilde v_n$ converge pointwise to $v$ and $\tilde v$ respectively as $n \to \infty$. Let
$h_n:=v_n-\tilde v_n$; then $h_n$ solves
\begin{equation} \label{eq:hndef}
\partial_t h_n = \partial_x^2 h_n -\frac{d-1}x\partial_x h_n
+ h_n - \alpha_n h_n,
\quad h_n(0,t)=0,
\quad\text{where }\alpha_n=\sum_{p=1}^{n}\tilde v_n^{n-p}\ v_n^{p-1}\ge0,
\end{equation}
with initial condition $h_0=v_0-\tilde v_0$.
Since this equation is linear,
we can write $h_n(x,t)=h^+_n(x,t)-h^-_n(x,t)$ where $h^+_n$ is the
solution of~\eqref{eq:hndef} with initial condition $h_0^+:=\max(h_0,0)$ and 
$h^-_n$ is the solution of~\eqref{eq:hndef} with initial condition $h^-_0:=\max(-h_0,0)$.
Recall the definition of $G_t$ in~\eqref{eq:GCdef}.
By the comparison principle, for $x,t>0$,
$$0\le h^+_n(x,t)\le e^t G_t h^+_0(x)\qquad \text{and}\qquad
  0\le h^-_n(x,t)\le e^t G_t h^-_0(x).$$
Therefore
$$\big| h_n(x,t)\big|
\le e^t \max\big(G_th_0^+(x),G_th^-_0(x)\big)
\le   e^t  G_t |h_0|(x).$$
Recall that $h_n=v_n-\tilde v_n$ and $h_0=
v_0-\tilde v_0$.
Taking the limit $n\to\infty$ implies that
$$
|v(x,t)-\tilde v(x,t)|\le e^t G_t |v_0-\tilde v_0|(x).
$$
The bound on $\|v(\cdot,t) - \tilde v(\cdot,t)\|_{L^\infty}$ now follows by~\eqref{eq:Gv0bound} in Lemma~\ref{bounds int g}; the bound on $\|v(\cdot,t) - \tilde v(\cdot,t)\|_{L^1}$ follows from \eqref{Grint1} in Lemma~\ref{bounds int g}.
\end{proof}

We now prove a property of $v$ which we will use in the following subsection.

\begin{lem}\label{lem:vxpos}
If $v_0:[0,\infty)\to [0,1]$ is not identically zero, then $v(x,t) > 0$ for all $x > 0$, $t > 0$. If $v_0$ is also non-decreasing, then $\partial_x v(x,t) > 0$ for all $x > 0$, $t > 0$ such that $v(x,t) < 1$. 
\end{lem}
\begin{proof}
The first statement is a consequence of the comparison in Lemma \ref{bound below}.  Specifically, for any $t_0 > 0$, we may define $v^\ell_0(x) = e^{-t_0} v_0(x)$, and let $v^\ell(x,t)$ be given by \eqref{eq:vlformula}, which solves the linear equation~\eqref{linear equ}. Then by~\eqref{eq:Gv0bound} in Lemma~\ref{bounds int g}, $v^\ell(x,t) \leq 1$ for $x \geq 0$ and $t \in [0,t_0]$.  So, by Lemma \ref{bound below} we have $v(x,t_0) \geq v^\ell(x,t_0)$ for all $x \geq 0$.   But $v^\ell(x,t_0) > 0$ for all $x > 0$, so $v$ must also be positive.

Now suppose $v_0$ is non-decreasing.
By Lemma \ref{lem:basic prop v}, we know that $\partial_x v \geq 0$.  Applying Proposition \ref{prop:vsmoothpde} in the region where $v < 1$, we see that the function $w = \partial_x v$ is a non-negative solution to $\partial_t w = L w + (1 + \frac{d-1}{x^2}) w$, where $L$ is the differential operator defined in~\eqref{Ldef}.  The second statement is a consequence of the strong maximum principle applied to $w$, as follows.  Let $\Omega = \{ (x,t) \;:\; x > 0,\; t > 0, \; v(x,t) < 1 \}$.  Suppose $(x_0,t_0) \in \Omega$ and $w(x_0,t_0) = 0$.  Because $v$ is continuous, we may choose $r > 0$ sufficiently small, so that $\overline{Q_r^-(x_0,t_0)} \subset \Omega$, where $Q_r^-(x_0,t_0)$ is the backward parabolic cylinder defined in~\eqref{Qcylinder}. 
By the strong maximum principle (e.g.\@ Theorem 7.1.11 of \cite{Evans2010}), $w$ must be constant on $Q_r^-(x_0,t_0)$, since it attains its minimum at $(x_0,t_0)$. In particular, $w(x,t_0) = 0$ for all $x \in (x_0 - r, x_0 +r)$. 
Let $R_{t_0}=\inf\{x:v(x,t_0)=1\}>0$ by continuity of $v$ (note that we may have $R_{t_0}=\infty$).
We now have that
the subset of $(0,R_{t_0})$ on which $w(x,t_0) = 0$ must be open. Since $x \mapsto w(x,t_0)$ is continuous, this subset must also be relatively closed in $(0,R_{t_0})$. We conclude that $w(x,t_0) = 0$ for all $x \in (0,R_{t_0})$. This is impossible however, since $v(0,t_0) = 0$ and $v(x,t_0)$ is positive for $x > 0$.
\end{proof}

To complete the proof of Theorem~\ref{thm:exists v}, it only remains to prove the fourth itemized point in the statement of the theorem, which we will prove in the next subsection.

\subsection{Properties of \texorpdfstring{$R_t$}{Rt}} \label{sec:Rprops}

In this subsection, we only consider initial conditions $v_0$ which are non-decreasing and such that $v_0(+\infty)=1$, and we study the properties of the boundary $R_t:=\inf\{x: v(x,t)=1\}$.

\begin{prop} \label{prop:Rtcontinuous}
Let $v_0$ be non-decreasing and such that $v_0(+\infty) = 1$. Then $R_t:=\inf\{x: v(x,t)=1\}$  is finite and strictly positive for all $t>0$, and $t\mapsto R_t$ is
continuous for $t > 0$. Moreover, $\lim_{t \searrow 0} R_t = R_0:=\inf \{x: v_0(x) = 1\}\in[0,\infty]$,
and for any $\alpha<1/2$, there exists $C_\alpha <\infty$ such that $R_t-R_s\le C_\alpha (t-s)^\alpha$ for all $s\ge 0$ and $t\in (s,s+1]$.

Let $v^-(x,t)$ and $v^+(x,t)$ be solutions of \eqref{pbv} with respective initial conditions $v_0^-(x)$ and $v_0^+(x)$ as described above, where $v_0^-\le v_0^+$. Then for $x\ge 0$ and $t>0$, $v^-(x,t)\le v^+(x,t)$, and the associated free boundaries $R^-$ and $R^+$ satisfy $R^+_t \leq R^-_t$ for all $t > 0$. 
\end{prop}

\begin{proof}
These statements about $R_t$ follow from arguments similar to those in the proof of Proposition 1.3 of \cite{BBP} (which concerns the free boundary in a related obstacle problem, but one with symmetries which do not hold here in the case $d\neq 1$).  

Note that by continuity of $v$, for $t>0$, $R_t$ is well-defined (although possibly infinite) and $R_t>0$.
By Lemma~\ref{lem:basic prop v}, since $v_0$ is non-decreasing, $v(\cdot,t)$ is also non-decreasing for $t>0$.
We now argue that for any $t_0 > 0$,
\begin{equation}
\liminf_{t \to t_0} R_t \geq R_{t_0}. \label{Rlsc}
\end{equation}
If this were not the case, then there would exist a finite $\bar R \in [0, R_{t_0})$ and a sequence of times $t_n$ such that $t_n \to t_0$ as $n \to \infty$, $R_{t_n}<\infty$ for each $n$, and
\[
\lim_{n \to \infty} R_{t_n} = \bar R.
\]
By the continuity of $v$, we have $v(R_{t_n},t_n) \to v(\bar R, t_0)$ as $n \to \infty$. Since $v(R_{t_n}, t_n) = 1$, this implies $v(\bar R,t_0) = 1$. Because this contradicts the fact that $v(x,t_0) < 1$ for all $x < R_{t_0}$, by definition, we conclude that \eqref{Rlsc} must hold.

Next, we prove that for any $\alpha \in (0,1/2)$, there is $C > 0$ such that for $s\ge 0$, whenever $R_s\in [0,\infty)$,
\begin{equation}
R_{s + \epsilon} \leq R_s + C \epsilon^\alpha, \quad \forall \;\;\epsilon \in (0,1]. \label{Rholder}
\end{equation}
Without loss of generality, we may suppose that $s = 0$ and $R_0 \in [0,\infty)$.  

The proof of this bound will rely on comparison of $v$ with a subsolution to the linear equation.
We will show that for all $\alpha\in (0,1/2)$, and all $\epsilon>0$ smaller than some $\epsilon^*=\epsilon^*(\alpha)\le 1$, there exists a function $h_\epsilon(x,t)$ satisfying $h_\epsilon(x,t) \leq v(x,t)$ for $t \in [0,   \epsilon]$ and $h_\epsilon(R_0 + z,\epsilon) = 1$ for some $z\in[0,(d+1)\epsilon^\alpha]$. This implies that $R_\epsilon \le R_0+(d+1)\epsilon^\alpha$ if $\epsilon\le\epsilon^*$. 
Hence for $\epsilon\in (0,1]$ and $k\in \N$ with $\epsilon/k \le \epsilon^*$, we have
$R_\epsilon \le R_0 +k(d+1)(\epsilon/k)^\alpha$, and we
obtain~\eqref{Rholder} with $C=(d+1)(\lceil1/\epsilon^*\rceil)^{1-\alpha}$.

Take $\epsilon\in (0,1]$ and introduce  $b=R_0+2\epsilon^\alpha$ and  $c_b = (d-1)/(R_0+ \epsilon^{\alpha})$.
 Let $h_\epsilon(x,t)$ solve the linear problem
\begin{equation*} 
\begin{cases}
\partial_t h_\epsilon  = \partial_x^2 h_\epsilon - c_b\partial_x h_\epsilon + h_\epsilon, \quad & t > 0,\;\;\;  b + c_b t - \epsilon^{\alpha} < x < b + c_b t + \epsilon^{\alpha}, \\
h_\epsilon(b + c_b t - \epsilon^{\alpha}, t) =  0 = h_\epsilon(b+ c_b t + \epsilon^{\alpha},t), \quad & t > 0, \\
h_\epsilon(x,0) = \eta_\epsilon,  \quad  & b - \epsilon^{\alpha} < x < b + \epsilon^{\alpha},
\end{cases}
\end{equation*}
with the constant $\eta_\epsilon \in (0,1)$ to be determined. 
We let $h_\epsilon(x,t)=0$ if $x\not\in(b+c_bt-\epsilon^\alpha,b+c_bt+\epsilon^\alpha)$.
Observe that the function 
$g_\epsilon(y,s)= e^{-\epsilon s} h_\epsilon( \epsilon^{1/2} y + b + c_b \epsilon s, \epsilon s)/\eta_\epsilon$ satisfies
\begin{equation*} 
\begin{cases}
\partial_s g_\epsilon  = \partial_y^2 g_\epsilon   , \quad  & |y| < \epsilon^{-(\frac{1}{2} - \alpha)} , \;\;s > 0, \\
g_\epsilon\Big(\pm \epsilon^{-(\frac{1}{2} - \alpha)},s\Big)= 0, \quad & s > 0,  \\
g_\epsilon(y,0) = 1,  \quad & |y| < \epsilon^{-(\frac{1}{2} - \alpha)} . 
\end{cases}
\end{equation*}
 By symmetry, $y \mapsto g_\epsilon(y,s)$ attains a unique maximum at $y = 0$. We now show that, if $\epsilon$ is small enough, then $s\mapsto e^{\epsilon s}g_\epsilon(0,s)$ is increasing for $s\in[0,1]$.

Recall that $\frac{1}{2} - \alpha > 0$.
Observe that $g_\epsilon(y,s) = \P_{\!y}\big( |W_r| < \epsilon^{-(\frac{1}{2} - \alpha)}, \; \forall\; r \in[0,s] \big)$, where $\P_y$ is the probability measure under which $(W_r)_{r\ge 0}$ is a Brownian motion in $\R$ (with diffusivity $\sqrt 2$) and $W_0=y$. 
Then for $|y|<\epsilon^{-(\frac 12 -\alpha)}$,
$$\begin{aligned}
1-g_\epsilon(y,s)
&=\P_{\!y}\Big(\exists r\in[0,s]\;:\; |W_r| = \epsilon^{-(\frac{1}{2} - \alpha)}\Big)\\
&
\le 2 \P_{|y|}\Big(\exists r\in[0,s]\;:\; W_r = \epsilon^{-(\frac{1}{2} -\alpha)}\Big)
= 4\P_{|y|}\Big(W_s \ge \epsilon^{-(\frac{1}{2} -\alpha)}\Big),
\end{aligned}
$$
where the last equality comes from the reflection principle. Then, for $\epsilon$ small enough,
\begin{equation}
\max_{|y| < 1} \big(1-g_\epsilon(y,s) \big) \leq 
2\erfc\Big[\tfrac1{2\sqrt s} \Big(\epsilon^{-(\frac12-\alpha)}-1\Big)\Big]
\le 2e^{- \frac1{4s} \big(\epsilon^{-(\frac12 - \alpha)}-1\big)^2}
\le 2e^{- \frac1{5s} \epsilon^{-(1 - 2\alpha)}}
. \label{gepsbound}
\end{equation}
We thus have that $\max_{s\in[0,1]}\big[1-g_\epsilon(0,s)\big]=o(\epsilon)$, and in particular $e^\epsilon g_\epsilon(0,1)>1$ if $\epsilon$ is small enough.
Because of \eqref{gepsbound} and regularity of solutions to the heat equation, $\max_{s \in [0,1]} \big|\partial_s g_\epsilon(0,s)\big| =o(\epsilon)$ holds, as well. Therefore, $\partial_s [e^{\epsilon s} g_\epsilon(0,s)] = e^{\epsilon s} [\epsilon g_\epsilon(0,s) + \partial_s g_\epsilon(0,s)] > 0$ is strictly positive for $s \in [0,1]$ if $\epsilon$ is small enough, and therefore  $e^{\epsilon s} g_\epsilon(y,s)< e^\epsilon g_\epsilon(0,1)$ for all $y$ and all $s<1$, with $e^\epsilon g_\epsilon(0,1)>1$. Then 
if we choose $\eta_\epsilon$ according to
\[
\eta_\epsilon = \frac{1}{ g_\epsilon(0,1)e^{\epsilon}} < 1, 
\]
the function $h_\epsilon$ satisfies
\[
\sup_{t \in [0,\epsilon],\ x\in\R} h_\epsilon(x,t) = h_\epsilon(b + c_b \epsilon,\epsilon) = \eta_\epsilon e^{\epsilon }g_\epsilon (0,1)=1.
\]
Observe that $b+c_b\epsilon=R_0+2\epsilon^\alpha +(d-1)\epsilon/(R_0+\epsilon^\alpha)
\le R_0+(d+1)\epsilon^\alpha$ since $\alpha<\frac12$. To conclude the proof of \eqref{Rholder}, it only remains to show that
$h_\epsilon(x,t)\le v(x,t)$ for all $t\le\epsilon$. This statement is proved as in the proof of Lemma~\ref{bound below}, by a maximum principle argument.
Recall that the function $h_\epsilon$ satisfies $h_\epsilon(x,t)\le1$ for $t\le\epsilon$, and note that the choice of $c_b$ was made to guarantee that
\begin{equation}
\partial_t h_\epsilon  \leq \partial_x^2 h_\epsilon - \frac{d-1}{x} \partial_x h_\epsilon + h_\epsilon \label{hsubineq}
\end{equation}
holds for $x \in [b + c_b t - \epsilon^{\alpha}, b + c_b t]$ and $t > 0$
(i.e.~the region where $\partial_x h_\epsilon \geq 0$). 

Introduce $\phi(x,t)=e^{-2t}[v(x,t)-h_\epsilon(x,t)]$, and $M=\inf_{t\le\epsilon,\, x>0} \phi(x,t)$. We will show that $M\ge0$, which implies \eqref{Rholder}. Let $(x_n,t_n)$ be a sequence with $x_n>0$ and $t_n \leq \epsilon$ for each $n$, and such that $\phi(x_n,t_n)\to M$ as $n\to\infty$. By taking a subsequence, we may assume that either $x_n\to\infty$ or $(x_n,t_n)\to(x^*,t^*)$ as $n\to\infty$ for some $x^*\ge 0$ and $t^*\le \epsilon$.
If $x_n\to \infty$, then $M\ge0$ because $h_\epsilon(x,t)=0$ for $x\ge b+c_b\epsilon +\epsilon^\alpha$, $t\le \epsilon$. We now assume $(x_n,t_n)\to(x^*,t^*)$.
If $t^*=0$, then $M\ge0$. Indeed, the fact that $x \mapsto v(x,t)$ is non-decreasing and $v_0(x) = 1$ on $(R_0,\infty)$ implies that $v(x,t) \to 1$ uniformly on $[R_0+\frac12\epsilon^\alpha,\infty)$ as $t\to0$.
Then if $x^*> R_0+\frac12\epsilon^\alpha$, we conclude by noticing that $h_\epsilon(x,t)\le\eta_\epsilon e^t\le\frac{1+\eta_\epsilon}2$ for $t$ small enough, and if $x^*\le  R_0+\frac12\epsilon^\alpha$ we conclude by noticing that $h_\epsilon(x_n,t_n)=0$ for $n$ large enough.
We now consider the case $t^*>0$. By continuity, we have that $M=\phi(x^*,t^*)$. If $x^*\not\in (b+c_bt^*-\epsilon^\alpha, b+ c_b t^*+\epsilon^\alpha)$, then $h_\epsilon(x^*,t^*)=0$ and $M\ge0$. If $v(x^*,t^*)=1$, then $M\ge0$ (since $h_\epsilon\le1$). In the remaining cases, $h_\epsilon$, $v$ (and hence $\phi$) are $C^{2,1}$ in a neighbourhood of $(x^*,t^*)$, and therefore we must have $\partial_t \phi(x^*,t^*)\le0$, $\partial_x\phi(x^*,t^*)=0$ and $\partial_x^2\phi(x^*,t^*)\ge0$.
Recalling that $\partial_x h_\epsilon(x,t) <0$ if $x\in(b+c_bt,b+c_bt+\epsilon^\alpha)$, and that $\partial_x v\ge0$ on $(0,\infty)$, we see that $\partial_x\phi(x^*,t^*)=0$ cannot hold unless $x^*\le b+c_bt^*$. In the remaining region, \eqref{hsubineq} holds and $\partial_t v =\partial_x^2 v-\frac{d-1}x\partial_x v + v$, so that $\partial_t\phi \ge \partial_x^2 \phi-\frac{d-1}x\partial_x \phi - \phi$, which implies $M=\phi(x^*,t^*)\ge0$.  In all the cases, we found $M\ge0$, and so we have now established~\eqref{Rholder}.

The fact that $R_t < \infty$ for all $t > 0$ follows by a similar comparison with subsolutions to the linear problem. By our assumptions about $v_0$, we know that for any $\epsilon \in (0,1)$ there is $n_\epsilon > 0$ such that $v_0(x) > 1- \frac 12 \epsilon$ if $x \geq n_\epsilon$.  Now define $c = (d-1)/n_\epsilon \ge 0$. For $L > \pi$ fixed, consider the function 
\[
\psi_\epsilon (x,t) = (1 - \epsilon) e^{ t \big(1 - (\pi/L)^2\big)} \sin \left( \frac{\pi (x - n_\epsilon - ct)}{L} \right)
\]
for $x\in (n_\epsilon+ct,n_\epsilon+ct+L)$,
which satisfies the linear problem
\begin{equation*}
\begin{cases} 
\partial_t \psi_\epsilon  = \partial_x^2 \psi_\epsilon - c \partial_x \psi_\epsilon + \psi_\epsilon, \quad & t > 0,\;\;\;  n_\epsilon + c t  < x < n_\epsilon + c t + L ,\\
\psi_\epsilon(n_\epsilon + c t, t) =  0 = \psi_\epsilon(n_\epsilon + c t + L, t), \quad & t > 0,\\
\psi_\epsilon(x,0) = (1 - \epsilon) \sin\left( \frac{\pi (x - n_\epsilon)}{L} \right) < v_0(x),  \quad  & n_\epsilon  < x < n_\epsilon + L.
\end{cases}
\end{equation*}
We let $\psi_\epsilon(x,t)=0$ if $x\notin (n_\epsilon+ct,n_\epsilon+ct+L)$.
Notice that for each $t$, $x \mapsto \psi_\epsilon(x,t)$ attains its maximum at $x = n_\epsilon + ct + L/2$.  Also, $c \ge 0$ is chosen so that
\[
\partial_t \psi_\epsilon  \le \partial_x^2 \psi_\epsilon - \frac{d-1}{x} \partial_x \psi_\epsilon + \psi_\epsilon
\]
holds if $x \in [n_\epsilon + ct , n_\epsilon + ct + L/2]$ and $t>0$ (i.e.~where $\partial_x \psi_\epsilon(x,t) \ge 0$). 
Let $t_1=|\log(1 - \epsilon)|/(1 - (\pi/L)^2)$ (remember that $L>\pi$); notice that $\psi_\epsilon(n_\epsilon+ct_1+L/2,t_1)=1$ and that $\psi_\epsilon (\cdot,t)\le 1$ for all $t\le t_1$.
Then by the same argument as above with $h_\epsilon$, the comparison $\psi_\epsilon \leq v$ holds over $\{ (x,t) \;:\; n_\epsilon + ct < x < n_\epsilon + ct + L, \,\,\,t \le t_1 \}$.  This implies that $v(n_\epsilon+ct_1+L/2,t_1)=1$ and thus $R_{t_1}<\infty$. But any value of $t_1 > 0$ can be obtained by taking $\epsilon\in(0,1)$, and so we conclude that $R_t <\infty$ for all $t > 0$.

Since we now have that $R_t<\infty$ for all $t>0$, the estimate \eqref{Rholder} implies that for $t_0>0$,
\[
\limsup_{t \searrow t_0 } R_{t} \leq R_{t_0}.
\]
Therefore, to prove that $R_t$ is continuous at $t_0 > 0$, it remains to prove that
\begin{equation}
\bar R := \limsup_{t \nearrow t_0} R_t \leq R_{t_0}. \label{contin2}
\end{equation}
Note that by~\eqref{Rholder} we have for $t\in[(t_0-1)\vee (t_0/ 2),t_0]$
that $R_t\le R_{(t_0-1)\vee (t_0/2)}+C<\infty$, and so $\bar R<\infty$.
By the definition of $\bar R$, for any $\epsilon>0$ and any $t<t_0$ one can find $t_1\in(t,t_0)$
such that $R_{t_1}\ge\bar R-\epsilon$. But 
by~\eqref{Rholder}, assuming $t\ge t_0-1$, we know that $R_t\ge R_{t_1} - C(t_1-t)^\alpha\ge\bar
R-\epsilon - C(t_0-t)^\alpha$, and we conclude that
\[
\liminf_{t \nearrow t_0} R_t \ge \bar R.
\]
That is, the limit exists: $\lim_{t \nearrow t_0} R_t = \bar R$.

Arguing by contradiction, suppose that $\bar R > R_{t_0}$, and
 let $b = (\bar R + R_{t_0})/2 \in (R_{t_0}, \bar R)$. Then there is $\epsilon> 0$ small enough so that $b < R_t$ for all $t \in [t_0 - \epsilon, t_0)$. Hence, $v(x,t) < 1$ for all $(x,t) \in [0,b] \times [t_0-\epsilon, t_0)$, although $v(x,t_0) = 1$ for all $x \in [R_{t_0}, b]$.  By Lemma \ref{lem:vxpos}, the function $w(x,t) = \partial_x v(x,t)$ is positive in the region where $v(x,t) < 1$ and $x > 0$ and $t > 0$.   In particular, $\inf_{x\in [R_{t_0}/2,b]}w(x,t_0 - \epsilon)> 0$.  From this it follows that, for any fixed $c \in (R_{t_0}, b)$ there is $\delta > 0$ such that
\begin{equation} \label{eq:wlower}
\inf_{\substack{x \in [R_{t_0}, c] \\ t \in [t_0 - \epsilon, t_0)}} w(x,t) > \delta > 0.
\end{equation}
(For example, 
since $\partial_t w=\partial^2_x w-\frac{d-1}x \partial_x w+(1-\frac{d-1}{x^2})w$ in the region
$\{(x,t):v(x,t)<1,x>0,\, t>0\}$, we can apply the Feynman-Kac formula in Proposition~\ref{FK}.
Then~\eqref{eq:wlower}
follows easily from the Feynman-Kac representation for $w$ in the rectangle $[R_{t_0}/2, b] \times [t_0 - \epsilon, t_0)$ by showing that a backward path $(X_s,t-s)$ started from any point $(x,t)\in [R_{t_0},c]\times[t_0-\epsilon,t_0)$ has a probability larger than some $\eta>0$ of reaching time $t_0-\epsilon$ without first touching the boundaries $R_{t_0}/2$ or $b$.) 
Since $w = \partial_x v$, the lower bound in~\eqref{eq:wlower} shows that for any $t \in (t_0 - \epsilon,t_0)$ we have
\[
v(c,t) - v(R_{t_0},t) = \int_{R_{t_0}}^c w(x,t) \,\diffd x \geq \delta (c - R_{t_0}) .
\]
By continuity of $v$, we let $t \nearrow t_0$ and conclude that $v(c,t_0) - v(R_{t_0}, t_0) \geq \delta  (c - R_{t_0}) > 0$, which is a contradiction, since $v(x,t_0) = 1$ for all $x \geq R_{t_0}$. This proves \eqref{contin2}, and completes the proof of continuity of $R_t$ for $t > 0$. 

Now we prove that
\[
\lim_{t \searrow 0} R_t = R_0 := \inf \{x: v_0(x) = 1\}.
\]
Recall that \eqref{Rholder} is valid for $s = 0$ if $R_0 < \infty$.  Thus, the only thing we need to show is
\begin{equation}
\liminf_{t \searrow 0} R_t \geq R_0. \label{R0lower}
\end{equation}
Suppose that \eqref{R0lower} does not hold; then there exist $M > 0$ and a decreasing sequence of times $\{t_n\}_{n=1}^\infty$ such that $t_n \to 0$ as $n \to \infty$, while $R_{t_n} \leq M < R_0$ for all $n$ (where $M$ is some finite number in the case that $R_0 = +\infty$).  Because $v$ is non-decreasing in $x$, this implies that $v(x,t_n) = 1$ for all $x \geq M$. Since $v(\cdot,t) \to v_0$ in $L^1_{\text{loc}}$ as $t\searrow 0$ and $v_0$ is non-decreasing, we conclude that $v_0(x) = 1$ for all $x > M$, which is a contradiction since $M < R_0$.  This proves that $R_t \to R_0$ as $t \to 0$.

For the last statement of the proposition, by Lemma~\ref{lem:basic prop v}
we have that $v^-(x,t) \leq v^+(x,t)$ for all $x \geq 0$ and $t > 0$, and the fact that 
$R^+_t \leq R^-_t$ follows immediately.
\end{proof}

Having completed the proof of Theorem~\ref{thm:exists v}, we now show that
$\partial_x v$ is jointly continuous in $(x,t)$ up to the moving boundary
in the case where $v_0$ is non-decreasing.
This result 
is not implied by Proposition~\ref{prop:vsmoothpde}, in which $v$ was shown to be smooth in the neighbourhood of a point where $v<1$, and
will be needed in the proof of Theorem~\ref{thm:exists u} in Section~\ref{sec:uprops}.

\begin{lem}\label{lem:jointcontin}
Let $v_0$ be non-decreasing and such that $v_0(+\infty) = 1$. Then $(x,t) \mapsto \partial_x v(x,t)$ is jointly continuous on $(0,\infty) \times (0,\infty)$. 
\end{lem}

\begin{proof}
In view of Proposition \ref{prop:vsmoothpde} and Proposition~\ref{prop:Rtcontinuous}, we only need to show that for any $t_0 > 0$,
\begin{align}
\lim_{(x,t) \to (R_{t_0},t_0)} \partial_x v(x,t) = 0. \label{vrcontin}
\end{align}
Note that $R_{t_0}\in (0,\infty)$ by Proposition~\ref{prop:Rtcontinuous}.
Because of \eqref{dvmodulus} in Proposition~\ref{prop:v is C1}, we know there is a constant $C <\infty$ such that
\begin{equation*}
\big|\partial_x v(x,t) - \partial_x v(x',t)\big| \le C|x-x'|^{1/2}
\end{equation*}
for all $x,x' \in [(1/2)R_{t_0}, (3/2) R_{t_0}]$, and $t \in [(1/2)t_0, (3/2) t_0]$. 
Therefore, since $\partial_x v(R_{t},t)= 0$ for $t>0$, we infer that if
$(x,t) \in [(1/2)R_{t_0}, (3/2) R_{t_0}] \times[(1/2)t_0, (3/2) t_0]$,
\begin{align*}
\big|\partial_x v(x,t)\big| \le C|x-R_t|^{1/2} \leq C|x-R_{t_0}|^{1/2} + C|R_{t_0}-     R_t|^{1/2}. 
\end{align*}
This last expression vanishes in the limit $(x,t) \to (R_{t_0},t_0)$ by
continuity of $R_t$ (from Proposition~\ref{prop:Rtcontinuous}).
\end{proof}

\section{Proof of Theorem \ref{thm:conv}: Convergence to the steady state for 
\texorpdfstring{$v$}{v}} \label{sec:vsteady}

Recall the definitions of $V$ and $R_\infty$ in~\eqref{eq:VJ} and~\eqref{eqU}.
We consider the solutions of the obstacle problem~\eqref{pbv} with two special initial conditions
\begin{equation} \label{eq:v0upperlower}
\bar v_0(x) = 1\quad \quad \text{and} \quad \quad \underline v_0(x)=c \indic{x >K} \quad \quad \forall \;x \geq 0,
\end{equation}
where $K>0$ and $c \in(0,1)$ are fixed with $c\le V(K)$.
Then since $V$ is non-decreasing and non-negative, we have that 
$\underline v_0(x)\le V(x)\le \bar v_0(x)$ for all $x\ge 0$.
We let $\bar
v$ and $\underline v$ denote the
solutions to \eqref{pbv} with initial conditions $\bar
v_0$ and $\underline v_0$ respectively. 
By Theorem~\ref{thm:exists v}, $\bar v(\cdot,t)$ and $\underline
v(\cdot,t)$ are both non-decreasing for all $t > 0$, and the
free boundary $\bar R_t := \inf \{ x\;:\; \bar v(x,t) = 1 \}$ is finite and continuous for $t > 0$.
Let $t^*=\inf\{t> 0:\lim_{x\to \infty}\underline v(x,t)=1\}$;
 the free boundary $\underline R_t := \inf \{ x\;:\; \underline v(x,t) = 1 \}$ is finite and continuous for $t>t^*$.  Again by Theorem~\ref{thm:exists v},
and since $V$ is a stationary solution of~\eqref{pbv}, we have
\begin{equation} \label{eq:vVvsandwich}
\underline v(x,t) \le V(x)\le \bar v(x,t)\qquad \text{and }\qquad \bar R_t \le R_\infty \le \underline R_t
\end{equation}
for all $x\ge 0$ and $t> 0$.
Let $I_{K,c}$ denote the set of initial conditions
that lie between $\underline v_0$ and $\bar v_0$:
$$ v_0 \in I_{K,c} \quad\Leftrightarrow\quad \underline
v_0(x)\le v_0(x)\le\bar v_0(x) \,\,\,\, \forall x\ge 0.$$ Then, 
letting $v$ denote the solution of~\eqref{pbv} with initial condition $v_0$, and letting
$R_t=\inf\{x:v(x,t)=1\}$,
by the comparison principle in Theorem~\ref{thm:exists v} it is clear that
\begin{equation} \label{eq:vsandwich}
 v_0\in I_{K,c} \quad\implies \quad
	\underline v(\cdot,t)\le v(\cdot,t)\le\bar v(\cdot,t)
\quad\text{and}\quad \bar R_t \le R_t \le \underline R_t\quad \text{for all }t>0.
\end{equation}
Since any non-zero, non-decreasing initial condition $v_0:[0,\infty)\to [0,1]$ 
satisfies $v_0 \in I_{K,c}$ for some $K>0$ and $c\in\big(0,V(K)\big]$,
 it
is therefore sufficient to consider the long term behaviour of $\bar v$
and $\underline v$. Specifically, we are going to prove the following result.
\begin{prop} \label{prop:conv v}
For $c\in(0,1)$ and $K\in (0,\infty)$ with $c\le V(K)$, there
exist $A>0$ (independent of $c$ and $K$), and $B=B(c,K)>0$ such that the following holds.
Let
$$\lambda=\frac{Z^2}{R_\infty^2}-1>0, \quad \text{ where $Z:=\inf\big\{x>0 : J_{\frac d2}(x)=0\big\}$,}
$$
i.e.~$Z$ is the first positive zero of the
Bessel function $J_{\frac d2}$.
Let $\underline v$ and $\bar v$ denote the solutions of~\eqref{pbv} with initial conditions $\underline{v}_0(x)=c\indic{x>K}$ and $\bar{v}_0(x)=1$ respectively, and let
$\underline{R}_t=\inf\{x:\underline{v}(x,t)=1\}$ and $\bar{R}_t=\inf\{x:\bar{v}(x,t)=1\}$ for $t>0$.
Then for $x\ge 0$ and $t>0$,
\[
V(x)-\frac{B} t\le \underline v(x,t)\le V(x)\le \bar v(x,t)\le V(x)+
A e^{-\lambda t}.
\]
For all $t> 0$, $R_\infty-Ae^{-\lambda t}\le \bar R_t \leq R_\infty \leq \underline R_t$, and for $t\ge B$, $\underline R_t \le R_\infty+\frac{B}t$.
\end{prop}
This result is proved in Lemmas~\ref{upbndv} and~\ref{upbndR} (for the bounds on
$\bar v$ and $\bar R$), 
and in Lemma~\ref{lowbndR} (for the bounds on $\underline v$
and $\underline R$).
Note that if $v$ is the solution to~\eqref{pbv} with initial condition $v_0:[0,\infty)\to [0,1]$, where $v_0$ is non-decreasing and $v_0(K)\ge c$, then $v_0\in I_{K,c}$, and so Theorem~\ref{thm:conv} follows directly from Proposition~\ref{prop:conv v} and~\eqref{eq:vsandwich}.

Recall the definition of $V$, $R_\infty$ and $\alpha$ from \eqref{eq:VJ}, and for $x\ge 0$, let
\begin{equation}
\label{defJt}
\tilde J(x) = \alpha x^\frac d2 J_{\frac d2}(x).
\end{equation}
Recall that $R_\infty$ is the position of the first positive local maximum of $\tilde J$, and that $\alpha$ is chosen in such a way that $\tilde J(R_\infty)=1$.
The function $\tilde J$ is a solution to
\begin{equation}\label{eqJt}
\tilde J'' -\frac{d-1}x\tilde J' +\tilde J=0, \quad \quad \text{for }x > 0,
\end{equation}
and $\tilde J(0) = 0$, and so, in particular, $\tilde J''(R_\infty)=-1$. 
Recall also that
$$V(x)=\begin{cases}\tilde J(x)& \text{for }0 \leq x < R_\infty,\\ 1 &\text{for }x\ge R_\infty. \end{cases}$$
Recall from the statement of Proposition~\ref{prop:conv v} that we
let $Z=\inf\{x>0:\tilde J(x)=0\}$ denote the position of the first positive zero of $\tilde J$.
Figure~\ref{figJV} shows the graph of $\tilde J$ and $V$ for $d=3$.
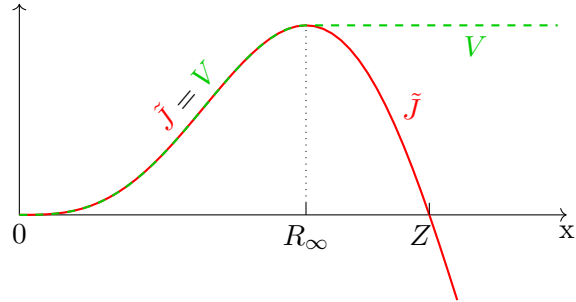
\begin{figure}[ht]\centering
\begin{tikzpicture}[xscale=1.2,yscale=.8]
\draw[<->] (0,3.5)--(0,0) node[below]{0} -- (6,0) node[below]{x};
\draw[red,thick] plot[domain=0:4.8,samples=50] 
(\x,{sin(\x*180/3.1416)-\x*cos(\x*180/3.1416)});
\draw[green!80!black,thick,dashed] plot[domain=0:3.1416,samples=50]
(\x,{sin(\x*180/3.1416)-\x*cos(\x*180/3.1416)})
	-- (5.9,3.1416);
\draw[red] (4.3,1.8) node{$\tilde J$};
\draw[green!80!black] (5,2.8) node{$V$};
\draw (1.8,2) node[rotate=50]{$\textcolor{red}{\tilde
J}=\textcolor{green!80!black}{V}$};

\draw (3.1416,0) node[below]{$R_\infty$} -- ++(0,.2);
\draw[dotted] (3.1416,.2) --(3.1416,3.1416);
\draw (4.4934,0) node[below,xshift=-.3em]{$Z$} -- ++(0,.2);
\end{tikzpicture}
\caption{The functions $\tilde J$ and $V$ for $d=3$}
\label{figJV}
\end{figure}

\subsection{Proof of Proposition~\ref{prop:conv v}: upper bound}
Our strategy to prove bounds on $\bar v$ and $\bar R$ is to compare $\bar v$ with the solution $\phi$ to the usual linear 
PDE
with the boundary condition $\phi(R_\infty,t)=1$.

\begin{lem}\label{upcmp}
Take 
$v_0:[0,\infty)\to [0,1]$ measurable.
Let $\phi_0:[0,\infty)\to [0,\infty)$ be a continuous function with $\phi_0(0)=0$, $\phi_0(R_\infty)=1$, and $\phi_0\ge v_0$ on $[0,R_\infty]$.
Suppose
$v$ is the solution to \eqref{pbv} with initial condition $v_0$, and $\phi$ is the solution to
\begin{equation}\label{eqw}
\begin{cases}
\partial_t \phi=\partial_x^2 \phi -\frac{d-1}x \partial_x \phi + \phi, \quad &\text{for
$t>0$, $x\in(0,R_\infty)$},\\
\phi(0,t)=0, \quad &\text{for }t>0,\\
\phi(x,t)=1,\quad &\text{for }t>0,\,\, x\ge R_\infty ,\\
\phi(x,0)=\phi_0(x), \quad &\text{for $x\in(0,R_\infty)$.}
\end{cases}
\end{equation}
Then for $t>0$ and $x\ge 0$,
$$\phi(x,t)\ge v(x,t).$$
\end{lem}
\begin{proof}
Recall the definition of $v_n$ in~\eqref{eqvn}.
In Section~\ref{sec:vexists}, we proved that the solution of~\eqref{pbv} is given by the pointwise limit
$v(x,t)=\lim_{n\to \infty}v_n(x,t)$.
Since $v_n\in [0,1]$ is a subsolution to \eqref{eqw}, it follows by the standard comparison principle for \eqref{eqw} that $v_n(x,t)\le \phi(x,t)$ for $t>0$ and $x\ge 0$.
Taking the limit $n\to \infty$ completes the proof.
\end{proof}

\begin{lem}\label{solw0}
Take $C>0$. Then
$$\phi(x,t)= 
\begin{cases} V(x)+C \tilde J\Big(\frac {Z x}{R_\infty}\Big)
e^{-\left(\frac{Z^2}{R_\infty^2}-1\right)t}\quad &\text{for $x\in[0,R_\infty)$,} \\
1 \quad &\text{for }x\ge R_\infty
\end{cases}
$$
solves \eqref{eqw}
with $\phi_0(x)= V(x)+C \tilde J\Big(\frac {Z x}{R_\infty}\Big)$ for $x\in [0,R_\infty]$.
\end{lem}
\begin{proof}
Note that $V(0)=0=\tilde J(0)$ and $\tilde J(Z)=0$, so the boundary conditions are satisfied.
For $x\in (0,R_\infty)$, by~\eqref{eqJt} we have 
$V''(x)-\frac{d-1}x V'(x) +V(x)=0$ and
$\tilde J''\left(\frac{Zx}{R_\infty}\right)-\frac{d-1}x\frac{R_\infty}Z  \tilde J'\left(\frac{Zx}{R_\infty}\right)+\tilde J\left(\frac{Zx}{R_\infty}\right)=0$.
The result follows by an elementary calculation.
\end{proof}
We now find an upper bound on $\bar v(\cdot,1)$ and use this to apply Lemmas~\ref{upcmp} and~\ref{solw0}.
\begin{lem}
\label{upbndv}
There exists a constant $C_1<\infty$ such that
for $t\ge 0$ and $x\in [0,R_\infty]$,
$$\bar v(x,t+1)\le {V(x)}+C_1 \tilde J\Big(\frac {Z 
x}{R_\infty}\Big)e^{-\left(\frac{Z^2}{R_\infty^2}-1\right)t}.$$
\end{lem}
\begin{proof}
By \eqref{defJt} and the series expansion around $x=0$ for the Bessel
function $J_{d/2}(x)$, there exists $c_1\in (0,\infty)$ such that $\tilde J(x)=V(x)\sim c_1 x^d$  as $x\searrow 0$.
Now let $\bar{v}^\ell $ denote the solution to the linear problem~\eqref{linear equ} with initial condition $v_0^\ell=\bar{v}_0=1$.
Then by Lemma~\ref{bound above}, $\bar v(x,1)\le \bar v^\ell (x,1)$ for all $x\ge 0$.
Using~\eqref{eq:vlformula}, it follows that for $x\ge 0$,
\begin{align*}
\bar v(x,1)&\le e^1 \int_0^\infty\diffd y \, G(y,x,1)\le C e^1 x^d,
\end{align*}
where the second inequality follows by~\eqref{eq:Gintbounds} in Lemma~\ref{bounds int g} and $C$ is a constant.
Note that $V(R_\infty)=1$, $V'(R_\infty)=0$ and, using~\eqref{eqJt}, $V''(R_\infty -)=-1$,
and so $1-V(R_\infty-\epsilon)\sim \frac 12 \epsilon^2$ as $\epsilon\searrow 0$.
Also, since $\bar R_t \le R_\infty$ by~\eqref{eq:vVvsandwich}, and since $\partial_x \bar v(x,t)=0$ for $x\ge \bar R_t$, we have $\bar v(R_\infty-\epsilon,1)=1+o(\epsilon)$ as $\epsilon \searrow 0$.
Since the zeros of the Bessel function $J_{\frac d2}$ are simple, 
there exists a constant $c_2\in (0,\infty)$ such that
$\tilde J(Z(R_\infty -\epsilon)/R_\infty)\sim c_2 \epsilon$  as $\epsilon\searrow 0$.
Therefore
$$
\sup_{x\in (0,R_\infty)}
\frac{\bar v(x,1)-V(x)}{\tilde J\left(\frac{Zx}{R_\infty}\right)}<\infty.
$$
It follows that there exists $C_1<\infty$ such that
$$\bar v(x,1)\le {V(x)}+C_1 \tilde J\Big(\frac {Z 
x}{R_\infty}\Big) \quad \text{ for all }x\in [0,R_\infty].$$
By Lemmas~\ref{upcmp} and~\ref{solw0}, the result follows.
\end{proof}
Note that $\sup_{x\in [0,R_\infty]}\tilde J(Zx/R_\infty)<\infty$ and, for $x\ge R_\infty$, $\bar v (x,t)\le 1=V(x)$. Therefore it follows from Lemma~\ref{upbndv} that there exists a constant $A<\infty$ such that for $t> 0$ and $x\ge 0$,
$$
\bar v(x,t)\le V(x)+A e^{-\left(\frac{Z^2}{R_\infty^2}-1\right)t}.
$$
We now need to prove a lower bound on $\bar R_t$.
\begin{lem} 
\label{upbndR}
There exists a constant $C_2<\infty$ such that for any $t> 0$,
$$\bar R_t  \ge R_\infty -C_2 e^{-\big(\frac{Z^2}{R_\infty^2}-1\big)t}.$$ 
\end{lem}
\begin{proof}
By the definitions of $V$ and $R_\infty$, and by~\eqref{eqJt},
$V(R_\infty)=1$, $V'(R_\infty)=0$ and $V''(R_\infty -)=-1$, so there exists
$\delta>0$ such that
$$V(R_\infty-\eta) < 1- \frac{0.99}2\eta^2\quad\forall\eta\in[0,\delta].$$ 
Since the zeros of the Bessel function $J_{\frac d2}$ are simple,
by making $\delta$ smaller if necessary we also have that for $\eta \in [0,\delta]$,
$$
\tilde J\left( \frac{Z(R_\infty-\eta)}{R_\infty}\right) \le
\frac{2Z}{R_\infty}\Big|\tilde J'(Z)\Big|\eta.
$$
Therefore by
Lemma~\ref{upbndv}, for $t\ge 0$,
$$
 \bar v(R_\infty-\eta,t+1)< 1+\frac{2C_1 Z}{R_\infty}\big|\tilde J'(Z)\big|e^{-
\left(\frac{Z^2}{R_\infty^2}-1\right)t}\eta-\frac{0.99}2\eta^2\quad\forall\eta\in[0,\delta].
$$
This implies that
$$\bar v (R_\infty-\eta,t+1)< 1 \qquad\text{if}\quad \frac{4C_1 Z}{0.99R_\infty}\big|\tilde 
J'(Z)\big|e^{-
\left(\frac{Z^2}{R_\infty^2}-1\right)t}\le \eta\le\delta.$$
In particular, for $t$ sufficiently large,
$$\bar R_{t+1}\ge R_\infty -\frac{4C_1 Z}{0.99R_\infty}\big|\tilde J'(Z)\big|e^{-
\left(\frac{Z^2}{R_\infty^2}-1\right)t},$$
which completes the proof.
\end{proof}

\subsection{Proof of Proposition~\ref{prop:conv v}: lower bound}
To prove bounds on $\underline v$ and $\underline R$, we
now compare $\underline v$ with the solution $\phi$ to the usual linear PDE
with the boundary condition $\phi(x_0,t)=0$, where $x_0$ is chosen to be close to $Z$.  Recall from~\eqref{eq:v0upperlower} that $\underline v$ depends on parameters $c \in (0,1)$ and $K > 0$. 

\begin{lem}\label{dncmp}
Take $x_0>0$ and $v_0:[0,\infty)\to [0,1]$ measurable. Let $\phi_0:[0,\infty)\to [0,1]$ be a continuous function with $\phi_0(0)=0$ and $\phi_0(x_0)=0$, and satisfying $\phi_0\le v_0$ on $[0,x_0]$.
Suppose
$v$ is the solution to \eqref{pbv} with initial condition $v_0$, and $\phi$ is the solution to
\begin{equation}\label{eqw2}
\begin{cases}
\partial_t \phi=\partial_x^2 \phi -\frac{d-1}x \partial_x \phi + \phi, \quad &\text{for
$t>0$, $x\in(0,x_0)$},\\
\phi(0,t)=0, \quad &\text{for }t>0,\\
\phi(x,t)=0,\quad &\text{for }t>0,\,\, x\ge x_0,\\
\phi(x,0)=\phi_0(x),  \quad &\text{for $x\in(0,x_0)$.}
\end{cases}
\end{equation}
Let $t_*=\inf\{s\geq 0:\sup_{x\in (0,x_0)}\phi(x,s)\geq 1\}$.
Then, for $t\le t_*$,
\begin{equation}
\phi(x,t)\le v(x,t), \quad \forall \;\; x \geq 0. \label{wvlower}
\end{equation}
\end{lem}

\begin{proof}
This is a consequence of Lemma \ref{bound below}. Let $t_1$ denote the time at which \eqref{wvlower} first fails:
\[
t_1 = \inf \{ t \geq 0\;:\;  \phi(x,t) > v(x,t) \;\;\text{for some $x \geq 0$} \}.
\]
The lemma is proved by showing that $t_1 \geq t_*$.
Suppose that $t_1<\infty$.
Let $v^\ell(x,t)$ be defined by~\eqref{eq:vlformula} with $v_0^\ell(x) = \phi(x,t_1)$ for $x \ge 0$. Thus, $v^\ell$ solves the linear problem~\eqref{linear equ}. By the comparison principle for this linear PDE, we have $\phi(x,t_1 + \epsilon) \leq v^\ell(x,\epsilon)$ for all $x \in (0,x_0)$ and $\epsilon \geq 0$.  By Lemma~\ref{bound below}, 
since $\phi(\cdot,t_1)\le v(\cdot,t_1)$,
we also have $v^\ell(x,\epsilon) \leq v(x,t_1 + \epsilon)$ for all $x \geq 0$, as long as $\sup_{x\ge 0} v^\ell(x,t) \leq 1$ for all $t\le \epsilon$. 
Suppose that $t_1<t_*$. Then $\sup_{x\ge 0} v^\ell_0(x) = \sup_{x\ge 0} \phi(x,t_1) < 1$, and
so by~\eqref{eq:Gintbounds} in Lemma~\ref{bounds int g},
for $\epsilon > 0$ sufficiently small, $\sup_{x\ge 0} v^\ell(x,t) < 1$ holds for all $t \in [0,\epsilon]$. Hence for $\epsilon>0$ sufficiently small, $\phi(x,t_1 + \epsilon) \leq v^\ell(x,\epsilon) \leq v(x,t_1 + \epsilon)$ holds for all $x \geq 0$; this contradicts the definition of $t_1$, so we must have $t_1 \geq t_*$.
\end{proof}

\begin{lem}\label{solw1}
Take $x_0>0$.  Then
$$\phi(x,t)=
\begin{cases}
\tilde
J\Big(\frac{Zx}{x_0}\Big)e^{\Big(1-\frac{Z^2}{x_0^2}\Big)t}\qquad &\text{for
$x\in[0,x_0)$},\\
0 \quad &\text{for $x\ge x_0$},
\end{cases}
$$
solves \eqref{eqw2} with $\phi_0(x)=\tilde
J\Big(\frac{Zx}{x_0}\Big)$ for $x\in [0,x_0]$.
\end{lem}
\begin{proof}
Since $\tilde J(0)=0=\tilde J(Z)$, the boundary conditions are satisfied.
The result follows by~\eqref{eqJt} and an easy calculation.
\end{proof}

\begin{lem} \label{lem:toolvlower}
There exists $\beta=\beta(K)>0$ such that for any $x_0\ge Z$ and $c \in (0,e^{-1})$, for $t\ge 0$ and $x\ge 0$,
\begin{equation}\label{ulvtJ}
\underline v(x,t+1) \ge c \beta
V\Big(\frac{Zx}{x_0}\Big)e^{\left(1-\frac{Z^2}{x_0^2}\right)t}\qquad\text{if $c\beta e^{\left(1-\frac{Z^2}{x_0^2}\right)t}\le 1$}.
\end{equation}
\end{lem}
\begin{proof}
Recalling the definition of $w$ in~\eqref{wdef}, let 
\[
\underline{v}^\ell(x,t) = c e^t w(K,x,t) = ce^t \p{\|B_t\|< x \,\, |\,\, \|B_0\| =K},
\]
where $(B_t)_{t\ge 0}$ is a $d$-dimensional Brownian motion with diffusivity $\sqrt 2$.
Due to~\eqref{weqn1}, this function is a solution of the linear problem~\eqref{linear equ} with initial condition $v_0^\ell(x)=\underline{v}_0(x)=c\indic{x>K}$.
Since $c<e^{-1}$ we have $\underline{v}^\ell (\cdot,s)\le e^1 c\le 1$ for all $s\le 1$, and so by Lemmas~\ref{bound above} and~\ref{bound below}
\[
\underline v (x,1)=\underline v^\ell(x,1) = e^1 c w(K,x,1) =  e^1 c \p{\|B_1\|< x \,\, |\,\, \|B_0\| =K}
\ge e^1 c (4\pi)^{-d/2} e^{-\frac 14 (K+x)^2} \omega_d x^d,
\]
where $\omega_d$ is the volume of the $d$-dimensional unit ball.
As in the proof of Lemma~\ref{upbndv}, recall that there exists $c_1\in (0,\infty)$ such that $\tilde J(x)=V(x)\sim c_1 x^d$ as $x\searrow 0$, and note that both $V(x)$ and $w(K,x,1)$ converge to 1 as $x\to\infty$.
It follows that the constant 
\[
\beta := \inf_{x>0} \frac{e^1  w(K,x,1) }{V(x)}>0
\]
is positive. 
Note that by~\eqref{eqJt}, $\tilde J'(x)<0$ for $x\in (R_\infty,Z)$, and so $V(x)\ge \tilde J(x)$ for all $x\le Z$.
Now take $x_0\ge Z$. Since $V$ is non-decreasing, we have
\begin{equation}
\underline v(x,1)=e^1cw(K,x,1)\ge c\beta V(x)\ge c\beta V\left( \frac{Zx}{x_0}\right) \ge c \beta \tilde J\left( \frac{Zx}{x_0}\right) 
\quad \text{for all }x\in [0,x_0].
\label{eq:v1lowerbound}
\end{equation}
By~\eqref{eq:v1lowerbound} and Lemmas~\ref{dncmp} and~\ref{solw1},
we have that if $c \beta e^{\left(1-{Z^2}/{x_0^2}\right)t}\le 1$ then
$$
\underline v(x,t+1)\ge c \beta \tilde J\left( \frac{Zx}{x_0}\right) e^{\left(1-\frac{Z^2}{x_0^2}\right)t}
\quad \text{for all }x\in [0,x_0].
$$
Since $\underline v(\cdot,t+1)$ is non-decreasing, the result follows by the definition of $V$.
\end{proof}
We can now complete the proof of Proposition~\ref{prop:conv v}.
\begin{lem}\label{lowbndR}
Let $\beta = \beta(K)>0$ be as in Lemma \ref{lem:toolvlower} and suppose $c \in (0,e^{-1}\wedge \beta^{-1})$. For $t>-\log (c \beta) $ and $x\ge 0$,
$$\underline v(x,t+1) \ge 
V\left(x\sqrt{1+\frac{\log (c\beta)}t}\right)\qquad\text{and}\qquad
\underline R_{t+1}\le \frac{R_\infty}{\sqrt{1+\frac{\log (c \beta)}t}}.$$
\end{lem}
\begin{proof}
Take $t>-\log(c\beta) > 0$ and choose $x_0=Z/\sqrt{1+\frac{\log (c\beta)
}{t}}>Z$ so  that
\[
c\beta e^{\left(1-{Z^2}/  {x_0^2}\right)t}=1.
\]
Then by Lemma~\ref{lem:toolvlower},
\[
\underline v(x,t+1) \ge 
V\Big(\frac{Zx}{x_0}\Big)
 \quad \text{for all }x\ge 0,
\]
and in particular $\underline R _{t+1}\le \frac{x_0}{Z} R_\infty$.
Since $\frac{Z}{x_0}=\sqrt{1+\frac{\log (c\beta) }{t}}$, the result follows.
\end{proof}
Note that for $t\ge -\frac 4 3 \log (c\beta)$, for $x\in [0,2R_\infty]$,
$$V\left(x\sqrt{1+\frac{\log (c \beta)}t}\right)\ge V(x)-\|V'\|_{L^\infty} \cdot 2R_\infty \left(1-\sqrt{1+\frac {\log (c\beta) }t}\right),$$
and for any $x\ge 2R_\infty$,
$V\left(x\sqrt{1+\frac{\log (c\beta)}t}\right)\ge 1=V(x)$.
This completes the proof of Proposition~\ref{prop:conv v} for $c, K$ with $c<e^{-1}\wedge \beta(K)^{-1}$;
the result for $c \ge e^{-1}\wedge \beta(K)^{-1}$ follows by the comparison principle in Theorem~\ref{thm:exists v}.

The convergence result in
Proposition~\ref{prop:conv v} does not clearly answer the question: what is the time $T(K,c)$ needed  for $v(1,t)$ to be $\mathcal O(1)$ when starting from the initial condition $v_0(x)=c\indic{x>K}$? From Lemma~\ref{lowbndR}, it is clear that $T(K,c)\lesssim -\log c + T'(K)$. Proposition~\ref{prop:movemassPDE} below implies that
\[
T(K,c) \le t_0(K-\log_2 c)
\]
for some constant $t_0$ which does not depend on $K$ and $c$.
This result will be used in the companion paper~\cite{BBNP2} to prove results about the long term behaviour of the $N$-BBM particle system for large $N$, by controlling the proportion of particles which are fairly close to the origin.
\begin{prop} \label{prop:movemassPDE}
There exist $t_0>1$ and $c_0\in (0,1/2)$ such that for all $c\in 
(0,c_0]$, all $K\ge 2$, and all $t_1\in [t_0,2t_0]$, for $v_0:[0,\infty)\to [0,1]$ measurable, the condition
\begin{equation*}
v_0(x) \geq c \indic{x > K} \;\;\forall x \geq 0\qquad\text{(\textit{i.e.} $v_0\in I_{K,c}$)}
\end{equation*}
implies that 
\begin{equation}
v(x,t_1) \geq 2c \indic{x> K-1} \;\;\forall x \geq 0\qquad\text{(\textit{i.e.} $v(\cdot,t_1)\in I_{K-1,2c}$)}, \label{vshift1}
\end{equation}
and
\begin{equation}
v(x,n t_1) \geq \min(2 c_0 \,, \,2^n c) \indic{x>\max(K-n,1)}, \quad \forall \;\;x \geq 0, \quad n \in \N, \label{vshiftn}
\end{equation}
where $v(x,t)$ denotes the solution of~\eqref{pbv} with initial condition $v_0$. 
\end{prop}
\begin{proof}
Let $c \in (0,1)$ and suppose that $v_0(x) \geq c \indic{x>K}$. 
Recall the definition of $w$ in~\eqref{wdef}.
From~\eqref{weqn1}, the function $v^\ell(x,t) = c e^t w(K,x,t)$ satisfies the linear problem \eqref{linear equ} with initial condition $v^\ell_0(x) = c \indic{x>K}$. Since $w \leq 1$ always, we have $c e^t w(K,x,t) \leq 1$ for all $x\ge 0$, if $t \leq t'_c := -\log c$. Therefore, by Lemma \ref{bound below}, we have
\begin{equation}
v(x,t) \geq c e^t w(K,x,t), \quad \forall \;\; x \geq 0, \quad t \leq t'_c = - \log c. \label{vclower1}
\end{equation}

We claim that there exists a constant $\sigma > 0$ such that $w(K,K-1,t) \geq \sigma t^{-d/2}$ for all $K \geq 2$ and $t \geq 1$. Indeed, this lower bound follows easily from the explicit formula \eqref{wdef} for $w$.  Now take $t_0 > 1$ such that $e^{t_0} \sigma (2t_0)^{-d/2} \geq 2$, and let $c_0 = e^{-2 t_0}$. Then for all $c \in (0,c_0]$, we have $2 t_0 \leq t'_c$. Hence from~\eqref{vclower1}, we obtain that for $t_1\in [t_0,2t_0]$,
\[
v(K-1,t_1) \geq c e^{t_1} w(K,K-1,t_1) \geq c e^{t_0} \sigma (2t_0)^{-d/2} \geq 2 c.
\] 
Since $v(\cdot,t_1)$ is non-decreasing, this implies that
\[
v(x,t_1) \geq 2c \indic{x>K-1} \quad \forall x\ge 0.
\]
This proves \eqref{vshift1}.  The bound \eqref{vshiftn} follows immediately by iterating \eqref{vshift1}.
\end{proof}

\section{Proof of Theorem \ref{thm:exists u}: Existence and uniqueness of
\texorpdfstring{$u$}{u}} \label{sec:uprops}

We now focus on the properties of solutions $u$ to the free boundary problem~\eqref{pbu}.

Let $(B_t)_{t\ge 0}$ be a Brownian motion (with diffusivity $\sqrt 2$) on $\R^d$. For a Borel probability measure $\mu$ on $\R^d$, we use $\P_{\!\mu}$ to denote the Wiener measure on $C([0,\infty);\R^d)$ under which $\omega \mapsto B_t(\omega) = \omega(t)$ is a Brownian motion (with diffusivity $\sqrt 2$), with $B_0 \sim \mu$.  Let $\E_{\mu}$ denote expectation with respect to $\P_{\!\mu}$.  For the particular case that $\mu = \delta_x$ is a point mass at some $x \in \R^d$, we use the notation $\P_{\!x}$ and $\E_x$.  Given a continuous function $R:(0,\infty) \to \,(0,\infty)$, $t\mapsto R_t$ with $\lim_{t \searrow 0} R_t = R_0 \in [0,\infty]$, define the random time
\begin{align}
\tau=\tau_R = \inf \big\{ t > 0 \;:\; \|B_t\| \geq R_t \big\}, \label{tauRdef}
\end{align}
which is a stopping time with respect to the usual right-continuous filtration $\mathcal{F}_{t}^+$. When there is no ambiguity, we write $\tau$ rather than $\tau_R$ to lighten the notation. Notice that it is possible that $\tau > 0$ even if $\|B_0\| = R_0$ when $R_0 < \infty$.  By the Blumenthal 0-1 law, we know that for each $x\in \R^d$, $\P_{\!x}( \tau = 0) \in \{0,1\}$. 

For $t > 0$, define a family of measures on $\R^d$ according to
\begin{equation} \label{eq:rhotdef}
\rho_t(x,A) = \P_{\!x}( B_t \in A, \; \tau \geq t) 
\end{equation}
for all Borel sets $A \subseteq \R^d$.  (This $\rho_t$ depends implicitly on $R$ through $\tau$.) The measure $\rho_t(x,\diffd y)$ is not a probability measure, because mass is lost when $B$ hits the moving boundary defined by $R$. Nevertheless, $\rho_t(x,\diffd y)$ is absolutely continuous with respect to Lebesgue measure, so it has a density. Abusing notation, we denote this density by $\rho_t(x,y)$. Since $\rho_t(x,A) \leq \P_{\!x}( B_t \in A)$, for all $x \in \R^d$ the upper bound
\begin{align}
\rho_t(x,y) \leq \Phi(x,0,y,t) \label{rhoupper}
\end{align}
holds for almost every $y \in \R^d$, where
\begin{equation}\label{eq:Phitransitiondef}
\Phi(x,s,y,t) = \frac{1}{(4 \pi(t-s))^{d/2}} e^{- \frac{\|x - y\|^2}{4 (t-s)}}
\end{equation}
denotes the transition density for Brownian motion. Whether the function $y\mapsto \rho_t(x,y) \in L^1(\R^d)$ has a continuous version depends on $R$. Indeed, $\rho_t(x,y)=0$ if $\|y\|>R_t$, but (for example) with $R_t=1+|1-t|^{1/10}$, the quantity $\rho_1(0,y)$ has a positive limit as $\|y\|\nearrow1$.

Given a probability measure $\mu_0$ on $\R^d$ and a continuous function $t\mapsto R_t$, define the function 
\begin{align}
u(y,t) & = e^t \int_{\R^d} \mu_0(\diffd x)\, \rho_t(x,y), \quad y \in \R^d, \; t > 0. \label{udef}
\end{align}
The integral in \eqref{udef} describes the initial measure $\mu_0$ evolving forward in time by diffusion, but with mass lost upon hitting the boundary defined by the function $R$; the factor $e^t$ compensates for the loss of mass. By definition, $u(\cdot,t) \in L^1(\R^d)$ for each $t > 0$.  For a probability measure $\mu_0$ on $\R^d$, define its radial distribution function $v_0:[0,\infty) \to [0,1]$ by
\begin{align}
v_0(x) = \int_{\B(x)} \mu_0(\diffd y) =  \mu_0\big(\B(x)\big),  \label{v0def}
\end{align} 
which is non-decreasing and left-continuous with $v_0(+\infty)=1$.  Recall that $\mathcal{B}(x) =\{y:\|y\|<x\}\subset \R^d$ is the open ball of radius $x$, centered at the origin.

The following result will be used to prove the existence part of Theorem~\ref{thm:exists u}, and provide a representation for the solution.

\begin{prop} \label{prop:ufb}
Let $\mu_0$ be a Borel probability measure on $\R^d$, 
and let $v_0$ be its radial distribution function, defined by~\eqref{v0def}.
Let $v$ denote the solution of
the obstacle problem~\eqref{pbv} with initial condition $v_0$;
for $t>0$, let $R_t=\inf\{x:v(x,t)=1\}$ and let $R_0=\lim_{t\searrow 0}R_t \in [0,\infty]$.
Using this boundary $R_t$, let $u$ be defined by~\eqref{udef}.
Then $u$ has a version which is continuous on $\R^d \times (0,\infty)$ and smooth on $\{(x,t):t>0,\|x\|<R_t\}$, and the pair $(u,R)$ satisfies 
\begin{subequations}\label{pbugivenR}
\begin{alignat}{3}
& \partial_t u = \Delta u + u,  && \text{for }t>0, \, \|y\| < R_t, \label{upde}\\
& u(y,t) = 0,  && \text{for }t > 0, \, \|y\| \geq R_t,  \label{uboundary} \\
& u(\cdot,t) \to \mu_0  && \text{weakly as } \; t \searrow 0, \label{t0conv} \\
& (y,t) \mapsto u(y,t) \text{ is continuous for } t > 0,   \label{continuity}
\end{alignat}
\end{subequations}
and the mass constraint 
\begin{equation}
\int_{\B( R_t)} u(y,t) \,\diffd y = 1, \quad \text{for }t > 0. \label{umass}
\end{equation}
Hence, $(u,R)$ is a classical solution of the free boundary problem \eqref{pbu}.  
\end{prop}

From Proposition \ref{prop:ufb}, the existence part of Theorem~\ref{thm:exists u} now follows. Given $\mu_0$, let $v_0$ be defined by \eqref{v0def}, and let $v$ be the solution to the obstacle problem \eqref{pbv} and $R_t=\inf\{x:v(x,t)=1\}$. Existence and uniqueness of $(v,R)$ is guaranteed by Theorem \ref{thm:exists v}.  
Moreover, by Theorem~\ref{thm:exists v}, $R_t$ is finite and continuous for $t>0$, and $\lim_{t\searrow 0}R_t=\inf\{x:v_0(x)=1\}=\inf\{r:\mu_0(\B(r))=1\}.$
Then Proposition \ref{prop:ufb} implies that the function \eqref{udef} is a classical solution to the free boundary problem~\eqref{pbu}.  Notice that Proposition~\ref{prop:ufb} does not imply that a solution to \eqref{pbugivenR} exists for any continuous function $R$; it only guarantees existence for the function $R$ obtained from solving the obstacle problem \eqref{pbv}.

To prove the uniqueness part of Theorem \ref{thm:exists u}, we first argue that the boundary $R_t$ is uniquely determined. 

\begin{lem} \label{lem:vfromu}  Let $(u,R)$ be any solution to the free boundary problem \eqref{pbu} with initial condition $\mu_0$. Then the function 
\[
v(x,t) = \int_{\B(x)} u(y,t) \,\diffd y, \quad \quad  x \geq 0,\,  t > 0
\]
solves the obstacle problem \eqref{pbv} with initial condition $v_0(x)=\mu_0(\B(x))$. 

For $t > 0$, $u(y,t) > 0$ holds if $\|y\| < R_t$. Moreover, $R_t = \inf\{x:v(x,t)=1\}$ for all $t > 0$.
\end{lem}
Suppose that there are at least two solutions, $(u\uppar1,R\uppar1)$ and $(u\uppar2,R\uppar2)$, to the free boundary problem \eqref{pbu} with the same initial measure $\mu_0$.  Then  by Lemma~\ref{lem:vfromu} the corresponding distribution functions
\begin{align*}
v\uppar1(x,t) = \int_{\B(x)} u\uppar1(y,t) \,\diffd y \quad \text{ and } \quad v\uppar2(x,t) = \int_{\B(x)} u\uppar2(y,t) \,\diffd y
\end{align*}
both satisfy the obstacle problem \eqref{pbv} with the same initial condition $v_0(x)=\mu_0(\B(x))$. By the uniqueness statement in Theorem \ref{thm:exists v}, $v\uppar1 = v\uppar2$.  Using Lemma~\ref{lem:vfromu} again, we conclude that $R\uppar1_t = \inf\{x:v\uppar1 (x,t)=1\} = \inf\{x:v\uppar2 (x,t)=1\}  = R\uppar2_t$ for all $t > 0$. From Theorem \ref{thm:exists v} we also infer that $R$ is finite, strictly positive and continuous for $t > 0$,
 and that $R_0 = \lim_{t\searrow 0}R_t =\inf\{r:\mu_0(\B(r))=1\}$. Since $u\uppar1$ and $u\uppar2$ both satisfy \eqref{pbugivenR} with the same boundary $R$, the next proposition implies that $u\uppar1$ and $u\uppar2$ coincide, and are given by the same formula \eqref{udef}.  Therefore, the solution $(u,R)$ to the free boundary problem \eqref{pbu} is unique.

\begin{prop} \label{prop:utilderho}
Let $\mu_0$ be a Borel probability measure on $\R^d$. Let $R:(0,\infty) \to (0,\infty)$ be continuous for $t>0$, with $\lim_{t \searrow 0} R_t = R_0 \in [0,\infty]$. 
Suppose that $u$ is any function satisfying~\eqref{pbugivenR}. Then the function $u$ must be given by the representation \eqref{udef}, meaning that for all $t > 0$, the equality \eqref{udef} holds for almost every $y \in \R^d$. In particular, for a given boundary function $R_t$, there is at most one function $u$ satisfying \eqref{pbugivenR}.
\end{prop}

The remaining statements in Theorem~\ref{thm:exists u} about properties of $R_t$ follow immediately from Proposition~\ref{prop:Rtcontinuous}, proved already.   Except for the proofs of Proposition \ref{prop:ufb}, Proposition~\ref{prop:utilderho}, and Lemma~\ref{lem:vfromu}, this completes the proof of Theorem \ref{thm:exists u}.  After giving a short proof of Lemma~\ref{lem:vfromu}, we devote the  rest of this section  to the proofs of Propositions~\ref{prop:ufb} and~\ref {prop:utilderho}. We prove Proposition~\ref{prop:utilderho} first, because the proof of Proposition~\ref{prop:ufb} will rely on Proposition \ref {prop:utilderho}.

\begin{proof}[Proof of Lemma~\ref{lem:vfromu}]
The statement that $v$ satisfies \eqref{pbv} is easy to check by direct calculation.  The only part that may not be clear is that $v(x,t)$ converges in $L^1_{\text{loc}}$ to the initial condition $v_0(x)=\mu_0(\B(x))$, as $t \to 0$. Let $\{\mu_t(\diffd y)\}_{t>0}$ denote the family of probability measures on $\R^d$ having density $u(y,t)$.  Thus, $v(x,t) = \mu_t(\mathcal{B}(x))$ is the radial distribution function, and $\mu_t \to \mu_0$ weakly as $t \to 0$, by assumption.  It is well known (by the Portmanteau theorem) that the weak convergence $\mu_t \to \mu_0$ as $t \to 0$ is equivalent to the condition that $\mu_t(A) \to \mu_0(A)$ whenever $A$ is a continuity set for $\mu_0$ (see Chapter 1, Section 2 of \cite{Bill99}). Therefore, if $\mu_t \to \mu_0$ weakly, and if $x \geq 0$ is a point of continuity for the function $v_0(x)=\mu_0(\B(x))$, then $v(x,t) \to v_0(x)$ holds at $x$ (since we may take $A = \mathcal{B}(x)$). Thus, $v(x,t) \to v_0(x)$ holds for all $x \geq 0$ except, possibly, at points where $v_0(x)$ is discontinuous, which is a countable set of points since $v_0(x)$ is non-decreasing.  The dominated convergence theorem then implies $v(x,t) \to v_0(x)$ in $L^1_\text{loc}$.

Any solution to \eqref{pbu} must satisfy $u(y,t) > 0$ if $\|y\| < R_t$ and $t > 0$. This is a consequence of the strong maximum principle, as follows.  Suppose $u(y_0,t_0) = 0$ for some $t_0 > 0$, and $\|y_0\| < R_{t_0}$. For $r \in (\|y_0\|,R_{t_0})$, by continuity of $u$ we can take $\epsilon \in (0,t_0)$ small enough so that $R_t > r$ for all $t \in [t_0-\epsilon,t_0]$. Then the strong maximum principle implies that $u(y,t) = 0$ for all $(y,t) \in \mathcal{B}(r)\times  [t_0-\epsilon,t_0]$.  Because $r$ may be chosen arbitrarily close to $R_{t_0}$, we conclude that $u(y,t_0) = 0$ for all $\|y \| < R_{t_0}$, which contradicts the mass constraint \eqref{umass}.  Hence, $u(y,t) > 0$ holds if $\|y\| < R_t$ and $t > 0$.  
The relation $R_t = \inf\{x:v (x,t)=1\}$ for $t > 0$ now follows from the definition of $v$.
\end{proof}

\begin{rmk} \label{rmk:holder}
If $(u,R)$ solves the free boundary problem \eqref{pbu}, then the mass conservation \eqref{umass} and the representation \eqref{udef} imply that $\P_{\!\mu_0}( \tau > 0) = 1$ (see also Lemma \ref{lem:rhomass} below).  In particular, if $R_0 < \infty$ and the initial measure $\mu_0$ puts positive mass at $\|x\| = R_0$, then for any $\alpha \geq 1/2$ and $C > 0$, it cannot be true that $R_t - R_0 \leq C t^{\alpha}$ as $t \to 0$.  This is because, with probabilty one, $\|B_t\|$ is not $\alpha$-H\"older continuous for $\alpha \geq 1/2$. In this case, $R_t$ must quickly increase away from $R_0$ to avoid eliminating positive mass instantaneously.  Thus, the condition $\alpha < 1/2$ in the one-sided H\"older bound of Theorem \ref{thm:exists u} cannot be improved at $s= 0$ without assuming more regularity of $\mu_0$.  
\end{rmk}

\subsection{Proof of Proposition \ref{prop:utilderho}}

For
$R:(0,\infty)\to (0,\infty)$ continuous with $\lim_{t\searrow 0}R_t =R_0\in [0,\infty]$, for
 $0 \le s < t$ and $x \in \R^d$, let us define $y \mapsto p_R(x,s,y,t)=p(x,s,y,t)$ to be the density of the measure
\begin{equation}
A  \mapsto \P( B_t \in A, \  \tau_{R,s} \geq t  \;|\; B_s = x) \label{pdensitydef}
\end{equation}
where
\begin{equation} \label{eq:tauRsdefn}
\tau_{R,s} = \tau_s = \inf \{ t > s \;:\; \|B_t\| \geq R_t \}.
\end{equation}
(Usually we will drop the subscript $R$ and simply write $p(x,s,y,t)$ and $\tau_s$.)
Hence $p(x,0,y,t) = \rho_t(x,y)$. 
For $t>0$, we define
\begin{equation} \label{eq:tau'def}
\tau'_t=\tau'_{R,t}=\inf\{t'\in (0,t]\,:\, \|B_{t'}\|\ge R_{t-t'}\}.
\end{equation}
We shall use the following result in the proofs of Propositions~\ref{prop:utilderho} and~\ref{prop:ufb}.
\begin{lem}\label{lem:timereverseFK}
Suppose $0\le s <t$, and $R:[0,\infty) \to [0,\infty)$ is continuous on $[s,t]$, and define $\tau_{s}$, $\tau'_t$, and $p$ as in~\eqref{eq:tauRsdefn},~\eqref{eq:tau'def} and~\eqref{pdensitydef}. For any bounded, measurable function $g:\R^d \to \R$,
\[
\int_{\R^d}  g(x) p(x,s,y,t) \,\diffd x = \E_y\left[  g(B_{t - s}) \indic{\tau_{t}'\geq t - s} \right]
\]
holds for almost every $y \in \R^d$.
\end{lem}

\begin{proof}
Note first that if we let $\tilde R_{s'}=R_{s'+s}$ for $s'\in [0,t-s]$, then $\tilde R$ is continuous on $[0,t-s]$ and we have $p(x,s,y,t)=p_R(x,s,y,t)=p_{\tilde R}(x,0,y,t-s)$, and $\{\tau'_t=\tau'_{R,t}\ge t-s\}=\{\tau'_{\tilde R,t-s}\ge t-s\}$.
Hence
it suffices to assume $s = 0$ to simplify the notation.  So we want to show that
\[
\int_{\R^d}  g(x) p(x,0,y,t) \,\diffd x = \E_y\left[  g(B_t) \indic{\tau_{t}' \geq t} \right], \quad \text{ for a.e.~}y \in \R^d,
\]
assuming that $R$ is continuous on $[0,t]$. Define the function
\[
w(y) = \int_{\R^d}  g(x) p(x,0,y,t) \,\diffd x, \quad y \in \R^d.
\]
Let $\P_{x,y,t}$ be the measure on $C([0,t];\R^d)$ under which $(B_s)_{s\in [0,t]}$ is a Brownian bridge (with diffusivity $\sqrt 2$) in $\R^d$ from $(x,0)$ to $(y,t)$, and let $\E_{x,y,t}$ denote expectation with respect to this measure.  Then, recalling from~\eqref{tauRdef} that $\tau=\tau_0$, and recalling the definition of the transition density $\Phi$ in~\eqref{eq:Phitransitiondef}, we have
\[
p(x,0,y,t)=\Phi(x,0,y,t) \P_{x,y,t}(\tau \geq t), \quad \text{a.e.}\;\;y\in \R^d
\]
for all $x \in \R^d$.
Since under $\P_{x,y,t}$, $(B_s)_{s\in [0,t]}$ is a Brownian bridge from $(x,0)$ to $(y,t)$, we have that under the same measure, the time-reversed process $B_r' = B_{t-r}$ is a Brownian bridge from $(y,0)$ to $(x,t)$.  Therefore,
since $\{\tau\ge t\}=\{\|B_s\|<R_s\,\, \forall s\in (0,t)\}=\{\|B'_s\|<R_{t-s}\,\forall s\in (0,t)\},$ we have that
\[
 \P_{x,y,t}(\tau \geq t) = \P_{y,x,t}( \tau_{t}' \geq t).
\]
Since $\Phi(x,0,y,t)=\Phi(y,0,x,t)$, it
follows that for almost every $y\in \R^d$, by Fubini's theorem,
\begin{align*}
w(y) & = \int_{\R^d}  \Phi(y,0,x,t) \P_{y,x,t}(\tau_{t}' \geq t)g(x) \,\diffd x  = \E_y \left[ g(B_t) \indic{\tau_{t}' \geq t} \right],
\end{align*}
as desired.
\end{proof}

\begin{proof}[Proof of Proposition \ref{prop:utilderho}]
Assume that $u$ satisfies \eqref{pbugivenR}.
Due to \eqref{t0conv}, we know that 
\begin{equation} \label{eq:uintnear1}
\lim_{s \to 0} \int_{\R^d} u(y,s) \,\diffd y = \int_{\R^d} \mu_0(\diffd y) = 1.
\end{equation}
Let $0 < s < t$. Although $\mu_0$ is a measure, the function $u$ is continuous on $\R^d \times [s,\infty)$, by \eqref{continuity}. Therefore, by the Feynman-Kac formula for $u$ (e.g. Theorem II.2.3 of \cite{Freidlin85}) applied on the time interval $[s,t]$, we have
that if $\|y\|\neq R_t$,
\begin{align}
 u(y,t) & = e^{t - s}  \E_y\left[  u(B_{t - s},s) \indic{\tau_{t}' \ge t - s} \right], \label{tubackward} 
\end{align}
where 
\[
\tau'_{t}=\inf\big\{t'\in (0,t]:\|B_{t'}\| \geq R_{t-t'}\big\}
\]
as defined in~\eqref{eq:tau'def}. This is a stopping time with respect to the usual right-continuous filtration $\mathcal{F}_{t'}^+$ for $B_{t'}$.  Since $s > 0$, we know that $x \mapsto u(x,s)$ is a bounded and measurable function (it is continuous and compactly supported). Applying Lemma~\ref{lem:timereverseFK}, with $g(x) = u(x,s)$, we see that for $s \in(0,t)$,~\eqref{tubackward} implies that
\begin{align}
 u(y,t) & = e^{t-s} \int_{\R^d}  u(x,s) p(x,s,y,t) \,\diffd x \label{tuforwards}
\end{align}
for almost every $y\in \R^d$.

We now want to take the limit $s \to 0$ of the expression \eqref{tuforwards}.
For $0\le s<t$ and $x,y\in \R^d$, define
\begin{equation}
h(x,s,y,t) = \Phi(x,s,y,t) - p(x,s,y,t) \geq 0.  \label{hdef}
\end{equation}
For any $0 < s < \epsilon < t$, for a.e.~$y\in \R^d$, the Chapman-Kolmogorov equation and then Fubini's theorem imply that
\begin{align}
 u(y,t) &  =  e^{t-s} \int_{\R^d}   \left( \int_{\R^d} u(x,s) p(x,s,z,\epsilon)  \,\diffd x \right) p(z,\epsilon,y,t) \,\diffd z \nonumber \\
& = e^{t-s} \int_{\R^d}   \left( \int_{\R^d} u(x,s) \Phi(x,s,z,\epsilon)  \,\diffd x \right) p(z,\epsilon,y,t) \,\diffd z  \nonumber \\
& \quad - e^{t-s} \int_{\R^d}   \left( \int_{\R^d} u(x,s) h(x,s,z,\epsilon)  \,\diffd x \right) p(z,\epsilon,y,t) \,\diffd z, \label{error1}
\end{align}
where $h$ was defined above in \eqref{hdef}. Now we want to let $s \to 0$, and then $\epsilon \to 0$, to recover the representation \eqref{udef}. For $\epsilon > 0$ fixed, we have uniform convergence of $\Phi(x,s,z,\epsilon)$ to $\Phi(x,0,z,\epsilon)$ as $s \to 0$:
\[
\lim_{s \to 0} \sup_{x,z\in \R^d} \big|\Phi(x,s,z,\epsilon) - \Phi(x,0,z,\epsilon)\big| = 0.
\]
Therefore,
since $ u(\cdot,s) \to  \mu_0$ weakly as measures on $\R^d$ as $s \to 0$, and
 using~\eqref{eq:uintnear1}, this implies that for the first integral on the right hand side of~\eqref{error1} we have
\begin{align}
& \lim_{s \to 0} e^{t-s} \int_{\R^d}   \left( \int_{\R^d} u(x,s) \Phi(x,s,z,\epsilon)  \,\diffd x \right) p(z,\epsilon,y,t) \,\diffd z \nonumber\\
&\quad  = e^{t} \int_{\R^d}   \left( \int_{\R^d} \mu_0(\diffd x) \Phi(x,0,z,\epsilon)  \right) p(z,\epsilon,y,t) \,\diffd z. \nonumber
\end{align}
Using the definition of $h$ in~\eqref{hdef},
let us write this last integral as
\begin{align}
\int_{\R^d}   \left( \int_{\R^d} \mu_0(\diffd x) \Phi(x,0,z,\epsilon) \right) p(z,\epsilon,y,t) \,\diffd z &  =   \int_{\R^d}   \left( \int_{\R^d} \mu_0(\diffd x) p(x,0,z,\epsilon)  \right) p(z,\epsilon,y,t) \,\diffd z \nonumber \\
& \qquad  + \int_{\R^d}   \left( \int_{\R^d} \mu_0(\diffd x) h(x,0,z,\epsilon) \right) p(z,\epsilon,y,t) \,\diffd z \nonumber \\
 & =    \int_{\R^d} \mu_0(\diffd x) p(x,0,y,t) \nonumber \\
& \qquad  +  \int_{\R^d}   \left( \int_{\R^d} \mu_0(\diffd x) h(x,0,z,\epsilon) \right) p(z,\epsilon,y,t) \,\diffd z,  \label{error2}
\end{align}
by Fubini's theorem and the Chapman-Kolmogorov equation.
Notice that for $0\le s <t$ and $x\in \R^d$, for any Borel set $A \subseteq \R^d$,
\begin{align}
0 \leq \int_A h(x,s,y,t) \,\diffd y & =  \P( B_t \in A \;|\; B_s = x) - \P( B_t \in A, \; \tau_{s} \geq t  \;|\; B_s = x) \nonumber \\
& = \P( B_t \in A, \; \tau_{s} < t  \;|\; B_s = x). \label{hbound1}
\end{align}
For $\epsilon < t/2$, $p(z,\epsilon,y,t)\le \Phi(z,\epsilon,y,t)\le  C_t := (2\pi t )^{-d/2}$.  Thus, for $\epsilon < t/2$, the second integral on the right hand side of~\eqref{error2} is controlled by
\begin{align*}
0 \leq  \int_{\R^d} \left( \int_{\R^d} \mu_0(\diffd x) h(x,0,z,\epsilon) \right) p(z,\epsilon,y,t) \,\diffd z  & \leq C_t \int_{\R^d} \int_{\R^d} \mu_0(\diffd x) h(x,0,z,\epsilon)  \,\diffd z  \nonumber \\
& =  C_t \P_{\!\mu_0} ( \tau < \epsilon),  
\end{align*}
where the last line follows by~\eqref{hbound1} and since $\tau=\tau_{0}$, by the definition in~\eqref{tauRdef}.
Similarly, for $\epsilon < t/2$, the second integral on the right hand side of~\eqref{error1} is controlled by
\begin{align}
0  \leq  \int_{\R^d}   \left( \int_{\R^d} u(x,s) h(x,s,z,\epsilon)  \,\diffd x \right) p(z,\epsilon,y,t) \,\diffd z & \leq C_t \int_{\R^d}    u(x,s) \int_{\R^d}  h(x,s,z,\epsilon) \,\diffd z \,\diffd x \nonumber \\
& = C_t \int_{\R^d}    u(x,s) \P \big( \tau_{s} < \epsilon \;\big|\; B_s = x \big) \,\diffd x \nonumber 
\end{align}
by~\eqref{hbound1}. In summary, these computations prove that for $t>0$, there is a constant $C_t<\infty$ such that for any $\epsilon \in (0,t/2)$, for a.e.~$y\in \R^d$,
\begin{align}
\left|u(y,t) - e^t  \int_{\R^d} \mu_0(\diffd x) p(x,0,y,t) \right| & \leq e^t C_t \P_{\!\mu_0} ( \tau < \epsilon)  
+  e^{t} C_t  \limsup_{s \to 0}  f(s,\epsilon), \label{massloss}
\end{align}
where 
$$
f(s,\epsilon):=\int_{\R^d}    u(x,s) \int_{\R^d}  h(x,s,z,\epsilon) \,\diffd z\, \diffd x
=\int_{\R^d}    u(x,s) \P \big( \tau_{s} < \epsilon \;\big|\; B_s = x \big) \,\diffd x.
$$

Now we argue that the right hand side of \eqref{massloss} vanishes as $\epsilon \to 0$. 
By \eqref{tuforwards} again, for $s\in (0,\epsilon)$ we have
\begin{align*}
\int_{\R^d} u(z,\epsilon) \,\diffd z 
&=e^{\epsilon-s}\int_{\R^d}\int_{\R^d}u(x,s)p(x,s,z,\epsilon)\,\diffd z\, \diffd x\\
&=e^{\epsilon-s}\int_{\R^d}\int_{\R^d}u(x,s)\Big(\Phi(x,s,z,\epsilon)-h(x,s,z,\epsilon)\Big)\,\diffd z \,\diffd x \\
&=e^{\epsilon-s}\left( \int_{\R^d}u(x,s)\,\diffd x -f(s,\epsilon)\right),
\end{align*}
which implies that
\[
f(s,\epsilon)=\int_{\R^d}u(x,s)\, \diffd x-e^{s-\epsilon}\int_{\R^d} u(z,\epsilon) \,\diffd z .
\]
Therefore, by~\eqref{eq:uintnear1},
\[
\lim_{s \to 0} f(s,\epsilon)= 1-e^{-\epsilon}\int_{\R^d} u(z,\epsilon) \,\diffd z\to 0 \quad \text{as } \epsilon \to 0. 
\]
We now have that the last term on the right hand side of~\eqref{massloss} vanishes as $\epsilon \to 0$;
recalling the definition of $f$, we have shown that:
\begin{equation}
\lim_{\epsilon \to 0} \lim_{s \to 0} \int_{\R^d}u(x,s)\P \big(\tau_{s} < \epsilon \;\big|\; B_s =x \big)\,\diffd x = 0. \label{tauRseps}
\end{equation}
Finally, we argue that 
\begin{align}
\lim_{\epsilon \to 0} \P_{\!\mu_0} ( \tau < \epsilon) = 0 \label{tauR0}
\end{align}
which, by monotone convergence, is equivalent to the statement that
\begin{equation}
\P_{\!\mu_0} ( \tau > 0) = 1.\label{tauR02}
\end{equation}
Since $u(\cdot,t)$ converges weakly to $\mu_0$ as $t\searrow 0$, and $u(x,t)$ vanishes outside $\B(R_t)$, $\mu_0$ must be supported on $\{x:\|x\|\le R_0\}$.
If it happens that $\mu_0$ puts no mass on the boundary of its support, meaning that $\mu_0( \{ \|x\| < R_0\}) = 1$, then \eqref{tauR02} is an immediate consequence of the almost-sure continuity of Brownian motion and the continuity of $R_t$.  The more delicate case is when $R_0<\infty$ and $\mu_0( \{ \|x\| = R_0\}) = m$ for some $m > 0$. Suppose this is the case, but that \eqref{tauR02} does not hold --- we will derive a contradiction.  
We are going to show that if $R_t$ is such that~\eqref{tauR02} does not hold, then there cannot be a function $u$ as in the statement of the proposition: specifically,~\eqref{t0conv} cannot hold because the boundary would instantaneously absorb some mass from $u$.

If \eqref{tauR02} does not hold, then by the Blumenthal 0-1 law, it follows that
\begin{equation}
\P \big( \tau = 0 \;\big|\; \|B_0\|  = R_0 \big) = 1, \label{contra1}
\end{equation}
since the event $\{\tau > 0\}$ is measurable with respect to the germ $\sigma$-field $\mathcal{F}_0^+$ and must have probability $0$ or $1$, given $\|B_0\| = R_0$. Our argument will be that if \eqref{contra1} holds and if $\mu_0( \{ \|x\| = R_0\}) = m$ for some $m > 0$, then some fraction  of the mass must exit the domain before time $\epsilon$, violating \eqref{tauRseps} in the limit $\epsilon\to 0$.  Recall the definition of $\tau$ in~\eqref{tauRdef}:
\[
\tau = \inf \big\{ t > 0 \;:\; \|B_t\| \geq R_t \big\}. 
\]
Let us also define
\begin{align}
\tau^+ = \inf \big\{ t > 0 \;:\; \|B_t\| > R_t \big\}. \label{tauRdefP}
\end{align}
Clearly $\tau^+ \geq \tau$ must hold. With the function $t\mapsto R_t$ being
deterministic, continuous and positive, one could show that
$\tau^+=\tau$ almost surely. However, to keep the argument short we only prove
what we need here.
We claim that \eqref{contra1} implies
\begin{equation}
\P\big(\tau^+ = 0 \;\big|\; \|B_0\| = R_0\big) = 1.  \label{contra2}
\end{equation}
To see why this is true, observe that if $\tau =0$,  then with probability one, $\|B_{t_n}\| \geq R_{t_n}$ holds along a infinite sequence of (random) times $t_n \to 0$. Hence, if $\theta \in \R^d$ is any fixed unit vector, $\|B_{t_n}+\theta t_n\|=\big(\|B_{t_n}\|^2+ t_n^2+ 2t_n \theta \cdot B_{t_n}\big ){}^{1/2}>R_{t_n}$ if $\theta\cdot B_{t_n} >0$.  Let $e_1,\dots,e_d$ denote the standard orthonormal basis for $\R^d$.  For each $t_n$, one can find at least one vector $\theta \in \{ \pm e_k, \; k = 1,\dots,d\}$ such that $\theta \cdot B_{t_n} > 0$ holds. Then, there must be at least one vector $\theta\in \{ \pm e_k, \; k = 1,\dots,d\}$ such that the relation $\theta \cdot B_{t_n} > 0$ holds for infinitely many $t_n$.  Therefore, there is a vector $\theta \in \{ \pm e_k, \; k = 1,\dots,d\}$ such that 
\[
\max_{s \in [0,\epsilon]} \big(\big\|B_s + \theta s \big\| - R_s\big) > 0, \quad \forall \; \epsilon > 0.
\]
By Girsanov's theorem, for $\theta\in \R^d$, the laws of the processes $Y^\theta_s := B_s + \theta s$  and $B_s$ are mutually absolutely continuous, so 
\[
\P\Big( \max_{s \in [0,\epsilon]} \big(\|B_s\| -R_s\big) > 0 ,\ \forall\epsilon>0 \, \Big| \, \|B_0\|=R_0 \Big)>0
\]
must also hold.  By Blumenthal's 0-1 law, this event must have probability 1 since it cannot have probability 0.  This proves that \eqref{contra2} must hold.

Now we show that \eqref{contra2} contradicts \eqref{tauRseps}. The statement \eqref{contra2} means that $\|B_s\|$ exceeds $R_s$ infinitely often as $s \to 0$.  In other words, for any $\epsilon>0$,
$$\P\bigg(\max_{s \in [0,\epsilon/2]} \big(\|B_s\| - R_s\big) > 0\;\bigg|\;\|B_0\|=R_0\bigg)=1.$$
It follows that for any $\epsilon > 0$, there exists $\delta > 0$ such that
\[
\P\bigg( \max_{s \in [0,\epsilon/2]} \big( \|B_s\| - R_s\big) > \delta \;\bigg|\; \|B_0\| = R_0\bigg) \geq \frac23.
\]
Hence for any $x \in \R^d$ such that $\big|\|x\| - R_0\big| < \delta/2$,
\[
\P_{\!x}\bigg( \max_{s \in [0,\epsilon/2]} \big(\|B_s\| - R_s\big) > \delta/2\bigg ) \geq \frac23.
\]
Since the function $R_s$ is uniformly continuous on $[0,\epsilon]$, there is $r_0 \in(0,\epsilon/2)$ small enough (depending on $\epsilon$ and $\delta$) such that  for any $r \in [0,r_0]$, we have $\max_{s\in[0,\epsilon/2]}|R_s-R_{s+r}|<\delta/4$. This implies that for $r\le r_0$ and $x \in \R^d$ such that $\big|\|x\| - R_0\big| < \delta/2$,
\begin{equation}
\P_{\!x}\bigg(\max_{s \in [0,\epsilon/2]} \big(\|B_s\| - R_{s+r}\big) > \delta/4 \bigg) \geq \frac23. \label{BRexceed}
\end{equation}
Because $u(\cdot,r)$ converges weakly to $\mu_0$ as $r \to 0$, and because we assumed that $\mu_0( \{ \|x\| = R_0\}) = m>0$, we have that for $r_0$ small enough,
\[
\inf_{r \in [0,r_0]} \int_{\{|\|x\| - R_0| < \delta/2 \} } u(x,r) \,\diffd x \geq \frac{m}{2}.
\]
Therefore, for any $r \in (0,r_0]$, the integral in \eqref{tauRseps} can be bounded below by writing
\begin{align}
\int_{\R^d}u(x,r)\P(\tau_{r}< \epsilon \; | \; B_r=x)\, \diffd x 
& \geq \frac m 2 \inf_{\big|\|x\| - R_0\big| < \delta/2}  \P(\tau_{r}< \epsilon \; | \; B_r=x) \nonumber \\
 & \geq \frac m 2 \inf_{\big|\|x\| - R_0\big| < \delta/2}  \P(\tau_{r}< \tfrac 12 \epsilon + r \; | \; B_r=x) \label{uintlower1}
\end{align}
because $r\le r_0<\epsilon/2$.
By the definition of $\tau_{r}$ in~\eqref{eq:tauRsdefn}, 
\begin{align*}
\P(\tau_{r}<\tfrac 12  \epsilon + r \; | \; B_r=x) & = \P\big( \|B_{r + s} \| \geq R_{r + s} \;\;\text{for some $s \in (0,\epsilon/2)$} \; | \; B_r=x\big) \\
& = \P_{\!x}\big( \|B_{s} \| \geq R_{r + s} \;\;\text{for some $s \in (0,\epsilon/2)$}\big)  \\
& \geq \P_{\!x}\Big( \max_{s \in [0,\epsilon/2]}\big( \|B_{s} \| - R_{r + s}\big) > 0 \Big). 
\end{align*}
In particular, this is bounded below by \eqref{BRexceed} and we have that
\begin{align*}
\inf_{\big|\|x\| - R_0\big| < \delta/2}  \P(\tau_{r}< \tfrac 12 \epsilon + r \, | \, B_r=x) 
& \geq \inf_{\big|\|x\| - R_0\big| < \delta/2} \P_{\!x}\bigg(\max_{s \in [0,\epsilon/2]} \big(\|B_s\| - R_{s+r}\big) > \delta/4 \bigg) \geq \frac{2}{3}.
\end{align*}
Returning to \eqref{uintlower1}, we now see that
\begin{align*}
\int_{\R^d}u(x,r)\P(\tau_{r}< \epsilon \; | \; B_r=x)\, \diffd x  & \geq  \frac{m}{3}
\end{align*}
holds for all $r \in (0,r_0]$. Letting $r\to 0$ and then $\epsilon \to 0$, this contradicts \eqref{tauRseps}.  We conclude that~\eqref{tauR0} holds, which, together with~\eqref{massloss} and~\eqref{tauRseps}, completes the proof of Proposition \ref{prop:utilderho}.
\end{proof}

\subsection{Proof of Proposition~\ref{prop:ufb}}

Suppose that $v$ solves the obstacle problem \eqref{pbv} with the non-decreasing initial condition $v_0$ given by~\eqref{v0def}.  
For $t>0,$ let $R_t=\inf\{x:v(x,t)=1\}$, and let $R_0=\lim_{t\searrow 0}R_t$.
Define the function
\begin{equation}
\tilde u(y,t) = \frac{1}{d\,\omega_d \|y\|^{d-1}}\partial_x v(\|y\|,t) \geq 0, \quad \quad y \in \R^d \setminus\{0\}, \, t>0, \label{utildedef}
\end{equation}
and $\tilde u(0,t)=\lim_{y\to 0}\tilde u(y,t)$ for $t>0$,
where $\omega_d = |\B(1)|$, so that $d \, \omega_d = |\partial \mathcal{B}(1)|$  is the $(d-1)$-dimensional measure of the unit sphere in dimension $d$. 
(We shall show below that $\tilde u(0,t)$ is well-defined.)
Observe that the function $y \mapsto \tilde u(y,t)$ has rotational symmetry; as our analysis will show, this function is a symmetrization of the desired solution $u$. Therefore, it will be convenient to define a symmetrization operator, as follows.  

Let $C_b=C_b(\R^d)$ denote the space of continuous and bounded functions $f:\R^d \to \R$. We define the symmetrization map $\Gamma:C_b \to C_b$ by $\Gamma f(0) = f(0)$ and
\begin{align*}
\Gamma f(x) = \dashint_{\partial \B(\|x\|)} f(y)\,\diffd S(y), \quad x \in \R^d, & \quad x\neq 0.
\end{align*}
That is, the symmetrized function $\Gamma f(x)$ is the average of $f$ over the sphere of radius $\|x\|$.  By duality, this map on $C_b$ induces a map $\Gamma^*$ on probability measures: given a probability measure $\mu$ on $\R^d$, let $\Gamma^* \mu$ denote the probability measure defined by
\begin{equation}
\int_{\R^d} f(x) \Gamma^* \mu(\diffd x) = \int_{\R^d} \Gamma f(x)\, \mu(\diffd x), \quad \forall \; f \in C_b(\R^d). \label{GammaStardef}
\end{equation}
The probability measure $\Gamma^* \mu$ is the unique measure which is invariant under any orthogonal transformation of $\R^d$ and which satisfies
\[
\Gamma^* \mu\big( \B(r) \big) =  \mu\big(\B(r)\big), \quad \forall r > 0.
\]
If $\mu(\diffd x)$ has a density with respect to Lebesgue measure, $\mu(\diffd x) = h(x)\,\diffd x$, then $\Gamma^* \mu$ has the density $\Gamma h$.

Given the probability measure $\mu_0$, define $\tilde \mu_0 = \Gamma^* \mu_0$, the symmetrized initial measure. Recalling \eqref{v0def}, we have
\begin{equation*}
v_0(x) = \mu_0\big(\mathcal{B}(x)\big) = \tilde \mu_0\big(\mathcal{B}(x)\big) 
\end{equation*}
for all $x > 0$.

\begin{lem} \label{lem:utilde}
The function $\tilde u$ is well-defined at $y=0$ and satisfies the PDE \eqref{upde}, the boundary condition \eqref{uboundary}, the continuity condition \eqref{continuity}, and the mass constraint \eqref{umass}. Moreover, $\tilde u$ satisfies the initial condition~\eqref{t0conv} with the symmetrized initial measure $\tilde \mu_0$, and $\tilde u$ is given by the representation
\begin{equation}
\tilde u(y,t) = e^t \int_{\R^d} \tilde \mu_0(\diffd x) \,\rho_t(x,y) \quad \text{for }y\in \R^d \text{ and }t>0, \label{tildeurep2}
\end{equation}
where $\rho_t(x,y)$ is defined in~\eqref{eq:rhotdef}.  If $\|y\| < R_t$, then $\tilde u(y,t) > 0$.
\end{lem}
The proof of Lemma~\ref{lem:utilde} will use the following technical lemma.

\begin{lem} \label{lem:weakconv}
Let $\big\{\mu_t(\diffd y)\big\}_{t \geq 0}$ be a family of probability measures on $\R^d$, and let
\[
v(x,t) = \mu_t\big(\mathcal{B}(x)\big), \quad x \geq 0, \, t\ge 0
\]
be the radial distribution functions. If $v(x,t) \to v(x,0)$ in $L^1_\text{loc}$ as $t\searrow 0$, then $\Gamma^* \mu_t \to \Gamma^* \mu_0$ weakly as $t \searrow 0$. 
\end{lem}

\begin{proof}[Proof of Lemma~\ref{lem:weakconv}]
For $t\ge 0$, let $\tilde \mu_t = \Gamma^* \mu_t$ denote the symmetrized measure on $\R^d$.  We must show that 
\[
\lim_{t \to 0} \int_{\R^d} f(y)\, \tilde \mu_t(\diffd y) - \int_{\R^d} f(y)\, \tilde \mu_0(\diffd y) = 0
\]
holds for any continuous and bounded function $f$ on $\R^d$. Because of the rotational symmetry of $\tilde \mu_t$, it suffices to assume that $f$ has the form $f(y) = g(\|y\|)$, where $g:[0,\infty) \to \R$ is continuous and bounded.  Then,
\[
\int_{\R^d} g(\|y\|)\, \tilde \mu_t(\diffd y) - \int_{\R^d} g(\|y\|)\, \tilde \mu_0(\diffd y) = \int_{[0,\infty)} g(r)\, \eta_t(\diffd r) - \int_{[0,\infty)} g(r) \,\eta_0(\diffd r) 
\]
where $\eta_t(\diffd r) = M_{\#} (\tilde \mu_t(\diffd y) )$ is the push-forward of the measure $\tilde \mu_t(\diffd y) $ on $\R^d$ under the map $M:\R^d \to [0,\infty)$ defined by $M(y) = \|y\|$. That is, $\eta_t(C)=\tilde\mu_t(\{x\in\R^d\;:\:\|x\|\in C\})$ for all Borel sets $C \subseteq [0,\infty)$.  Thus, $\eta_t$ is a family of probability measures on $[0,\infty)$, and $v(r,t) =\eta([0,r))$ for all $r \geq 0$.  Since $v(r,t)$ is non-decreasing in $r$ for all $t \geq 0$, the convergence of $v(\cdot,t)$ to $v(\cdot,0)$ in $L^1_{\text{loc}}$ implies that if $r$ is any point of continuity for $v(r,0)$, then pointwise convergence $v(r,t) \to v(r,0)$ must hold at $r$, as $t \to 0$. Therefore, by the Portmanteau theorem (see Chapter~1, Section~2 of~\cite{Bill99}), $\eta_t \to \eta_0$ weakly as $t \to 0$. Hence, 
\[
\lim_{t \to 0}  \int_{[0,\infty)} g(r)\, \eta_t(\diffd r) - \int_{[0,\infty)} g(r)\, \eta_0(\diffd r) = 0 
\] 
as desired.
\end{proof}
\begin{proof}[Proof of Lemma~\ref{lem:utilde}]
Since $R_t>0$ for $t>0$ by Proposition~\ref{prop:Rtcontinuous}, the
continuity \eqref{continuity} of $\tilde u$ at the boundary $\{\|y\| = R_t\}$ for $t>0$ follows from Lemma~\ref{lem:jointcontin}.
 Since $\partial_x v(x,t)=0$ for $x\ge R_t$ and $t>0$, $\tilde u$ satisfies the boundary condition~\eqref{uboundary}.
By Proposition \ref{prop:vsmoothpde}, $v(x,t)$ is a smooth function on $\{(x,t) \;:\; 0 < x < R_t, \;\; t > 0 \}$; hence $\tilde u$ is smooth on $\{(y,t) \in \R^d \times (0,\infty) \;:\; 0 < \|y\| < R_t \}$
and satisfies the PDE~\eqref{upde} there.  We claim that $\tilde u$ is well-defined at $y = 0$ for all $t>0$ and satisfies the PDE \eqref{upde} at $y = 0$, as well. To see this, suppose that $R_t > \epsilon > 0$ for $t \in [t_1,t_2]$, for some $t_1 > 0$.  Let $\hat u$ solve the PDE \eqref{upde} on the fixed cylindrical domain $\{ \|y\| < \epsilon \} \times [t_1,t_2]$ with  boundary condition $\hat u(y,t) = \tilde u(y,t)$ on the parabolic boundary of $\{ \|y\| < \epsilon \} \times [t_1,t_2]$.  This boundary condition is continuous, except possibly at the single point $(0,t_1)$.  Nevertheless, the initial condition $y \mapsto \tilde u(y,t_1)$ is in $L^\infty(\B(\epsilon))$, due to \eqref{utildedef} and \eqref{eq:vxNew}. It follows that the function 
\[
\hat v(x,t) = \int_{\B(x) } \hat u(y,t) \,\diffd y
\]
satisfies the same linear PDE as $v(x,t)$ satisfies on the domain $\{ 0<x < \epsilon \} \times [t_1,t_2]$, with the same initial condition and boundary condition (Neumann at $x = \epsilon$, Dirichlet at $x = 0$). Specifically, the function $\phi(x,t) = e^{-t}(v - \hat v)$ satisfies
\begin{alignat*}{3}
& \partial_t \phi = \partial_x^2 \phi - \frac{d-1}{x} \partial_x \phi, \quad && x \in (0,\epsilon), \;\; t \in [t_1,t_2], \\
& \phi(0,t) = 0, \quad && t \in [t_1,t_2],\\
& \partial_x \phi(\epsilon,t) = 0, \quad && t \in [t_1,t_2],\\
& \phi(x,t_1) = 0, \quad && x \in [0,\epsilon].
\end{alignat*}
This problem has a unique solution: $\phi \equiv 0$.  Indeed, multiplying the PDE (which has the form $\partial_t \phi = x^{(d-1)} \partial_x ( x^{-(d-1)} \partial_x \phi)$) by $x^{-(d-1)} \phi$ and integrating by parts in $x$, one derives:
\[
\partial_t \int_0^\epsilon x^{-(d-1)} \phi^2(x,t) \,\diffd x = - 2 \int_0^\epsilon x^{-(d-1)} (\partial_x \phi(x,t))^2 \,\diffd x \leq 0 , \quad t \in [t_1,t_2].
\]
The integration by parts is justified since $\partial_x \phi(x,t) = \mathcal{O}(x^{d-1})$ as $x \to 0$, uniformly over $t \in [t_1,t_2]$; indeed, as $\hat u$ is bounded we have $\partial_x \hat v(x,t)=\mathcal O(x^{d-1})$, and~\eqref{eq:vxNew} in Proposition~\ref{prop:v is C1} gives us that $\partial_x v(x,t)=\mathcal O(x^{d-1})$. This inequality shows that $\int_0^\epsilon x^{-(d-1)} \phi^2 \,\diffd x$ is non-increasing in $t$.  Since $x^{-(d-1)} \phi^2$ is non-negative and $\phi(\cdot,t_1) \equiv 0$, this implies $\phi \equiv 0$ on $[0,\epsilon] \times [t_1,t_2]$.  Hence $\hat v$ and $v$ coincide on this domain.  Thus, $\tilde u$ coincides with $\hat u$ on $\{\|y\|<\epsilon\}\times [t_1,t_2]$; in particular, $\tilde u$ is well-defined at $y = 0$, $t>0$ by continuity, and it is a solution to \eqref{upde} also at $y=0$, $t>0$.  The fact that $\tilde u$ satisfies the mass constraint \eqref{umass} follows immediately from integration over $\B( R_t)$ and the fact that $v(R_t,t) = 1$ for $t>0$. 
Since $v(\cdot,t)\to v_0(x)$ in $L^1_{\text{loc}}$ as $t\searrow 0$, the
 fact that~\eqref{t0conv} holds with the initial condition $\tilde \mu_0$ is an immediate consequence of Lemma \ref{lem:weakconv}, applied to the family of measures $\mu_t(\diffd y) = \tilde u(y,t) \,\diffd y$ which are invariant under rotation (i.e. $\Gamma^* \mu_t = \mu_t$).

In view of these properties of $\tilde u$, Proposition \ref{prop:utilderho} implies that $\tilde u$ must be given by the representation~\eqref{tildeurep2}. The statement that $\tilde u(y,t) > 0$ if $\|y\| < R_t$ is an immediate consequence of the strong maximum principle, as in the proof of Lemma~\ref{lem:vfromu}. 
\end{proof}

The fact that $\tilde u$ as defined in~\eqref{utildedef} satisfies the Dirichlet boundary condition at $\|y\| = R_t$ will be useful in proving the following lemma, which is needed to show that $u$ as defined in~\eqref{udef} satisfies the boundary condition~\eqref{uboundary}. Let us define the space-time domains
\begin{equation}
D=\big\{(x,t)\;:\; t > 0, \;\; \;\|x\|<R_t\big\}, \label{domainD0}
\end{equation}
and, for $0<a<b$,
\begin{equation}
D_{a,b}=\big\{(x,t)\;:\; t \in (a,b), \;\; \;\|x\|<R_t\big\}. \label{domainD}
\end{equation}
Recall from~\eqref{eq:tau'def} that for $t>0$,
\[
\tau_{t}' = \inf \big\{ s \in (0,t] \;:\; \|B_s\| \geq R_{t - s} \big\}.
\]
Recall that $\P_{\!y}$ is the Wiener measure under which $B$ is a Brownian motion with $B_0 = y$. Under  $\P_{\!y}$, the stopping time $\tau_{t}'$ is the time when the Brownian path hits the moving boundary $R$, moving backwards in time from the point $(y,t)$.

The following result says that under $\P_{\!y}$, $\tau_{t}' \to 0$ in probability as $(y,t)$ approaches any boundary point $(y_*,t_*)$ with $t_*>0$. 

\begin{lem}  \label{lem:taucontin}
Let $(y_*,t_*)\in\partial D$ with $t_*>0$. Then for all $t_1\in(0,t^*)$,
$$\P_{\!y}(\tau'_t\ge t-t_1)\to0\quad\text{as $(y,t)\to(y_*,t_*).$}$$
\end{lem}

\begin{proof}
Let $\epsilon \in (0,R_{t_1}/2)$.  We start with the following simple bound:
for $y\in \R^d$ and $t>t_1$, by the continuity of $R$ and the almost sure continuity of Brownian motion,
\begin{align*}
\P_{\!y}\big( \tau_{t}' \geq t - t_1 \big) & \leq  \P_{\!y}\big(\|B_{t-t_1}\|  - R_{t_1} \in (-\epsilon,0] \big) + \P_{\!y}\big( \|B_{t - t_1}\| \leq R_{t_1} - \epsilon, \;  \tau_{t}' \geq t - t_1  \big). 
\end{align*}
Assuming $t > (t_* + t_1)/2$, we have that $t-t_1 \geq (t_*-t_1)/2>0$, and the first term on the right hand side is bounded by
\[
\sup_{y\in \R^d} \P_{\!y}\big( \|B_{t - t_1}\| -R_{t_1} \in(-\epsilon,0] \big) \leq \frac{1}{(4 \pi (t - t_1))^{d/2}} \,\omega_d  \big[R_{t_1}^d-(R_{t_1}-\epsilon)^d\big] =  \mathcal O(\epsilon )
\]
as $\epsilon \to 0$, where $\omega_d$ is the volume of the $d$-dimensional unit ball. This shows that for any $\delta > 0$, we may choose $\epsilon > 0$  small enough (independent of $y$ and $t$) so that
\begin{equation}
\P_{\!y}\big( \tau_{t}' \geq t - t_1 \big) \leq \delta + \P_{\!y}\big( \|B_{t - t_1}\| \leq R_{t_1} - \epsilon, \;  \tau_{t}' \geq t - t_1  \big) \label{Pdelta1}
\end{equation}
holds as $(y,t) \to (y_*,t_*)$.

By Lemma~\ref{lem:utilde}, the function $\tilde u(x,s)$ is continuous on the closed set $\overline{D_{t_1,t}}$ (recall \eqref{domainD}), satisfies  $\partial_s \tilde u = \Delta \tilde u+\tilde u$ in $D_{t_1,t}$, and satisfies $\tilde u(y,s)=0$ for $s>0$, $\|y\|\ge R_s$.  Applying the usual Feynman-Kac formula to $\tilde u$ in this region, we have that for $\|y\|<R_t$,
\[
\tilde u(y,t) = e^{t-t_1}\E_y\left[ \tilde u(B_{t-t_1}, t_1) \indic{\tau_{t}' \ge t - t_1} \right],
\]
since $\tilde u$ vanishes on the spatial boundary of the domain.  In particular, 
\[
\tilde u(y,t) \geq e^{t-t_1}\P_{\!y}\big( \|B_{t - t_1}\| \leq R_{t_1} - \epsilon, \; \tau_{t}' \geq t - t_1 \big) \inf_{\|x\| \leq R_{t_1} - \epsilon}  \tilde u(x,t_1).
\]
By Lemma~\ref{lem:utilde}, $x \mapsto \tilde u(x,t_1)$ is positive on the compact set $\{ x \in \R^d\;:\; \|x\| \leq R_{t_1} - \epsilon\} $.  Hence
\[
\inf_{ \|x\| \leq R_{t_1} - \epsilon}  \tilde u(x,t_1) > 0.
\]
Therefore, since $\tilde u(y,t) \to 0$ as $(y,t) \to (y_*,t_*) \in \partial D$, we conclude that 
\[
\lim_{(y,t) \to (y_*,t_*)} \P_{\!y}\big( \|B_{t - t_1}\| \leq  R_{t_1} - \epsilon,\;  \tau_{t}' \geq t - t_1  \big) = 0.
\]
This, together with~\eqref{Pdelta1}, implies the desired result, since $\delta > 0$ is arbitrary.
\end{proof}

The following lemma will be useful in proving that the function $u$ as defined in~\eqref{udef} satisfies the mass constraint~\eqref{umass}.
Recall the definition of $\tau$ in~\eqref{tauRdef}.
\begin{lem} \label{lem:rhomass}
For $\mu_0$-almost every $x$,
\[
\lim_{t \to 0} \int_{\R^d} \rho_t(x,y) \,\diffd y =  \P_{\!x}(\tau > 0 )  = 1.
\]
\end{lem}
The interesting case of this lemma is when $\|x\|=R_0$; it is possible that $\mu_0(\{x:\|x\|=R_0\})>0$.
\begin{proof}
Define $m(x,t) = \int_{\R^d} \rho_t(x,y) \,\diffd y$ for $x\in \R^d$ and $t>0$, and observe that
\[
m(x,t) = \P_{\!x}(\tau \geq t ), 
\]
by the definition of $\rho_t(x,\cdot)$ in~\eqref{eq:rhotdef}.  In particular, $t \mapsto m(x,t)$ is non-increasing, so $m(x,t)$ has a well-defined limit as $t\to 0$:
\[
m(x,0) := \lim_{t \to 0} m(x,t) = \P_{\!x}(\tau > 0).
\]
Since, by Lemma~\ref{lem:utilde}, $\tilde u$ satisfies the mass constraint~\eqref{umass}, and then by the representation of $\tilde u$ in~\eqref{tildeurep2}
and Fubini's theorem, we know that for $t>0$,
\[
1 = \int_{\R^d} \tilde u(y,t) \,\diffd y = e^t \int_{\R^d} \tilde \mu_0(\diffd x) \left( \int_{\R^d} \rho_t(x,y) \,\diffd y \right) = e^t  \int_{\R^d} \tilde \mu_0(\diffd x)\, m(x,t).
\]
By the monotone convergence theorem, it follows that
\[
\int_{\R^d} \tilde \mu_0(\diffd x)\, m(x,0) = 1.
\]
Since $m(x,0) \leq 1$ and $\tilde \mu_0(\R^d) = 1$, we conclude that $m(x,0) = 1$ must hold $\tilde \mu_0$-almost everywhere.
By the definition of $\tilde \mu_0 = \Gamma^* \mu_0$, the result follows.
\end{proof}
We are now ready to prove Proposition~\ref{prop:ufb}.
\begin{proof}[Proof of Proposition \ref{prop:ufb}]
{\bf Step 1:} Our first task is to show that \eqref{upde} holds, at all points $(y,t) \in D$, where $D$ is the space-time domain defined in~\eqref{domainD0}. Consider the function 
\begin{equation} \label{eq:wfromudef}
w(y,t) = e^{-t} u(y,t) = \int_{\R^d} \mu_0(\diffd x) \,\rho_t(x,y) \quad \text{for }y\in \R^d, \, t>0.
\end{equation}
Recalling the definitions of $\rho_t(x,y)$ in~\eqref{eq:rhotdef} and of the stopping time  $\tau$ in~\eqref{tauRdef}, we have for any continuous function $f:\R^d\to \R$,
$$\int_{\R^d} f(y) w(y,t)\,\diffd y = \E_{\mu_0}\big[f(B_t)\indic{\tau\ge t}\big].$$
We claim that for any $t_*>0$ and any function $(y,t) \mapsto \phi(y,t)$ which is smooth and compactly supported in the region $D\cap \{(x,t): t \leq t_*\}$,
\begin{equation}
\int_{\R^d}  \phi(y,t_*) w(y,t_*) \,\diffd y   = \int_0^{t_*} \int_{\R^d} \big(\partial_t \phi + \Delta \phi \big)(y,t) w(y,t)\,\diffd y \,\diffd t.  \label{wweak1}
\end{equation}
The left hand side of \eqref{wweak1} is:
\begin{equation*}
\int_{\R^d}  \phi(y,t_*) w(y,t_*) \,\diffd y = \E_{\mu_0} \big[ \phi(B_{t_*},t_*) \indic{\tau \ge t_*} \big]
\end{equation*}
and, by Fubini's theorem and then by It\^o's formula, the right hand side of \eqref{wweak1} is:
\begin{align*}
\int_0^{t_*} \int_{\R^d} \left(\partial_t \phi + \Delta \phi \right)(y,t) w(y,t)\,\diffd y \,\diffd t & = \E_{\mu_0} \left[ \int_0^{t_*}   \left(\partial_t \phi + \Delta \phi \right) (B_t,t) \indic{\tau \ge t}   \,\diffd t\right] \\
& =  \E_{\mu_0} \left[ \int_0^{{t_*} \wedge \tau}   \left(\partial_t \phi + \Delta \phi \right) (B_t,t)   \,\diffd t\right] \\
& =  \E_{\mu_0} \big[ \phi(B_{t_*\wedge \tau},t_* \wedge \tau) - \phi(B_{0},0)    \big]. 
\end{align*}
Since $\phi = 0$ on $\partial D$, we conclude that
\begin{align*}
\int_0^{t_*} \int_{\R^d} \left(\partial_t \phi + \Delta \phi \right)(y,t) w(y,t)\,\diffd y \,\diffd t 
& = \E_{\mu_0} \big[\phi(B_{t_*},t_*) \indic{\tau \ge t_*} \big].
\end{align*}
This establishes \eqref{wweak1}.

The condition \eqref{wweak1} implies that the function $w \in L^\infty_{\text{loc}}(D)$ is actually a smooth function within $D$ and is a classical solution to $\partial_t w = \Delta w$ in $D$. This follows by a standard argument, as follows (for example, this is a particular case of the argument in \cite{JKO98}, starting at equation~(51) therein). Fix $(x_1,t_1) \in D$, and take $r > 0$ small enough that $\overline{\mathcal{C}_{4r}(x_1,t_1)} \subset D$, where $\mathcal{C}_r$ is the parabolic cylinder
\[
\mathcal{C}_r(x_1,t_1) = \big\{ (x,t) \; : \; \| x - x_1\| < r, \quad |t - t_1| < r^2 \big\}.
\] 
Let $\chi:D \to [0,1]$ be a smooth cut-off function that vanishes outside $\mathcal{C}_{3r}(x_1,t_1)$, such that $\chi = 1$ on $\mathcal{C}_{2r}(x_1,t_1)$. Let $(x_2,t_2) \in \mathcal{C}_{r}(x_1,t_1)$.  
Recall the definition of the transition density $\Phi$ in~\eqref{eq:Phitransitiondef}.
For $\delta > 0$, the function 
\[
\phi_\delta(y,t) = \chi(y,t) \Phi(y,t,x_2,t_2+\delta)
\] 
is smooth and compactly supported on $D\cap\{(x,t):t\le t_2\}$.  Using this test function in \eqref{wweak1}, and using the fact that $(\partial_t + \Delta_y) \Phi(y,t,x_2,t_2 + \delta) = 0$ for $t\le t_2$ and $y\in \R^d$, we obtain 
\begin{align}
\int_{\R^d}  \phi_\delta(y,t_2) w(y,t_2) \,\diffd y 
& = \int_0^{t_2} \int_{\R^d}  \bigg[ w(\partial_t\chi + \Delta \chi) \Phi(y,t, x_2, t_2 + \delta ) 
\nonumber \\ &\hphantom{= \int_0^{t_2} \int_{\R^d}  \bigg[} 
+    2 w \nabla \chi \cdot \nabla_{\!y} \Phi(y,t  ,x_2, t_2 + \delta)\bigg] \,\diffd y \,\diffd t. \nonumber
\end{align}
Recall from~\eqref{rhoupper} and the definition of $w$ in~\eqref{eq:wfromudef} that $w(y,t)\le (4\pi t)^{-d/2}$ for $y\in \R^d$, $t>0$. Hence the
functions $w(\partial_t \chi + \Delta \chi)$ and $2 w \nabla \chi$, which do not depend on $(x_2,t_2)$, are bounded, compactly supported, and vanish on $\mathcal{C}_{2r}(x_1,t_1)$ since $\chi = 1$ on that set. Moreover, the kernels $\Phi(y,t, x_2, t_2+\delta )$ and $\nabla_y \Phi(y,t, x_2, t_2+\delta)$ are bounded uniformly with respect to $(x_2,t_2) \in \mathcal{C}_{r}(x_1,t_1)$, $\delta \geq 0$, and $(y,t) \in (\R^d \times [0,t_2]) \setminus \mathcal{C}_{2r}(x_1,t_1)$. 
Then, letting $\delta\searrow0$, this implies that for each $t_2$ and almost every $x_2$ such that $(x_2,t_2) \in  \mathcal{C}_{r}(x_1,t_1)$, we have the identity
\[
w(x_2,t_2)  = \int \!\! \int_{(\R^d \times [0,t_2]) \backslash \mathcal{C}_{2r}(x_1,t_1)}  \bigg[ w(\partial_t\chi + \Delta \chi) \Phi(y,t, x_2, t_2) +   2 w \nabla \chi \cdot \nabla_{\!y} \Phi(y ,t, x_2, t_2 )\bigg] \,\diffd y \,\diffd t.
\] 
This integral expression varies smoothly with $(x_2,t_2) \in \mathcal{C}_{r}(x_1,t_1)$, which shows
that $w$ has a version which is infinitely differentiable at $(x_2,t_2)$.  Thus, $w$ is smooth on a neighbourhood of the arbitrarily chosen $(x_1,t_1) \in D$, and hence everywhere in $D$. Returning to \eqref{wweak1}, we can now integrate by parts, to conclude that
\begin{equation*}
0 = \int_0^\infty \int_{\R^d} \phi(y,t)( \partial_t w - \Delta w)(y,t)\,\diffd y \,\diffd t
\end{equation*}
holds for any $\phi \in C^\infty_0(D)$.
This implies $\partial_t w - \Delta w = 0$ inside $D$.

{\bf Step 2:} Next, we argue that the initial condition \eqref{t0conv} holds.  Let $\phi:\R^d \to \R$ be any bounded Lipschitz continuous function. Then for $t>0$,
\begin{align*}
\int_{\R^d} u(y,t) \phi(y) \,\diffd y - \int_{\R^d}  \phi(y) \,\mu_0(\diffd y) & = \int_{\R^d} \mu_0(\diffd x)\left(  \int_{\R^d}e^t \rho_t(x,y) \phi(y)\,\diffd y - \phi(x)  \right) \\
& = \int_{\R^d} \mu_0(\diffd x)\left( \int_{\R^d} e^t  \rho_t(x,y) \Big[\phi(y) - \phi(x)\Big]\,\diffd y  \right) \\
& \quad + \int_{\R^d} \mu_0(\diffd x)\,\phi(x) \left(  \int_{\R^d} e^t  \rho_t(x,y) \,\diffd y -1 \right) .
\end{align*}
Using \eqref{rhoupper}, we observe that
\begin{align*}
\left|\int_{\R^d} \mu_0(\diffd x)\left( \int_{\R^d} e^t  \rho_t(x,y) \Big[\phi(y) - \phi(x)\Big]\,\diffd y  \right)\right|
 & \leq \norm{\phi}_{\text{Lip}} e^t \int_{\R^d} \mu_0(\diffd x) \left(\int_{\R^d} \rho_t(x,y) \|y - x\| \,\diffd y \right)  \\
 & \leq \norm{\phi}_{\text{Lip}}e^t \int_{\R^d}\frac 1 {(4\pi t)^{d/2}} e^{-\frac{\|y\|^2}{4t}}\|y\|\,\diffd y\\
&\leq \norm{\phi}_{\text{Lip}} e^t C_d \sqrt t
\end{align*}
for some constant $C_d<\infty$. Note that for $x,y\in \R^d$ and $t>0$, $\int_{\R^d}\rho_t (x,y)\,\diffd y \le 1$. Hence,
using Lemma \ref{lem:rhomass} and the dominated convergence theorem, we also obtain
\[
\left|\int_{\R^d} \mu_0\,(\diffd x)\phi(x) \left(\int_{\R^d} e^t  \rho_t(x,y) \,\diffd y -1 \right)\right| \leq \norm{\phi}_{L^\infty}\int_{\R^d} \mu_0(\diffd x) \left|  \int_{\R^d} e^t  \rho_t(x,y) \,\diffd y -1 \right| \to 0
\]
as $t\to 0$.
Therefore, 
\[
\lim_{t \to 0} \left| \int_{\R^d} u(y,t) \phi(y) \,\diffd y - \int_{\R^d}  \phi(y) \,\mu_0(\diffd y)\right| =0
\]
holds for any bounded Lipschitz function $\phi$.  This proves \eqref{t0conv}.

{\bf Step 3:} We have already established that $u$ is smooth at points $(y,t)$ where $t > 0$ and $\|y\| < R_t$.  Next, we show that for every $t_0 > 0$ and $y_0\in \R^d$ with $\|y_0\|=R_{t_0}$,
\begin{equation}
\lim_{(y,t)\to (y_0,t_0)} u(y,t) = 0, \label{ucontin}
\end{equation}
from which the continuity condition~\eqref{continuity} and the boundary condition~\eqref{uboundary} follow.
For $0<s<t$, from
the Markov property
and the definition of $p(x,s,y,t)$ in~\eqref{pdensitydef}, we have that
\begin{align*}
u(y,t) & = e^t \int_{\R^d} \mu_0 (\diffd x') \int_{\R^d}\diffd x\, \rho_s (x',x)p(x,s,y,t)\\
&= e^{t - s}  \int_{\R^d} u(x,s) p(x,s,y,t) \,\diffd x.
\end{align*}
For $s > 0$, $t\mapsto R_t$ is continuous on $[s,\infty)$ and, using~\eqref{rhoupper}, $x \mapsto u(x,s)$ is a bounded, measurable function.  So, Lemma \ref{lem:timereverseFK} with $g(x)=u(x,s)$ implies that
\begin{equation} \label{eq:uytbackwards}
u(y,t)  = e^{t-s}\E_y\left[ u(B_{t-s}, s) \indic{\tau_{t}' \ge t - s} \right]
\end{equation}
for all a.e.~$y\in \R^d$. 
Since $u$ is smooth on the domain $D$, it follows that~\eqref{eq:uytbackwards} holds for all $y\in \R^d$ with $\|y\|<R_t$.
Fix $s\in (0,t_0)$.  Then by
Lemma \ref{lem:taucontin} (with $t_1 = s$), we have that $\E_y\left[ \indic{\tau_{t}' \ge t - s} \right] \to 0$ as $(y,t) \to (y_0,t_0)$. Since $y \mapsto u(y,s)$ is bounded, it follows that
\begin{equation*}
\lim_{(y,t)\to (y_0,t_0)} \E_y\left[ u(B_{t-s}, s) \indic{\tau_{t}' \ge t - s} \right] = 0.
\end{equation*}
Thus,~\eqref{ucontin} must hold.

{\bf Step 4:} Finally, we argue that $u$ satisfies the mass constraint~\eqref{umass}.
For $t>0$, by Lemma~\ref{lem:utilde}, and then since $\rho_t(x,\R^d) = \rho_t(x',\R^d)$ if $\|x\|=\|x'\|$, and since $\tilde \mu_0=\Gamma^* \mu_0$,
\begin{align*}
1 = \int_{\R^d} \tilde u(y,t) \,\diffd y & = e^t \int_{\R^d} \tilde \mu_0(\diffd x) \,\rho_t(x,\R^d) \\
& =  e^t \int_{\R^d} \mu_0(\diffd x)\, \rho_t(x,\R^d)\\
& =  \int_{\R^d} u(y,t) \,\diffd y.
\end{align*}
This establishes \eqref{umass}.

The proof of Proposition \ref{prop:ufb} is now complete.
\end{proof}

\begin{rmk} \label{rmk:ubounded}
If $u$ solves the free boundary problem \eqref{pbu}, then
 the Chapman-Kolmogorov equation and the fact that $\int_{\R^d} u(x,s)\,\diffd x = 1$ implies that for any $s > 0$ and $t \in [1,2]$,
\[
u(y,s + t) = e^{t} \int_{\R^d} u(x,s) p(x,s,y,s+t)\, \diffd x \leq e^{2} \int_{\R^d} u(x,s) \Phi(x,s,y,s + t) \,\diffd x \leq  \frac{e^2}{(4 \pi)^{d/2}}.
\]
This implies that $u$ is bounded by $\|u(\cdot,t)\|_{L^\infty} \leq  \frac{e^2}{(4 \pi)^{d/2}}$ for all $t \geq 1$.
\end{rmk}

\section{Proof of Theorem~\ref{thm:conv u}: Convergence to the steady state for \texorpdfstring{$u$}{u}} \label{sec:uconv}

Let $(u,R)$ solve the free boundary problem~\eqref{pbu} with initial condition $\mu_0$.
For $t>0$ and $x\ge 0$, let
\[
v(x,t) = \int_{\B(x)} u(y,t) \,\diffd y.
\]
By Lemma~\ref{lem:utilde}, $v$
solves the obstacle problem \eqref{pbv} with the initial condition $v_0(x) = \mu_0( \B(x))$. 
Also by Lemma~\ref{lem:utilde}, $R_t=\inf\{x:v(x,t)=1\}$ for all $t>0$.
We have already proved Theorem~\ref{thm:conv}, which shows that $R_t \to R_\infty$ and $v(\cdot,t) \to V(\cdot)$ uniformly as $t \to \infty$.  Hence, for any $x \ge 0$, 
\[
\int_{\B(x)} u(y,t) \,\diffd y \to V(x)  \quad \text{as $t \to \infty$}.
\]
However, this does not immediately imply $u(\cdot,t) \to U(\cdot)$.

We now prove Theorem~\ref{thm:conv u} in several steps.

\paragraph{Step 1:} Let $\tilde \mu_0 = \Gamma^* \mu_0$ denote the symmetrized initial condition (recall \eqref{GammaStardef}). Then by Lemma~\ref{lem:utilde}, the function
\[
\tilde u(y,t) = (\Gamma u)(y,t) = e^t \int_{\R^d} \tilde \mu_0(\diffd x)\, \rho_t(x,y)  
\]
satisfies the free boundary problem \eqref{pbu}, but with symmetrized initial condition $\tilde \mu_0$. Moreover,
\[
v(x,t) = \int_{\mathcal{B}(x)} \tilde u(y,t)\,\diffd y, \quad x > 0, \, t>0,
\]
and, by~\eqref{utildedef}, for $y\in \R^d\setminus \{0\}$ and $t>0$, $\tilde u(y,t) =\frac 1 {d \, \omega_d \|y\|^{d-1}}\partial_x v(\|y\|,t)$.  The fact that $v(\cdot,t) \to V(\cdot)$ in $L^1_\text{loc}$ as $t\to \infty$ implies, by Lemma~\ref{lem:weakconv}, that $\tilde u(\cdot,t) \to U(\cdot)$ weakly as measures on $\R^d$, as $t \to \infty$.  We claim that this convergence holds in $L^\infty$. 
Since Remark~\ref{rmk:ubounded} applies to $\tilde u$, as well, we have that $\tilde u$ is bounded on $\R^d \times [1,\infty)$.
Because $R_t > 0$ for $t > 0$ and $R_t \to R_\infty > 0$ as $t \to \infty$, we know that there exist $\epsilon > 0$
and $K<\infty$ such that $2\epsilon<R_t \le K$ for all $t \geq 1$. 
Thus, $\tilde u$ is a bounded solution to $\partial_t \tilde u = \Delta \tilde u + \tilde u$ on the cylinder $\mathcal{B}(2\epsilon) \times [1,\infty)$.  By standard parabolic regularity estimates, $\|\nabla \tilde u\|$ is uniformly bounded over $\mathcal{B}(\epsilon) \times [2,\infty)$.   By Proposition~\ref{prop:v is C1}, we also know that the family of functions $\{\partial_x v(\cdot,t)\}_{t \geq 2}$ is uniformly bounded and equicontinuous; hence $\{\tilde u(\cdot,t)\}_{t \geq 2}$ is equicontinuous on $\{ \|y\| > \epsilon/2 \}$, as well.   Combining these observations, we conclude that $\{\tilde u(\cdot,t)\}_{t \geq 2}$ is a uniformly bounded and equicontinuous family of functions on $\R^d$. If $\tilde u(\cdot,t)$ does not converge uniformly to $U$ as $t \to \infty$, then,
since $R_t\le K$ for all $t\ge 1$,
 there exist $\delta > 0$ and a diverging sequence of times $\{t_n\}_{n \geq 1}$ such that 
 \[ 
 \sup_{x\in \overline{\B(K)}}|\tilde u(x,t_n) - U(x) |> \delta
 \]
  for all $n$.  By the Arzel\`a–Ascoli Theorem, there must be a subsequence of times $\{t_{n_j}\}$ such that $\tilde u(\cdot,t_{n_j})$ converges uniformly on $\overline{\B(K)}$ to a limit. Since we already know that $\tilde u(\cdot,t) \to U$ weakly as measures, this limit must be $U$, a contradiction.  Thus, we conclude that $\tilde u$ converges uniformly to $U$:
\begin{equation} \label{eq:tildeuUconv}
\lim_{t \to \infty}  \big\|\tilde u(\cdot,t) - U(\cdot)\big\|_{L^\infty} = 0.
\end{equation}
Therefore, to prove that $u(\cdot,t) \to U(\cdot)$ uniformly as $t\to \infty$, it suffices to show that
\[
\lim_{t \to \infty} \big\|u(\cdot,t) - \tilde u(\cdot,t)\big\|_{L^\infty} = 0.
\]
The two functions $\tilde u$ and $u$ satisfy the same PDE and boundary condition, but with the initial measure $\tilde \mu_0$ (for $\tilde u$) being the symmetrization of the initial measure $\mu_0$ (for $u$). Hence, we must show that any solution of the free boundary problem becomes asymptotically spherically symmetric as $t \to \infty$. 

\paragraph{Step 2:} For $\epsilon > 0$, let $T_\epsilon \geq 1$ be large enough so that $|R_t - R_\infty| \leq \epsilon$ for all $t \geq T_\epsilon$. 
Recall the definition of $p(x,s,y,t)$ in~\eqref{pdensitydef}.
For any integer $n \geq 1$, 
by Proposition~\ref{prop:utilderho} and then by the Chapman-Kolmogorov equation,
$u$ and $\tilde u$ satisfy
\begin{align}
&u(y,n+1) -  \tilde u(y,n+1) \notag
\\&\qquad\qquad = e^{n+1}\left(\int_{\R^d} \mu_0(\diffd x)\,\rho_{n+1}(x,y)- \int_{\R^d} \tilde \mu_0(\diffd x)\,\rho_{n+1}(x,y) \right) \notag \\
&\qquad\qquad= e^1 \int_{\R^d}  \left( e^n \int_{\R^d}\mu_0(\diffd x')\,\rho_n(x',x)-e^n \int_{\R^d}\tilde \mu_0(\diffd x')\,\rho_n(x',x)\right)p(x,n,y,n+1)\,\diffd x \notag \\
&\qquad\qquad=e^1 \int_{\R^d} \big[u(x,n) - \tilde u(x,n)\big] p(x,n,y,n+1) \,\diffd x. \label{utildedif}
\end{align}
Recall that $p(x,n,y,n+1)$ is characterised by the relation
\[
\int_{A} p(x,n,y,n+1) \,\diffd y  = \P_{\!x} \big  ( B_1 \in A, \; \|B_t\| < R_{n + t}  \; \forall t \in (0,1) \big), 
\]
for all Borel sets $A \subseteq \R^d$. We now define the kernels $p^+_1(x,y)$ and $p^-_1(x,y)$ by the relations
\begin{equation}
\begin{aligned} 
\int_{A} p^+_1(x,y) \,\diffd y & = \P_{\!x} \big  ( B_1 \in A, \; \|B_t\| < R_{\infty} + \epsilon \; \forall t \in (0,1) \big) \\
\text{and }\quad \int_{A} p^-_1(x,y) \,\diffd y & = \P_{\!x} \big  ( B_1 \in A, \; \|B_t\| < R_{\infty} - \epsilon \; \forall t \in (0,1) \big)
\end{aligned} \label{eq:p+-def}
\end{equation}
for all Borel sets $A \subseteq \R^d$.  Therefore, for $n \geq T_\epsilon \geq 1$, we have
\begin{equation}
p^-_1(x,y) \leq p(x,n,y,n+1) \leq p^+_1(x,y) \label{rhopmcomp}
\end{equation}
for all $x,y \in \R^d$. (Note: $p^\pm_1(x,y) = 0$ for $\|x\| > R_\infty \pm \epsilon$ or $\|y\| > R_\infty \pm \epsilon$.) From \eqref{utildedif} we can now write the difference $u - \tilde u$ as
\begin{align}
u(y,n+1) -  \tilde u(y,n+1) & =   e^1\int_{\B(R_\infty + \epsilon)} \big[u(x,n) - \tilde u(x,n)\big] p^+_1(x,y) \,\diffd x \nonumber \\
& \quad +  e^1 \int_{\B(R_\infty+\epsilon)}\big [u(x,n) - \tilde u(x,n)\big]  \big[p(x,n,y,n+1) - p^+_1(x,y)\big] \,\diffd x. \label{utildedif2}
\end{align}

\paragraph{Step 3:}  Consider the first integral on the right hand side of~\eqref{utildedif2}: for $y\in \R^d$, let
\[
W_n(y) = e^1 \int_{\B(R_\infty + \epsilon)} \big[u(x,n) - \tilde u(x,n)\big] p^+_1(x,y) \,\diffd x. 
\]
We claim that there is a constant $\alpha > 0$ such that for $\epsilon>0$ sufficiently small,
\begin{equation}
\norm{W_n(\cdot)}_{L^2} \leq e^{-\alpha} \norm{ u(\cdot,n) - \tilde u(\cdot,n)}_{L^2} \label{Wnbound}
\end{equation}
holds for all $n$.  By the definition of $p^+_1$, we observe that $W_n(y) = e^1 w(y,n+1)$ where for $t \in [n,n+1]$, $w(y,t)$ is the unique solution to 
\begin{alignat}{2}
& \partial_t w = \Delta w, \quad && \text{for }\|y\| < R_\infty + \epsilon, \quad t > n, \label{wpde1} \\
& w(y,t) = 0, \quad && \text{for }\|y\| \geq R_\infty + \epsilon, \quad t \geq n,    \\
& w(y,n) = u(y,n) - \tilde u(y,n),  \quad && \text{for }\|y\| < R_\infty + \epsilon. \label{wpde3}
\end{alignat}

The claim \eqref{Wnbound} is an immediate consequence of the following spectral-gap estimate:

\begin{lem} \label{lem:L2exp}
There exists $\alpha > 0$ such that for all $\epsilon > 0$ sufficiently small, for all integer $n\ge 1$,
\[
\norm{ w(\cdot, n+1) }_{L^2}  \leq e^{- (1 + \alpha)} \norm{ u(\cdot,n) - \tilde u(\cdot,n)}_{L^2}.
\]
\end{lem}		

\begin{proof}
By Theorem 7.1.3 of \cite{Evans2010}, the function $w(x,t)$ may be expanded as a series in terms of eigenfunctions of the Laplacian:
\[
w(x,n+t) = \sum_{k=1}^\infty a_k e^{- t \lambda^\epsilon_k} U^\epsilon_k(x), \quad t \in [0,1],\quad \|x\| \leq R_\infty + \epsilon,
\]
where the partial sums of the series converge weakly in $L^2\big([0,1];H^1_0\big)$. Here $\big\{\big(U_k^\epsilon(x),\lambda^\epsilon_k\big) \big\}_{k \geq 1}$ denote the Dirichlet eigenfunctions and eigenvalues for $-\Delta$ on the ball $\big\{\|x\| \leq R_\infty + \epsilon \big\}$:
\begin{alignat}{2}
 - \Delta U^{\epsilon}_k &= \lambda_k^\epsilon U_k^{\epsilon}, \quad && \text{for }\|x\| < R_\infty + \epsilon, \nonumber \\
U_k^{\epsilon}(x)& = 0, \quad &&\text{for }\|x\| = R_\infty + \epsilon. \nonumber 
\end{alignat}
These may be normalised to form an orthonormal basis in $L^2$ (orthogonal in $H^1_0$), and we may assume $0 < \lambda_1^\epsilon < \lambda_2^\epsilon \leq \dots$  (c.f. \cite{Evans2010}, Section 6.5.1). With $\epsilon = 0$, the principal eigenfunction is $U^0_1(x) = \|U\|^{-1}_{L^2} U(x)$, and the principal eigenvalue is $\lambda_1^0 = 1$, see~\eqref{eqU}. (Recall that we normalized $U$ and $U^0_1$ by $\|U\|_{L^1}=1$ and $\|U^0_1\|_{L^2}=1$, hence the constant scaling factor.)  By scaling, we have 
\begin{equation} \label{eq:lamUscaling}
\lambda^\epsilon_k  = \bigg(\frac{R_\infty }{R_\infty + \epsilon}\bigg)^2 \lambda^0_k\qquad\text{and}\qquad U^{\epsilon}_k(x) = \left(\frac{R_\infty}{R_\infty + \epsilon}\right)^{d/2} U^0_k\bigg(x\frac{R_{\infty}}{R_\infty + \epsilon}\bigg)
\end{equation}
 for all $k\in \N$ and $\epsilon > 0$. In particular, the principal eigenfunctions $U_1^\epsilon(x)$ depend only on the radial coordinate $r = \|x\|$. Therefore, $U_1^\epsilon(\cdot)$ and $w(\cdot,n)$ must be orthogonal in $L^2$ because the integral of $w(\cdot,n)$ over each spherical shell vanishes: 
\begin{align}
\int_{\B({R_\infty + \epsilon})} U_1^\epsilon(x) w(x,n) \,\diffd x  & = \int_{\B({R_\infty + \epsilon})} U_1^\epsilon(x) \big[u(x,n) - \tilde u(x,n)\big] \,\diffd x \nonumber \\
& = \int_0^{R_\infty + \epsilon} U_1^\epsilon(r) \bigg( \int_{\partial\B(r)}  \big[u(y,n) - \tilde u(y,n)\big] \,\diffd S(y) \bigg) \,\diffd r = 0. \label{orthogW}
\end{align}
Since the principal eigenvalue $\lambda_1^\epsilon$ is simple, and since $w(\cdot,n)$ is orthogonal to the principal eigenfunction $U_1^\epsilon$, we must have $a_1 = 0$, so that 
\[
w(x,n+t) = \sum_{k=2}^\infty a_k e^{- t \lambda^\epsilon_k} U^\epsilon_k(x), \quad t \in [0,1],\quad\|x\| \leq R_\infty + \epsilon,
\]
and
\[
\norm{w(\cdot,n+t)}_{L^2}^2 = \sum_{k \geq 2} e^{- 2 t \lambda_k^\epsilon } a_k^2 \leq e^{- 2t \lambda_2^\epsilon } \sum_{k \geq 2}  a_k^2 = e^{- 2 t  \lambda_2^\epsilon } \norm{w(\cdot,n)}_{L^2}^2.
\]
Let $\alpha =(\lambda^0_2-\lambda^0_1)/2=(\lambda_2^0-1)/2>0$.
By the scaling relation in~\eqref{eq:lamUscaling}, $\lambda_2^\epsilon = [R_\infty /(R_\infty + \epsilon)]^2 \lambda^0_2 = [R_\infty /(R_\infty + \epsilon)]^2 (1 + 2\alpha)$. Hence, $\lambda_2^\epsilon \geq 1 + \alpha$ if $\epsilon$ is small enough. Therefore,
\[
\norm{w(\cdot,n+1)}_{L^2} \leq e^{- (1 + \alpha)} \norm{w(\cdot,n)}_{L^2} 
\]
holds for all $\epsilon > 0$ sufficiently small.
\end{proof}

\paragraph{Step 4:} Now we estimate the second integral on the right hand side of~\eqref{utildedif2}: for $y\in \R^d$, let
\[
H_n(y) =  e^1\int_{\B(R_\infty + \epsilon)} \big[u(x,n) - \tilde u(x,n)\big]
\big[p(x,n,y,n+1) - p^+_1(x,y)\big] \,\diffd x.
\]
Then by~\eqref{rhopmcomp}, for $n\ge T_\epsilon$,
\[
\big|H_n(y)\big| \le  e^1\int_{\R^d} \big|u(x,n) - \tilde u(x,n)\big|
\;\big|p^+_1(x,y) - p^-_1(x,y)\big| \,\diffd x.
\]
By the Cauchy-Schwarz inequality, applied to $\big|u(x,n) - \tilde u(x,n)\big| \;\big |p_1^+(x,y) - p^-_1(x,y)\big|{}^{1/2} $ and $\big |p_1^+(x,y) - p^-_1(x,y)\big|{}^{1/2}$, 
for $n\ge T_\epsilon$ we have that
\begin{align*}
\big|H_n(y)\big|^2 
& \leq  e^2\int_{\R^d} \big|u(x,n) - \tilde u(x,n)\big|^2\, \big|p^+_1(x,y) - p^-_1(x,y)\big| \,\diffd x
\times
 \int_{\R^d} \big|p^+_1(x,y) - p^-_1(x,y)\big| \,\diffd x  \\
& \leq  e^2\int_{\R^d} \big|u(x,n) - \tilde u(x,n)\big|^2\, \big|p^+_1(x,y) - p^-_1(x,y)\big| \,\diffd x
 \times
\sup_{z\in\R^d} \int_{\R^d} \big|p^+_1(x,z) - p^-_1(x,z)\big| \,\diffd x.
\end{align*}
Integrating in $y$ and using Fubini's theorem, we conclude that for $n\ge T_\epsilon$,
\begin{align}
\norm{H_n(\cdot)}_{L^2}
&\leq e^1\norm{u(\cdot,n) - \tilde u(\cdot,n)}_{L^2} \notag \\
&\qquad  \times \left[ \sup_{x\in \R^d} \int_{\R^d} \big|p^+_1(x,y)
- p^-_1(x,y)\big| \,\diffd y \right]^{1/2} 
\left[ \sup_{z\in \R^d}  \int_{\R^d} \big|p^+_1(x,z) - p^-_1(x,z)\big| \,\diffd x \right]^{1/2} \notag \\
&= e^1\norm{u(\cdot,n) - \tilde u(\cdot,n)}_{L^2} \ \sup_{x\in \R^d} \int_{\R^d}\big |p^+_1(x,y) - p^-_1(x,y)\big| \,\diffd y \label{eq:HnL2bound}
\end{align}
since $p^{\pm}_1(x,z)=p^{\pm}_1(z,x)$ $\forall x,z\in \R^d$.

The following lemma will now allow us to bound $\|H_n\|_{L^2}$.
\begin{lem} \label{lem:p+p-}
As $\epsilon \to 0$,
\[
\sup_{x\in \R^d} \int_{\R^d} \big|p^+_1(x,y) - p^-_1(x,y)\big| \,\diffd y \to  0.
\]
\end{lem}

\begin{proof}
By the definition of $p_1^\pm$ in~\eqref{eq:p+-def} and by~\eqref{rhopmcomp}, for $x\in \R^d$,
\begin{align}
\int_{\R^d} \big|p^+_1(x,y) - p^-_1(x,y)\big| \,\diffd y 
& =
\P_{\!x} \big  (\|B_t\| < R_{\infty} + \epsilon \; \forall t \in (0,1) \big) 
- \P_{\!x} \big  (\|B_t\| < R_{\infty} - \epsilon \; \forall t \in (0,1) \big) 
\nonumber \\
& = f(x,1,R_\infty + \epsilon) - f(x,1,R_\infty - \epsilon), \label{cdfBmax}
\end{align}
where
\begin{align}
f(x,t,z)=\P_{\!x}\left(\| B_s\| <z \;\forall s\in (0,t)\right) = \P_{\!x}\Big(\max_{s\in[0,t]} \| B_s\| <z\Big). 
\end{align}
The second equality follows from the fact that $\P_x(\|B_t\| = z) = 0$. Thus, $f(x,t,z)$ is the distribution function for the running maximum of $\|B\|$ up to time $t$.  We claim that the right hand side of~\eqref{cdfBmax} is $\mathcal{O}(\epsilon)$, uniformly in $x$. By the scaling properties of Brownian motion, we have
\[
f(x,t,z) = f(x/z,t/z^2,1).
\]
The function $f(x,t,1)$ is the solution to the Cauchy problem for the heat equation:
\begin{align*}
& \partial_t f =\Delta_x f,  &&\kern-10em\|x\|<1, \, t>0, \\
& f(x,t,1)=0,&&\kern-10em \|x\| \geq 1, \, t > 0,\\
& f(x,0,1)=\indic{\|x\|<1}.
\end{align*}
Standard estimates for the heat equation imply that for any $\delta > 0$, there are constants $C_{1,\delta}$ and $C_{2,\delta}$ such that
\[
\sup_{\substack{\|x\| < 1 \\t \geq \delta} } \big\|\nabla_{\!x}
f(x,t,1)\big\| < C_{1,\delta} \quad \quad \text{and} \quad \quad
\sup_{\substack{\|x\| < 1  \\t \geq \delta} }  \big|\partial_t f(x,t,1)\big| < C_{2,\delta}. 
\]
Therefore, for $z \in [R_\infty - \epsilon, R_\infty + \epsilon]$ and $\|x\|<z$, the derivative $\partial_z f(x,1,z)$ is bounded by
\begin{align*}
\big|\partial_z f(x,1,z)\big| = \big|\frac{\diffd}{\diffd z} f(x/z,1/z^2,1)\big|
& = \left| \left(\frac{1}{z^2} x \cdot \nabla_{\!x} f + \frac{2}{z^3} \partial_t f\right) (x/z,1/z^2,1)\right| \nonumber \\
& \leq \frac{1}{z} C_{1,\delta} + \frac{2}{z^3} C_{2,\delta}
\end{align*}
with $\delta = (R_\infty +\epsilon)^{-2} > 0$. Thus, $|\partial_z
f(x,1,z)|$ is bounded uniformly over $x \in \R^d$, $\|x\| < z$, $z \in [R_\infty - \epsilon, R_\infty + \epsilon]$.  
Since $f(x,1,z)=0$ for $\|x\|\ge z$, this
 immediately implies that the right hand side of~\eqref{cdfBmax} is $\mathcal{O}(\epsilon)$; in particular,
\[
\lim_{\epsilon \to 0} \sup_{x \in \R^d} |f(x,1,R_\infty + \epsilon) - f(x,1,R_\infty - \epsilon)| = 0.
\]
This completes the proof.
\end{proof}
By~\eqref{eq:HnL2bound} and Lemma~\ref{lem:p+p-}, we have that for any $\delta>0$, there is $\epsilon_\delta>0$ such that
\begin{equation} \label{eq:Hnbound}
\norm{H_n(\cdot)}_{L^2} \leq \delta \norm{u(\cdot,n) - \tilde u(\cdot,n)}_{L^2}
\end{equation}
holds for all $n \geq T_\epsilon$ if $\epsilon < \epsilon_\delta$.

\paragraph{Step 5:} We now combine the preceding steps with
\eqref{utildedif2}. Let $\delta = e^{- \alpha/2} - e^{-\alpha} > 0$, where $\alpha>0$ is defined in Lemma~\ref{lem:L2exp}.  Taking $\epsilon_\delta$ sufficiently small that Lemma~\ref{lem:L2exp} holds with $\epsilon=\epsilon_\delta$, we conclude that for all $n \geq T_{\epsilon_\delta}$,
\begin{align*}
\norm{u(\cdot,n+1) - \tilde u(\cdot,n+1)}_{L^2} & \leq \norm{W_n}_{L^2}
+\norm{H_n}_{L^2} \nonumber \\
& \leq (e^{- \alpha} + \delta) \norm{u(\cdot,n) - \tilde u(\cdot,n)}_{L^2} \nonumber \\
& = e^{- \alpha/2} \norm{u(\cdot,n) - \tilde u(\cdot,n)}_{L^2},
\end{align*}
where the second inequality follows by~\eqref{Wnbound} and~\eqref{eq:Hnbound}.
Note that $u(\cdot,t)$ and $\tilde u(\cdot,t)$ are continuous and compactly supported on $\R^d$ for each $t>0$, and so
$\|u(\cdot,n)-\tilde u(\cdot, n)\|_{L^2}<\infty$ for each $n\in \N$.
Hence there is a constant $C_1<\infty$ such that
\begin{equation}
\norm{u(\cdot,n+1) - \tilde u(\cdot,n+1)}_{L^2} \leq C_1 e^{- n \alpha /2} \label{udifndecay}
\end{equation}
holds for all $n \geq 1$. 
Take $t\ge 2$ sufficiently large that $R_s \le R_\infty +1$ for all $s\ge t-2$.
 Now letting $n=\lfloor t \rfloor -1$, by the same argument as in~\eqref{utildedif} we have for all $y \in \R^d$:
\begin{align*}
\big|u(y,t)-\tilde u(y,t)\big|
&= e^{t-n}\left| \int_{\R^d}\big(u(x,n)-\tilde u(x,n)\big)p(x,n,y,t)\,\diffd x \right|\\
&\le e^2 \int_{\B(R_\infty +1)}\big|u(x,n)-\tilde u(x,n)\big|\, \big(4\pi (t-n)\big)^{-d/2}\, \diffd x \\
&\le e^2 (4\pi )^{-d/2} (R_\infty +1)^{d/2} \omega_d^{1/2}\big\|u(\cdot,n)-\tilde u(\cdot, n)\big\|_{L^2},
\end{align*}
where the second line follows since $R_n \le R_\infty +1$ and $p(x,n,y,t)\le \Phi(x,n,y,t) \le (4\pi (t-n))^{-d/2} $ $\forall x \in \R^d$, and the last line follows by Jensen's inequality and since $t-n \ge 1$.
By~\eqref{udifndecay}, it follows that there exists a constant $C_2<\infty$ such that 
\[
\norm{u(\cdot,t) - \tilde u(\cdot,t)}_{L^\infty} \leq C_2 e^{- t \alpha /2}, 
\]
for all $t$ large enough.
By~\eqref{eq:tildeuUconv}, we conclude that
$$
\lim_{t \to \infty}  \|u(\cdot,t) - U(\cdot)\|_{L^\infty} = 0.
$$
This completes the proof of Theorem~\ref{thm:conv u}.

\section{Appendix} \label{appendix}

\begin{proof}[Proof of Lemma \ref{bounds int g}]
Recall the definition of $\Phi(z_1,z_2,t)$ in~\eqref{eq:PhiSec3def}.
By the scaling property of $\Phi$ and the definition of $G$ in~\eqref{Gdef},
for $t,x,y>0$, 
\begin{equation}
\label{scaling gG}
G(y,x,t)=\frac1{\sqrt t} G\left(\frac{y}{\sqrt t},\frac x{\sqrt t},1\right).
\end{equation}
Hence it is sufficient to prove the integral bounds in~\eqref{eq:Gintbounds} and~\eqref{eq:Gintbound2} for $t=1$. Indeed, by the scaling relation~\eqref{scaling gG},
$$\int_0^\infty\diffd y\, G(y,x,t)=
\int_0^\infty\diffd y\, \frac1{\sqrt t} G\Big(\frac y{\sqrt t},\frac x {\sqrt t},1\Big)=
\int_0^\infty\diffd y\, G\Big(y,\frac x {\sqrt t},1\Big),$$
and similarly, for $y_0\ge 0$,
\begin{align}
& \int_{y_0}^\infty\diffd y\, \big|\partial_x G(y,x,t)\big|=
\frac1{\sqrt t} \int_{y_0/\sqrt{t}}^\infty\diffd y\, \Big|\partial_x G\Big(y,\frac x {\sqrt t},1\Big)\Big|,
\label{dxGscale1} \\
& \int_0^\infty\diffd y\, \big|\partial_x^2 G(y,x,t)\big|=
\frac1{t} \int_0^\infty\diffd y\, \Big|\partial_x^2 G\Big(y,\frac x {\sqrt t},1\Big)\Big|.
\end{align}

The first inequality in \eqref{eq:Gintbounds} is immediate from \eqref{Gyintegral}:
$$
\int_0^\infty G(y,x,1) \,\diffd y =  \P\big( \|B_1\| < x \, \big|\, B_0=0 \big)
= \int_{\B(x)} \Phi(0,z,1)\,\diffd z\le \omega_d x^d (4\pi)^{-d/2},
$$
where $\omega_d$ is the volume of the $d$-dimensional unit ball and where the last inequality follows by bounding $\Phi(0,z,1)$ by $ (4\pi)^{-d/2}$.
Furthermore, this first integral is a probability and is therefore smaller than 1.

We now turn to the second and third inequalities in~\eqref{eq:Gintbounds}.
From~\eqref{Gdef}, we have $G=-\partial_y w$ and so we can write $\partial_xG = -\partial_y (\partial_x w)$ and $\partial_x^2G = - \partial_x\partial_y (\partial_x w)$. 
Recall \eqref{wdef} and write
\begin{equation}
\partial_x w(y,x,1)
= \int_{\partial\B(x)} \Phi(y \,\mathbf e_1,z,1) \, \diffd S(z)
=  \frac{e^{- \frac{x^2 + y^2}{4}}}{ (4 \pi )^{d/2}} \int_{\partial\B(x)}
e^{\frac{ y {\bf e}_1 \cdot z}{2}} \, \diffd S(z),
\label{dxw}
\end{equation}
where $\mathbf e_1$ is an arbitrary fixed unit vector.

We first consider the case $d\ge2$.
Integrate in spherical coordinates: $\theta$ is the angle between $z$ and
$\mathbf e_1$,  $\omega_d$ is the volume of the $d$-dimensional unit ball, and
$C$ is an arbitrary constant depending only on $d$ which can change from line to line.
We obtain that for $x,y>0$,
\begin{align}
\partial_x w(y,x,1)&=
\frac{(d-1)\omega_{d-1}}{ (4 \pi )^{d/2}}{e^{- \frac{x^2 +y^2}{4}}} \int_0^\pi  x^{d-1}(\sin\theta)^{d-2} e^{\frac{x y \cos\theta}{2}}\,\diffd\theta
\notag\\&=
C  x^{d-1}\int_0^\pi (\sin\theta)^{d-2} {e^{- \frac{(x\cos\theta -y)^2}{4}}}e^{-\frac{x^2\sin^2\theta }{4}}
\,\diffd\theta .
\label{dxw2}
\end{align}
(Note: if $d > 2$, $(d-1)\omega_{d-1}$ is the $(d-2)$-dimensional surface measure of the unit sphere in dimension $d-1$; if $d = 2$, this constant is $2$.) Then, using that $\partial_x G= -\partial_y(\partial_x w)$,
\begin{align}
\partial_x G(y,x,1)
=C  x^{d-1} \int_0^\pi \diffd\theta\, (\sin\theta)^{d-2} \big(y-x\cos\theta\big)
{e^{- \frac{(x\cos\theta -y)^2}{4}}}e^{-\frac{x^2\sin^2\theta }{4}}.
\label{dXG}
\end{align}
Therefore,
\begin{align}
\int_0^\infty\diffd y\,\big|\partial_x G(y,x,1)\big| 
&\le C  x^{d-1}
 \int_0^\pi \diffd\theta\,(\sin\theta)^{d-2} e^{-\frac{x^2\sin^2\theta }{4}}
\int_0^\infty \diffd y \,\big|y-x\cos\theta\big|{e^{- \frac{(x\cos\theta -y)^2}{4}}} \notag 
\\&\le C x^{d-1}
 \int_0^\pi \diffd\theta\, (\sin\theta)^{d-2} e^{-\frac{x^2\sin^2\theta }{4}},
\label{noy}
\end{align}
where the last expression was obtained by extending the integral in $y$ to an integral over $\R$.

We now bound the right hand side of~\eqref{noy} in two different ways. First, bounding the integrand by 1 gives
$$\int_0^\infty\diffd y\,\big|\partial_x G(y,x,1)\big|\le C  x^{d-1}.$$
Second, we write the integral from 0 to $\pi$ in \eqref{noy} as twice the integral from 0 to $\pi/2$. On the interval $[0,\pi/2]$we have $\frac2\pi\theta\le\sin\theta\le\theta$; hence
$$\int_0^\infty\diffd y\,\big|\partial_x G(y,x,1)\big|\le C  x^{d-1}
 \int_0^\infty \diffd\theta\, \theta^{d-2} e^{-\frac{x^2\theta ^2}{\pi^2}}
\le C.
$$
Then, combining both bounds, for $x>0$,
$$\int_0^\infty\diffd y\,\big|\partial_x G(y,x,1)\big|\le C \min\Big(1,x^{d-1}\Big),$$
which, by \eqref{dxGscale1} with $(x/\sqrt{t})$ in place of $x$, gives us the second inequality in \eqref{eq:Gintbounds}, for $d \geq 2$.

We now turn to the third inequality in \eqref{eq:Gintbounds} in the case $d\ge 2$. Differentiate \eqref{dXG} to obtain
\begin{multline*}
\partial_x^2 G(y,x,1)=
C  x^{d-1} \int_0^\pi \diffd\theta\, (\sin\theta)^{d-2} 
{e^{- \frac{(x\cos\theta -y)^2}{4}}}e^{-\frac{x^2\sin^2\theta }{4}}
\\
\times
\bigg[-\cos\theta +(y-x\cos\theta)\bigg(\frac{d-1}x+\frac{y-x\cos\theta}2-x\frac{\sin^2\theta}2\bigg)\bigg],
\end{multline*}
and therefore
\begin{multline*}
\big|\partial_x^2 G(y,x,1)\big|\le
C  x^{d-1} \int_0^\pi \diffd\theta\, (\sin\theta)^{d-2} 
{e^{- \frac{(x\cos\theta -y)^2}{4}}}e^{-\frac{x^2\sin^2\theta }{4}}
\\
\times
\bigg[1 +|y-x\cos\theta|\bigg(\frac{d-1}x+\frac{|y-x\cos\theta|}2+x\frac{\sin^2\theta}2\bigg)\bigg].
\end{multline*}
By integrating in $y$, and then extending the integral in $y$ on the right hand side to an 
integral over $\R$,
$$
\int_0^\infty\big|\partial_x^2 G(y,x,1)\big|\,\diffd y\le
C  x^{d-1} \int_0^\pi \diffd\theta\, (\sin\theta)^{d-2} 
e^{-\frac{x^2\sin^2\theta }{4}}
\big[1 +x^{-1}+x\sin^2\theta\big].
$$
We bound this expression in two ways, as we bounded \eqref{noy}, to obtain
$$
\int_0^\infty\big|\partial_x^2 G(y,x,1)\big|\,\diffd y\le C (x^{d}+x^{d-1}+x^{d-2})
\qquad\text{and}\qquad
\int_0^\infty\big|\partial_x^2 G(y,x,1)\big|\,\diffd y\le C (1+x^{-1}).
$$
Therefore, for $d\ge2$, for $x>0$,
$$
\int_0^\infty\big|\partial_x^2 G(y,x,1)\big|\,\diffd y\le C \min\big(1,x^{d-2}\big),
$$
which is a more precise bound (for $d\ge2$) than the third inequality in \eqref{eq:Gintbounds}.

To complete the proof of the bounds in~\eqref{eq:Gintbounds}, it
remains to prove the second and third inequalities in the case $d=1$. 
For $d=1$,~\eqref{dxw} reads
\begin{equation} \label{dxw2d=1}
\partial_x w(y,x,1)
= \int_{\partial \B(x)} \Phi(y \,\mathbf e_1,z,1) \, \diffd S(z)  =  \frac{e^{- \frac{x^2 + y^2}{4}}}{ \sqrt{4 \pi }} \Big[e^{\frac{yx}2}+e^{-\frac{yx}2}\Big]
=C\Big[e^{-\frac{(x-y)^2}4}+e^{-\frac{(x+y)^2}4}\Big],
\end{equation}
and \eqref{dXG} becomes
\begin{equation}
\label{dxGd=1}
\partial_x G(y,x,1)=-\partial_x\partial_y w(y,x,1)=
C  \Big[(y-x) e^{-\frac{(x-y)^2}4}+(y+x)e^{-\frac{(x+y)^2}4}\Big].
\end{equation}
Then for $x>0$,
$$
\int_0^\infty \diffd y \, \big|\partial_x G(y,x,1)\big|\le
\int_{-\infty}^\infty \diffd y \,  C  \Big[|y-x| e^{-\frac{(x-y)^2}4}+(y+x)e^{-\frac{(x+y)^2}4}\Big],
$$
which is bounded by a constant. Similarly,
$$\big|\partial_x^2 G(y,x,1)\big|\le
C  \Big[\Big(1+\tfrac12(y-x)^2\Big) e^{-\frac{(x-y)^2}4}+\Big(1+\tfrac12(y+x)^2\Big)e^{-\frac{(x+y)^2}4}\Big],
$$
and the right hand side integrated in $y$ over $\R$ is also bounded by a constant.
This completes the proof of~\eqref{eq:Gintbounds}.

We now turn to \eqref{eq:Gintbound2}. For $d\ge2$, notice from \eqref{dXG} that $\partial_x G(y,x,1)=-\partial_y\partial_xw(y,x,1)>0$ for $x<y$. Then, for $x<y_0$,
\begin{equation}\label{conc21}
 \int_{y_0}^\infty |\partial_x G(y,x,1)|\,\diffd y
=\int_{y_0}^\infty  \partial_x G(y,x,1) \,\diffd y=\partial_x w(y_0,x,1)
\le C x^{d-1} e^{-\frac{(y_0-x)^2}4},
\end{equation}
where we used~\eqref{dxw2} for the final inequality. For $d=1$, we use~\eqref{dxGd=1} to notice that $\partial_x G(y,x,1)>0$ for $x<y$. Then, for $x<y_0$,~\eqref{conc21} still holds, this time using~\eqref{dxw2d=1}.
Hence~\eqref{conc21} holds for any dimension, and \eqref{eq:Gintbound2} follows by the scaling property.

To derive \eqref{Grint1}, we shall integrate the PDE in~\eqref{eqG} in $x$, and use integration by parts. By~\eqref{Gdef}, we have that for $y,t>0$ fixed, $G(y,x,t)=\mathcal O(x^d)$ as $x\to 0$.
Therefore, $G(y,x,t)\to 0$ and $\frac{d-1}x G(y,x,t)\to 0$ as $x\to 0$.  We can write the formula for $w$ in~\eqref{wdef} as
\[
w(y,x,t)= \P( \|B_t + y \mathbf e_1 \| < x \;|\; B_0 = 0) = \int_{\{z:\|z+y\mathbf e_1\|<x\}}\Phi(0,z,t) \, \diffd z.
\]
It follows that
\begin{align} \label{eq:Galtformula}
G(y,x,t)=-\partial_y w(y,x,t)&= \int_{\{z:\|z+y\mathbf e_1\|=x,\, (z+y\mathbf e_1)\cdot \mathbf e_1>0\}}\Phi(0,z,t) \, \diffd S(z) \notag \\
&\quad -\int_{\{z:\|z+y\mathbf e_1\|=x,\, (z+y\mathbf e_1)\cdot \mathbf e_1<0\}}\Phi(0,z,t) \, \diffd S(z).
\end{align}
Since $\Phi(0,z,t)$ is decreasing in $\|z\|$, \eqref{eq:Galtformula} implies that
\[
|G(x,y,t)| \leq |\partial \B(x)| \max_{z \in \partial \B(x - y)} \Phi(0,z,t) = x^{d-1} d \omega_d \frac{e^{- \frac{(x-y)^2}{4t}}}{(4 \pi t)^{d/2}}, \quad x, y > 0, \quad t > 0.
\] 
In particular, for $y,t>0$ fixed, $G(y,x,t) \to 0$ as $x\to \infty$. By the scaling relation in~\eqref{scaling gG} and the expression for $\partial_x G$ in~\eqref{dXG}, we also have that 
for $y,t>0$ fixed, $\partial_x G(y,x,t)\to 0$ as $x\to \infty$. 
Hence, by integrating~\eqref{eqG} in $x$, and using integration by parts,
\[
\partial_t  \int_0^\infty \diffd x \,  G(y,x,t) = - \int_0^\infty \diffd x  \, \frac{d-1}{x^2} G(y,x,t) - \partial_x G(y,0,t) \leq 0
\]
since $\partial_x G(y,0,t)\ge 0$ by~\eqref{eqG}.
Thus, $\int_0^\infty \diffd x \,  G(y,x,t)$ is non-increasing in $t$,
and so~\eqref{Grint1} follows from the initial condition in~\eqref{eqG}.

The inequality \eqref{eq:Gv0bound} is a direct consequence of the first inequality in \eqref{eq:Gintbounds}.

Finally, the statements about the convergence \eqref{GintConvv} holding in $L^p_\text{loc}$ or locally uniformly as $t\searrow 0$ follow from standard arguments (see e.g. \cite{Evans2010}, Appendix C.5).  
\end{proof}

\begin{proof}[Proof of Proposition \ref{FK}]
Fix $(x,t)\in \Omega$. For $s\in[0,\tau]$, let
\[
M_s=w(X_s,t-s) e^{I_s},
\qquad \text{where }\quad I_s=\int_0^s g(X_r,t-r) \, \diffd r.
\]
Since $w\in C^{2,1}(\Omega)$, we apply 
It\^o's formula for $s< \tau$ (with no leading $\frac12$ in front of the $\partial_x^2$
term because $(W_s)_{s\geq 0}$ has diffusivity $\sqrt 2$) to obtain
\begin{align*}
\diffd M_s  
&= \partial_x w(X_s,t-s)e^{I_s} \,\diffd X_s +  \partial^2_x w(X_s,t-s)e^{I_s} \,\diffd s-  \partial_t w(X_s,t-s)e^{I_s} \,\diffd s\\
&\qquad + w(X_s,t-s)e^{I_s}g(X_s,t-s) \,\diffd s
\\
&=
\partial_x  w(X_s,t-s)e^{I_s}\,\diffd W_s,
\end{align*}
where we used~\eqref{processX} and~\eqref{dirieq} in the last line, since $(X_s,t-s)\in \Omega$
for $s< \tau$. We see that $(M_s)_{s < \tau}$ is a local martingale, and, as it is bounded,
it is therefore a martingale.

For $\delta>0$, let
$$\tau^\delta =\inf\{s\geq 0:\text{dist}((X_s,t-s),\partial \Omega)\leq \delta \},
$$
the first
time that $(X_s,t-s)$ is at a distance
$\delta$ from $\partial\Omega$.
For $(x,t)\in \Omega$,
since $\tau^\delta<\tau$,
by the martingale property, 
$$w(x,t)=\E_x[M_0] = \E_x[M_{\tau^{\delta}}].$$
Then, letting $\delta \searrow 0$, by dominated convergence (since $X$ is continuous and so $\tau^\delta\to\tau$ as $\delta\searrow0$, and
$w,g\in C(\bar \Omega)$ are bounded), we conclude that $w(x,t) =\E_x[M_{\tau}]$, which completes the proof of \eqref{eq:FKgen}.
\end{proof}

\bibliography{globalrefs} 
\bibliographystyle{alpha}

\end{document}